\documentclass[a4 paper]{amsart}
\usepackage{latexsym}
\usepackage{tikz-cd}
\usepackage{hyperref}
\usepackage
[
margin=3cm
]
{geometry}
\newcommand{\be}{\begin{equation}}
\newcommand{\ee}{\end{equation}}
\newcommand{\bea}{\begin{eqnarray}}
\newcommand{\eea}{\end{eqnarray}}
\newcommand{\bean}{\begin{eqnarray*}}
\newcommand{\eean}{\end{eqnarray*}}

\newcommand{\brray}{\begin{array}}
\newcommand{\erray}{\end{array}}
\newcommand{\ben}{\begin{equation}{nonumber}}
\newcommand{\een}{\end{equation}{nonumber}}

\newtheorem{defn}{Definition}[section]
\newtheorem{thm}[defn]{Theorem}
\newtheorem{lemma}[defn]{Lemma}
\newtheorem{prop}[defn]{Proposition}
\newtheorem{corr}[defn]{Corollary}
\newtheorem{xmpl}[defn]{Example}
\newtheorem{rmk}[defn]{Remark}
\newcommand{\bdfn}{\begin{defn}}
\newcommand{\bthm}{\begin{thm}}
\newcommand{\blmma}{\begin{lemma}}
\newcommand{\bppsn}{\begin{prop}}
\newcommand{\bcrlre}{\begin{corr}}
\newcommand{\bxmpl}{\begin{xmpl}}
\newcommand{\brmrk}{\begin{rmk}}
\newcommand{\edfn}{\end{defn}}
\newcommand{\ethm}{\end{thm}}
\newcommand{\elmma}{\end{lemma}}
\newcommand{\eppsn}{\end{prop}}
\newcommand{\ecrlre}{\end{corr}}
\newcommand{\exmpl}{\end{xmpl}}
\newcommand{\ermrk}{\end{rmk}}
\newcommand{\IC}{\mathbb{C}}

\newcommand{\ahoma}{{}^\A\mathrm{Hom}_\A}
\newcommand{\homaa}{\mathrm{Hom}^\A_\A}
\newcommand{\homc}{\mathrm{Hom}_\IC}
\newcommand{\id}{\mathrm{id}}

\newcommand{\smoothfn}{C^\infty(M)}
\newcommand{\oneformclassical}{\Omega^1 (M)}
\newcommand{\oneform}{{{\Omega}^1 (\mathcal{A}) }}

\newcommand{\twoform}{{{\Omega}^2}(\mathcal{A})}
\newcommand{\tensora}{\otimes_{\mathcal{A}}}
\newcommand{\tensorsym}{\otimes^{{\rm sym}}_{\mathcal{A}}}
\newcommand{\tensorc}{\otimes_{\mathbb{C}}}
\newcommand{\A}{\mathcal{A}}

\newcommand{\B}{\mathcal{B}}

\newcommand{\C}{\mathcal{C}}

\newcommand{\E}{\mathcal{E}}

\newcommand{\F}{\mathcal{F}}

\newcommand{\zeroE}{{}_0\E}
\newcommand{\Ezero}{\E_0}

\newcommand{\Psym}{P_{\rm sym}}
\newcommand{\Hom}{{\rm Hom}}

\newcommand{\aomega}{\A_\Omega}

\newcommand{\staromega}{\ast_\Omega}
\newcommand{\Edelta}{{}_{\mathcal{E}} \Delta}
\newcommand{\deltaE}{\Delta_{\mathcal{E}}}

\newcommand{\sigmacan}{\sigma^{{\rm can}}}

\newcommand{\aone}{a_{(1)}}
\newcommand{\atwo}{a_{(2)}}

\newcommand{\omegaone}{\omega_{(1)}}

\newcommand{\RNum}[1]{\uppercase\expandafter{\romannumeral #1\relax}}

\begin{document}

\title{Covariant connections on bicovariant differential calculus}

\maketitle
\begin{center}
	{\large {Jyotishman Bhowmick and Sugato Mukhopadhyay}}\\
	Indian Statistical Institute\\
	203, B. T. Road, Kolkata 700108\\
	Emails: jyotishmanb$@$gmail.com, m.xugato@gmail.com \\
\end{center}

\begin{abstract}
	Given a bicovariant differential calculus $ (\mathcal{E}, d) $  such that the braiding map is diagonalisable in a certain sense, the bimodule of two-tensors admits a direct sum decomposition into symmetric and anti-symmetric tensors. This is used to prove the existence of a bicovariant torsionless connection  on $ \mathcal{E}. $ Following Heckenberger and Schm{\"u}dgen, we study invariant metrics and the compatibility of covariant connections with such metrics. A sufficient condition for the existence and uniqueness of bicovariant Levi-Civita connections is derived. This condition is shown to hold for cocycle deformations  of classical Lie groups.
\end{abstract}

\section{Introduction}
The theory of bicovariant differential calculi developed by Woronowicz brought quantum groups and their homogeneous spaces into the realm of noncommutative geometry. While the spectral triple framework of Connes (\cite{connes}) and the notion of equivariance with quantum group corepresentations led to seminal papers like \cite{chakpal}, \cite{trieste}, \cite{conneslocalindex} and \cite{neshveyev}, an alternative approach based on differential calculi has also attracted a lot of attention in recent years (see \cite{beggssmith}, \cite{matassakahler}, \cite{buachallacomplex}, \cite{buachallakahler} and references therein).

The goal of this article is to study bicovariant connections on bicovariant differential calculus on quantum groups and the notion of their metric compatibility. Such questions have already been studied by Heckenberger and Schm{\"u}dgen in \cite{heckenberger} in the context of Levi-Civita connections and then in \cite{heckenbergerlaplace} and \cite{heckenbergerspin}. On the other hand, Beggs, Majid and their collaborators studied Levi-Civita connections on quantum groups and homogeneous spaces and we refer to the book \cite{beggsmajidbook} for a comprehensive account. This article aims to build a general theory by working with an arbitrary bicovariant differential calculus and bi-invariant pseudo-Riemannian metrics.

In order to explain our main results, we will need some notations. In the sequel, $ \A $ will stand for a Hopf algebra and $ \E $ the $\A$-bimodule of one-forms coming from a bicovariant differential calculus. The symbol $ \zeroE $ will stand for the set of all left-invariant elements of the module $\E.$ Moreover, $\twoform$ will denote the bimodule of $2$-forms of the differential calculus as in \cite{woronowicz}.
 Woronowicz (\cite{woronowicz}) proved the existence of a canonical braiding map $ \sigma $ which is a bicovariant $\A$-bilinear map from $ \E \tensora \E $ to itself. Consequently, $ \sigma $ restricts to a map (to be denoted by $ {}_0 \sigma $) from $ \zeroE \tensorc \zeroE $ to itself. Throughout most of the article, we will work with the assumption that the map $ {}_0 \sigma $ is diagonalisable. This assumption is satisfied by a fairly large class of bicovariant differential calculi. (see Proposition \ref{14thoct20191}) and we have the following result: 

\begin{thm} (Theorem \ref{25thmay20195}, Remark \ref{24thaugustremark}) \label{3rddec20192}
Suppose $ (\E, d) $ is a bicovariant differential calculus such that the map $ {}_0 \sigma $ is diagonalisable. The symbol $\wedge$ will denote the wedge map $\wedge: \E \tensora \E \rightarrow \twoform.$

 There exists an $\A$-$\A$-bimodule $\F$ such that
$$\E \tensora \E = {\rm Ker} (\wedge) \oplus \F.$$ 
 Moreover, the map $\wedge|_\F: \F \to \twoform$ defines an isomorphism of right $\A$-modules.
\end{thm}
As a first corollary to Theorem \ref{3rddec20192}, we can define a bicovariant $\A$-bilinear idempotent $ \Psym \in \Hom_\A (\E \tensora \E, \E \tensora \E) $ having range $ {\rm Ker} (\wedge) $ and kernel $\F.$ However, unlike the classical case (where $ \Psym = \frac{1 }{2}(1 + \sigma) $), $ \Psym $ is given by a different formula (cf. \eqref{24thaugust20192}). Our next application of Theorem \ref{3rddec20192} is the following: 
\begin{thm} (Theorem \ref{20thaugust20192})
Under the assumptions of Theorem \ref{3rddec20192}, there exists a bicovariant torsionless connection on $\E.$
\end{thm}
We define a notion of compatibility of a left-covariant connection with a left-invariant pseudo-Riemannian metric in Section \ref{section4} and this leads us to the definition of Levi-Civita connection. Our main result concerning the existence and uniqueness of Levi-Civita connection is in terms of the map $ \Psym $ mentioned above. Due to the left-covariance of the map $ \Psym, $ it follows that it restricts to a map 
$$ {}_0 \Psym: \zeroE \tensorc \zeroE \rightarrow \zeroE \tensorc \zeroE. $$
Then we have the following result:
\begin{thm} (Theorem \ref{15thjune20193}, Theorem \ref{18thseptember20193})
Suppose $ (\E, d) $ is a bicovariant differential calculus over a cosemisimple Hopf algebra such that the map $ {}_0 \sigma $ is diagonalisable and $g$ be a bi-invariant pseudo-Riemannian metric on $\E.$ If we assume that the map
\begin{equation} \label{3rddec20193} ({}_0 \Psym)_{23}: {}_0 ({\rm Ker} (\wedge)) \tensora \zeroE \rightarrow \zeroE \tensorc {}_0 ({\rm Ker} (\wedge)) \end{equation}
is an isomorphism, then there exists a unique bicovariant Levi-Civita connection for the triple $ (\E, d, g). $

If in addition, $\Omega$ is a normalised dual $2$-cocycle on $\A$ and $g^\prime$ is a bi-invariant pseudo-Riemannian metric on the bicovariant differential calculus $ (\E_\Omega, d_\Omega), $ then there exists a bi-invariant pseudo-Riemannian metric $g$ on $ (\E, d) $ such that $ g^\prime = g_\Omega $ and moreover, the Levi-Civita connection for $ (\E, d, g) $ deforms to the unique bicovariant Levi-Civita connection for the triple $ (\E_\Omega, d_\Omega, g_\Omega). $ 
\end{thm}
Moreover, in \cite{suq2}, it is proven by an explicit computation that the isomorphism \eqref{3rddec20193} also holds for the $4 D_{\pm}$ calculi on $ SU_q (2).$ 

The proofs follow the strategy adopted in \cite{article1} and \cite{article2} where the set up was that of a class of centered bimodules $\E$ equipped with a map $\sigma: \E \tensora \E \rightarrow \E \tensora \E $ such that $ \sigma (\omega \tensora \eta) = \eta \tensora \omega $ for all $\omega, \eta$ in the center of $\E.$ This led to the existence-uniqueness of Levi-Civita connections on fuzzy-spheres, quantum Heisenberg manifold, a class of Rieffel-deformations (\cite{article1}, \cite{article3}) and Cuntz algebras (\cite{soumalya}). Our bimodules are not centered and our braiding operator $\sigma$ does not satisfy the above equation. Moreover, the Koszul-formula proof employed in \cite{article3} requires the pseudo-Riemannian metric to be $\A$-bilinear while our pseudo-Riemannian metric $g$ is assumed to be only right $\A$-linear. However, the bicovariance of the differential calculus as well as the braiding map $\sigma$ and pseudo-Riemannian metrics come to our rescue. The role played by the center of the module in \cite{article1} is now (in some sense) played by the finite dimensional vector space $\zeroE$ of invariant elements in $\E.$ Indeed, by virtue of the characterisation of left-covariant maps (Proposition \ref{14thfeb20192}), the maps $ \sigma, g $ and a left-covariant right connection $\nabla$ are determined by their restrictions $ {}_0 \sigma: \zeroE \tensorc \zeroE \rightarrow \zeroE \tensorc \zeroE, g : \zeroE \tensorc \zeroE \rightarrow \IC, \nabla: \zeroE \rightarrow \zeroE \tensorc \zeroE$ since $ \zeroE $ (respectively $ \zeroE \tensorc \zeroE $) is right $\A$-total in $\E$ (respectively $ \E \tensora \E $). This helps us to convert various module maps to maps between finite dimensional vector spaces. Let us remark that we have consciously avoided using the terminology of Yetter-Drinfeld modules  (\cite{klimyk}) since many definitions and results of this article apply for left-covariant bimodules over a Hopf algebra too (see the results of Subsection \ref{21staugust20191jb}, Definition \ref{metriccompatibility} and some results of Subsection \ref{4thdec20191}).

Let us now clarify  the differences of our approach with some other papers in the literature. Beggs, Majid and their collaborators  studied bimodule connections on many class of examples. We refer to \cite{beggsmajidbook} for a comprehensive account. We work with usual right-connections on bicovariant bimodules as opposed to bimodule connections and moreover, the choice of braiding for us is always the canonical braiding map (Proposition \ref{4thmay20193}) for a bicovariant bimodule. A dual approach of working in the set up of braided derivations has been pursued by Weber in  the paper \cite{weber}. Indeed, a special case considered by Weber is the (one-sided) covariant bimodule of braided derivations on a triangular Hopf algebra. Weber develops the theory of braided covariant derivative on this bimodule and then proves (  Lemma 3.12 and Corollary 3.13 of \cite{weber}) the existence of a unique Levi-Civita braided covariant derivative for any braided metric. The proof follows by a Koszul formula argument. 

Our approach is closer in spirit to that of Heckenberger and Schm\"udgen's paper \cite{heckenberger} who work on the level of one-forms coming from bicovariant differential calculi on the quantum groups $ SL_q (n), $ $ O_q (n) $ and $ Sp_q (n). $ The authors of \cite{heckenberger} show that for a fixed choice of a bi-invariant pseudo-Riemannian metric on the space of one forms, there exist unique bicovariant Levi-Civita connections for these three classes of quantum groups. Heckenberger and Schm\"udgen work on left-connections while we work with right connections and the definition of the torsion is the same. However, there are two main differences, namely, the metric compatibility of a connection and the definition of two-forms. In Proposition \ref{25thnov20191}, we prove that our definition of metric-compatibility matches with that of \cite{heckenberger} for cocycle deformations of classical groups. But in general, these two definitions are different. The issue regarding two-forms is explained in Remark \ref{28thnov20191}. Due to these two differences, we could not recover the results of \cite{heckenberger} from our approach. 

The plan of the article is as follows: in Section \ref{3rddec20191}, we recall the definition and basic properties of bicovariant differential calculus and prove a characterization of left covariant maps between left-covariant bimodules. In Section \ref{24thmay20192}, we obtain a splitting of the bimodule $\E \tensora \E $ which allows us to define the symmetrization map $ \Psym. $ In Section \ref{section2}, we recall the notion of invariant pseudo-Riemannian metric on a bicovariant bimodule from \cite{article5} and prove some additional results which will be needed later. In Section \ref{section5}, we prove the existence of a torsionless bicovariant connection on any bicovariant bimodule provided the map $ {}_0 \sigma $ is diagonalisable. Section \ref{section4} is devoted to studying metric-compatibility of a bicovariant connection with a bi-invariant metric and the bicovariance of some associated maps. In Section \ref{21staugust20197}, we prove a sufficient condition of existence and uniqueness of a unique bicovariant Levi-Civita connection. Finally, in Section \ref{sectioncocyclelc}, we deal with the question of existence and uniqueness of Levi-Civita connections on cocycle deformed bicovariant differential calculus. 

All vector spaces will be assumed to be over the complex field. For vector spaces $ V_1 $ and $ V_2, $ $ \sigma^{{\rm can}} : V_1 \tensorc V_2 \rightarrow V_2 \tensorc V_1 $ will denote the canonical flip map, i.e, $ \sigma^{{\rm can}} (v_1 \tensorc v_2) = v_2 \tensorc v_1. $ A subset $X$ of a right-module $ \mathcal{F}_1 $ over an algebra $B$ is said to be right $B$-total in $ \mathcal{F}_1 $ if $ {\rm Span} \{ f . b: f \in X, b \in B \} $ is equal to $ \mathcal{F}_1.$ If $ \mathcal{F}_2 $ is another right $B$-module, the set of all right $B$-linear maps from $ \mathcal{F}_1 $ to $ \mathcal{F}_2 $ will be denoted by $ \Hom_B (\mathcal{F}_1, \mathcal{F}_2). $ For the rest of the article, $(\A,\Delta)$ will denote a Hopf algebra. We will use the Sweedler notation for the coproduct $\Delta$. Thus, we will write
\begin{equation} \label{28thaugust20191} 
\Delta(a) = a_{(1)} \tensorc a_{(2)}.
\end{equation}
For a left comodule coaction $\Delta_V$ of $\A$ on a vector space $V$ we will write
\begin{equation} \label{28thaugust20192}
\Delta_V(v) = v_{(-1)} \tensorc v_{(0)}.
\end{equation}
Similarly, for a right comodule coaction ${}_V \Delta$, we will use the notation
\begin{equation} \label{28thaugust20193}
{}_V \Delta(v) = v_{(0)} \tensorc v_{(1)}.
\end{equation}

\section{Bicovariant differential calculus} \label{3rddec20191}

In this section we recall and prove some basic facts on bicovariant differential calculus on a Hopf algebra. In the first two subsections, we state the basic properties of bicovariant differential calculus and bicovariant bimodules. In Subsection \ref{21staugust20191jb}, we adapt the arguments of Heckenberger and Schm{\"u}dgen \cite{heckenberger} to prove a characterisation of left-covariant maps on covariant bimodules. Finally, in Subsection \ref{29thnov20191}, we recall the definition of two-forms of a bicovariant differential calculus.\\

\subsection{Covariant Differential Calculi} \label{21staugust20192jb}

Let us start with the definition of a first order differential calculus over any algebra $\B.$

\begin{defn}{(Definition 1.1 of \cite{woronowicz})}
	Let $\B$ be an algebra with unity, $\E$ be a bimodule over $\B$ and $$d: \B \to \E$$ be a $\IC$-linear map. We say that $(\E, d)$ is a first order differential calculus over $\B$ if
	\begin{itemize}
		\item[(i)] For any $a$, $b$ in $\B$, $$d(ab)=(da)b+adb,$$
		\item[(ii)] Any element $\rho$ in $\E$ is of the form $$\rho = \sum_{k=1}^{K}a_kdb_k,$$ for some $a_k$, $b_k$ in $\A$.
	\end{itemize}
\end{defn}
Now we come to the definition of left (or right) covariant first order differential calculus over a Hopf algebra $\A.$
\begin{defn}{(Definitions 1.2, 1.3 of \cite{woronowicz})} \label{3rdseptember20191}
	Let $(\E,d)$ be a first order differential calculus on a Hopf algebra $\A$.\\
	We say that $(\E,d)$ is left-covariant if for any $a_k,b_k$ in $\A$, $k=1, \dots, K$,
	$$(\sum_k a_kdb_k = 0) \text{ implies that }(\sum_k \Delta(a_k)(\id \tensorc d)\Delta(b_k)=0).$$
	We say that $(\E,d)$ is right-covariant if for any $a_k,b_k$ in $\A$, $k=1, \dots, K$,
	$$(\sum_k a_kdb_k = 0) \text{ implies that } (\sum_k \Delta(a_k)(d \tensorc \id)\Delta(b_k)=0).$$
	We say $(\E,d)$ is bicovariant if it is both left-covariant and right-covariant.
\end{defn}

Woronowicz (\cite{woronowicz}) proved that a bicovariant differential calculus is automatically endowed with a left as well as a right comodule coaction. This is the content of the next proposition. 

\begin{prop}{(Propositions 1.2, 1.3 and 1.4 of \cite{woronowicz})} \label{cdc}
	Let $(\E, d)$ be a bicovariant first order differential calculus on $\A$. Then there exists linear mappings 
	$$\Delta_{\E}:\E \to \A \tensorc \E, \quad {}_\E \Delta: \E \to \E \tensora \A $$ such that
	\begin{itemize}
		\item[(i)] $(\E, \Delta_{\E}, {}_\E \Delta)$ is a bicovariant $\A$-bimodule, i.e. $(\E ,\Delta_\E)$ is a left $\A$-comodule, $(\E, {}_\E \Delta)$ is a right $\A$-comodule and the following equations hold for all $e$ in $\E$ and $a$ in $\A$:
		\begin{eqnarray} 
			&	\Delta_\E(ae) = \Delta(a) \Delta_\E (e), \ \ \Delta_\E(ea) = \Delta_\E(e) \Delta(a) \label{8thnov20192} \\
			& {}_\E \Delta(ae) = \Delta(a) {}_\E \Delta(e), \ \ {}_\E \Delta(ea) = {}_\E \Delta(e) \Delta(a) \label{8thnov20193} \\
		& (\id \tensorc {}_\E \Delta) \Delta_\E = (\Delta_\E \tensorc \id) {}_\E \Delta \label{8thnov20194}
		\end{eqnarray}
		\item[(ii)] d is bicovariant, i.e. \begin{equation} \label{23rdaugust20191}
		\Delta_\E \circ d = (\id \tensorc d) \Delta \quad {}_\E \Delta \circ d = (d \tensorc \id) \Delta.
		\end{equation} 
	\end{itemize}
\end{prop}

We note the following consequence of Proposition \ref{cdc} which we will need in Theorem \ref{20thaugust20192} and Proposition \ref{11thmay20192}:

\begin{lemma} \label{6thjune20192}
	For any $ a \in \A, $ the following equations holds:
	\begin{enumerate}
		\item[(i)] $ a_{(1)} \tensorc d(a_{(2)}) = (da)_{(-1)} \tensorc (da)_{(0)} $\\
		\item[(ii)] $ d(a_{(1)}) \tensorc a_{(2)} = (da)_{(0)} \tensorc (da)_{(1)} $\\
		\item[(iii)] $ a_{(1)} \tensorc d(a_{(2)}) \tensorc a_{(3)} = (da)_{(-1)} \tensorc (da)_{(0)} \tensorc (da)_{(1)} $
	\end{enumerate} 
\end{lemma}
\begin{proof}
	Part (i) and part (ii) follow from \eqref{23rdaugust20191}. Finally, for Part (iii), we have
	\begin{eqnarray*}
		&& a_{(1)} \tensorc d(a_{(2)}) \tensorc a_{(3)}
		= (\id \tensorc d \tensorc \id) (\id \tensorc \Delta) \Delta (a) \\
		&=& (\id \tensorc \Edelta) (\id \tensorc d) \Delta (a) {\rm \ (by \eqref{23rdaugust20191})}
		= (\id \tensorc \Edelta) \deltaE (da) {\rm \ (by \eqref{23rdaugust20191})}\\
		&=& (\id \tensorc \Edelta) ((da)_{(-1)} \tensorc (da)_{(0)})
		= (da)_{(-1)} \tensorc \Edelta ((da)_{(0)})\\
		&=& (da)_{(-1)} \tensorc (da)_{(0)} \tensorc (da)_{(1)}.
	\end{eqnarray*}
	This proves the lemma.
\end{proof}

\subsection{Preliminaries on covariant bimodules}

Suppose $ \A $ is a Hopf algebra. Let us recall (\cite{woronowicz}) that a left $\A$-comodule $ (M, \Delta_M) $ is called a left-covariant $\A$-bimodule if $ M $ is an $\A$-bimodule and \eqref{8thnov20192} is satisfied. A right $\A$-comodule $ (M, {}_M \Delta) $ is called a right-covariant $\A$-bimodule if $ M $ is an $\A$-bimodule and \eqref{8thnov20193} is satisfied. A triplet $ (M, \Delta_M, {}_M \Delta) $ is called a bicovariant $\A$-bimodule if the conditions (i) in Proposition \ref{cdc} are satisfied. 
		 
By virtue of Proposition \ref{cdc}, if $(M, d)$ is bicovariant differential calculus over $\A$, then $ (M, \Delta_M, {}_M \Delta) $ is automatically a bicovariant bimodule. Thus, whenever we dealing with a bicovariant differential calculus, we can use all the results on bicovariant bimodules in \cite{woronowicz} and \cite{article5}. Let us record some results on covariant bimodules from \cite{woronowicz} and elsewhere which we will use repeatedly throughout the article. 
We start by recalling the definition of covariant maps. 
 
\begin{defn} \label{8thnov20191}
	Let $(M, \Delta_{ M})$ and $(N, \Delta_N)$ be left-covariant $\A$-bimodules and $ T $ be a $ \mathbb{C} $-linear map from $ M $ to $ N. $
	
	$ T $ is called left-covariant if for all $ m \in M, n \in N, a \in A, $
$$ (\id \tensorc T)(\Delta_{M}(m))=\Delta_N(T(m)). $$
$T$ is called right-covariant if for all $ m \in M, n \in N, a \in A, $
 $$ (T \tensorc \id) {}_M \Delta (m) = {}_N \Delta (T (m)). $$
Finally, a map which is both left and right covariant will be called a bicovariant map. The set of all right $\A$-linear left covariant maps from $ M $ to $ N $ will be denoted by the symbol $ {}^{\A} \Hom_\A (M, N). $ 
	\end{defn}
	
	Now we introduce the left (respectively, right) invariant elements of a left (respectively, right)-covariant bimodule. 

\begin{defn} \label{21staugust20193}
	If $ (M, \Delta_M) $ is a left-covariant bimodule, the subspace of left-invariant elements of $M$ is defined to be the vector space 
	$${}_0M:=\{m \in M : \Delta_M(m)=1\tensorc m\}.$$
	Similarly, if $(M, {}_M\Delta)$ is a right-covariant bimodule over $\A$ , the subspace of right-invariant elements of $M$ is the vector space 
	$$M_0:=\{m \in M : {}_M\Delta(m) = m \tensorc 1\}.$$
\end{defn}

Woronowicz (\cite{woronowicz}) proved that if $ M $ is a left-covariant bimodule over $\A,$ then $ M $ is free as a left (as well as a right) $ \A$-module. In fact, one has the following result:

\begin{prop}{(Theorem 2.1 and Theorem 2.3 of \cite{woronowicz})} \label{moduleiso}
	Let $ (M, \Delta_M) $ be a left-covariant $\A$-bimodule. Then the multiplication maps $\widetilde{u}^M:{}_0M \tensorc \A \to M$ and $ \widetilde{v}^M: \A \tensorc {}_0 M \rightarrow M $ are isomorphisms. Similarly, if $ (M, {}_M \Delta) $ is a right-covariant bimodule, then the multiplication maps $ M_0 \tensorc \A \to M$ and $ \A \tensorc M_0 \to M$ are isomorphisms.
\end{prop}
All bicovariant differential calculi $ (\E, d) $ under consideration in this article will be such that $ \E_0 $ and $ \zeroE $ are finite-dimensional vector spaces.

If $ (M, \Delta_M) $ and $ (N, \Delta_N) $ are left-covariant bimodules over $\A,$ then we have a left coaction $ \Delta_{M \tensora N} $ of $ \A $ on $ M \tensora N $ defined by the following formula:
$$ \Delta_{M \tensora N} (m \tensora n) = (m_\A \tensorc \id_M \tensorc \id_N) (\sigma^{{\rm can}})_{23} (\Delta_M (m) \tensorc \Delta_N (n)). $$
Here $m_\A: \A \tensorc \A \to \IC$ denotes the multiplication map. This makes $ M \tensora N $ into a left-covariant $\A$-bimodule. Similarly, there is a right coaction $ {}_{M \tensora N} \Delta $ on $ M \tensora N$ if $ (M, {}_M \Delta) $ and $ (N, {}_N \Delta) $ are right-covariant. If $ M $ and $ N $ are bicovariant bimodules, then it can be easily checked that $(M \tensora N, \Delta_{M \tensora N}, {}_{M \tensora N} \Delta)$ is again a bicovariant bimodule over $\A.$ 
By adapting the proof of Lemma 3.2 of \cite{woronowicz}, we have the following result:
\begin{corr} \label{3rdaugust20191}
		Let $(M,\Delta_M)$ and $(N, \Delta_N)$ be left-covariant bimodules over $\A$ and $\{m_i\}_i$ and $\{n_j\}_j$ be vector space bases of ${}_0 M$ and ${}_0 N$ respectively. Then each element of $M \tensora N$ can be written as $\sum_{ij} a_{ij} m_i \tensora n_j$ and $\sum_{ij} m_i \tensora n_j b_{ij}$, where $a_{ij}$ and $b_{ij}$ are uniquely determined.\\
		A similar result holds for right-covariant bimodules $(M, {}_M \Delta)$ and $(N, {}_N \Delta)$ over $\A$. Finally, if $(M,\Delta_M)$ is a left-covariant bimodule over $\A$ with basis $\{m_i\}_i$ of ${}_0 M$, and $(N, {}_N \Delta)$ is a right-covariant bimodule over $\A$ with basis $\{n_i\}_i$ of $N_0$, then any element of $M \tensora N$ can be written uniquely as $\sum_{ij} a_{ij} m_i \tensora n_j$ as well as $\sum_{ij} m_i \tensora n_j b_{ij}$.
\end{corr}
Our next proposition states that if $ M $ and $ N $ are left-covariant bimodules, then the left-invariant elements of $ M \tensora N $ is the tensor product of the vector spaces $ {}_0 M $ and $ {}_0 N. $
\begin{prop} \label{3rdaugust20192} (Theorem 5.7 of \cite{schauenberg}, Proposition 2.6 of \cite{article5})
	Let $(M,\Delta_M)$ and $(N, \Delta_N)$ be left-covariant bimodules. Following Definition \ref{21staugust20193}, we denote the left-invariant elements (with respect to the coaction $\Delta_{M \tensora N}$) of $M \tensora N$ by ${}_0(M \tensora N)$. Similarly, the right-invariant elements of $M \tensora N$ (with respect to the coaction ${}_{M \tensora N} \Delta$) will be denoted by $(M \tensora N)_0$. Then we have that
	\begin{equation} \label{21staugust20194}
	{}_0 (M \tensora N) = {\rm span}_\IC \{m \tensora n : m \in {}_0 M, n \in {}_0 N \}.
	\end{equation}
	Similarly, if $(M, {}_M \Delta)$ and $(N, {}_N \Delta)$ are right-covariant bimodules over $\A$, then we have that $$ (M \tensora N)_0 = {\rm span}_\IC \{m \tensora n : m \in M_0, n \in N_0 \}.$$
	Thus, ${}_0 (M \tensora N) = {}_0 M \tensorc {}_0 N$ and $(M \tensora N)_0 = M_0 \tensorc N_0$. Therefore, we are allowed to use the notations $ {}_{0} M \tensorc {}_{0} N $ and $ {}_{0} (M \tensora N) $ interchangeably.
\end{prop}

Finally we recall the canonical braiding map associated with a bicovariant bimodule.

\begin{prop}{(Proposition 3.1 of \cite{woronowicz})} \label{4thmay20193}
	Given a bicovariant bimodule $(M, \Delta_M, {}_M \Delta)$, there exists a unique bimodule homomorphism
	$$\sigma: M \tensora M \to M \tensora M ~ {\rm such} ~ {\rm that} $$ 
	\be \label{30thapril20191} \sigma(\omega \tensora \eta)= \eta \tensora \omega \ee
	for any left-invariant element $\omega$ and right-invariant element $\eta$ in $M$. $\sigma$ is invertible and the following equations hold, making $\sigma$ a bicovariant $\A$-bimodule map from $M \tensora M$ to itself:
	 \begin{equation} \label{21staugust20194jb}
	(\id_\A \tensora \sigma) \Delta_{M \tensora M} = \Delta_{M \tensora M} \circ \sigma, ~ (\sigma \tensora \id_\A) {}_{M \tensora M} \Delta = {}_{M \tensora M} \Delta \circ \sigma.
	\end{equation} 
	Moreover, $\sigma$ satisfies the following braid equation on $M \tensora M \tensora M:$
	$$ (\id \tensora \sigma)(\sigma \tensora \id)(\id \tensora \sigma)= (\sigma \tensora \id)(\id \tensora \sigma)(\sigma \tensora \id). $$
\end{prop}

\subsection{A characterisation of left-covariant maps and some consequences} \label{21staugust20191jb}

Our aim is to study properties of bicovariant connections on a bicovariant differential calculus, namely, their torsion and compatibility with bi-invariant metrics. We refer to the later sections for the definitions of bicovariant differential calculus and pseudo-Riemannian metrics. Our strategy is to exploit the left-covariance of the various maps (the connection, the metric, the de-Rham differential and the map $ \sigma $) between the underlying bicovariant bimodules of a bicovariant differential calculus to view them as maps between the finite-dimensional vector spaces of left-invariant elements of the respective bimodules. This was already observed and used crucially by Heckenberger and Schm{\"u}dgen in the paper \cite{heckenberger}. The goal of this subsection is to give a systematic treatment to this idea and our main goals are to prove Proposition \ref{14thfeb20192} and its corollaries. The proofs in this subsection are elementary but we provide all the details since we will need to refer to these results repeatedly. 

For the rest of this subsection, we will use the notations introduced in Proposition \ref{moduleiso} freely.

\begin{prop} \label{ahoma}
	Let $(M, \Delta_M)$ and $(N, \Delta_N)$ be left-covariant bimodules over $\A$ and $T$ be a left-covariant right $\A$-linear map from $M$ to $N$. Then $T({}_0 M) \subseteq {}_0 N$. Moreover, there exists a unique $\IC$-linear map ${}_0T$ in $\homc({}_0M,{}_0N)$ such that
	\begin{equation} \label{ahomadiag}
	 (\widetilde{u}^N)^{-1} \circ T = ({}_0 T \tensorc \id) (\widetilde{u}^M)^{-1}.
		\end{equation}
		In particular, a left covariant right $\A$-linear map $T$ from $M$ to $N$ is determined by its action on $ {}_0 M.$
\end{prop}
\begin{proof} 
	Let $\{m_i\}_i$ be a vector space basis for ${}_0M$ and $\{n_j\}_j$ be a vector space basis for ${}_0 N$. Since $T$ is a left-covariant right $\A$-linear map from $M$ to $N$, we have that
	$$ \Delta_N(T(m_i)) = (\id \tensorc T)\Delta_M(m_i) =(\id \tensorc T)(1 \tensorc m_i) =1 \tensorc (T(m_i)).$$
	Therefore, $T(m_i)$ is in ${}_0N$. This proves the first assertion.\\
	Define ${}_0T$ to be the restriction of $T$ on ${}_0 M$. Let $ m = \widetilde{u}^M (\sum_i m_i \tensorc a_i), $ where $ \widetilde{u}^M $ is as defined in Proposition \ref{moduleiso}. Then 
	$$ ({}_0 T \tensorc \id) {(\widetilde{u}^M)}^{-1} (m) = \sum_i {}_0 T (m_i) \tensorc a_i = {(\widetilde{u}^N)}^{-1} \circ T (\sum_i m_i a_i) = {(\widetilde{u}^N)}^{-1} \circ T (m) $$
	and thus \eqref{ahomadiag} follows. The uniqueness follows from the fact that the equation \eqref{ahomadiag} implies that ${}_0T(m_i) = T(m_i)$ for all $i.$		 
\end{proof}	

\begin{corr} \label{4thmay20192}
	Let $(M, \Delta_M)$ be a left-covariant bimodule over $\A$ and $T$ be a left-covariant right $\A$-linear map from $M$ to $\A$. Then there exists a unique $\IC$-linear map ${}_0T$ in $\homc({}_0M,\IC)$ such that
	$$ T = ({}_0 T \tensorc \id) (\widetilde{u}^M)^{-1}. $$
	\end{corr}
\begin{proof}
	The proof follows by taking $(N,\Delta_N)=(\A, \Delta)$ in Proposition \ref{ahoma}.
\end{proof}	

\begin{prop} \label{inviff}
	Let $(M, \Delta_M)$ and $(N, \Delta_N)$ be left-covariant bimodules over $\A$. Then $\ahoma(M,N)$ is isomorphic to $\rm Hom_\IC({}_0M,{}_0N)$ as complex vector spaces. Moreover a left-covariant right $\A$-linear map $ T $ from $M$ to $N$ is invertible if and only if ${}_0T$ is invertible. More generally, $\lambda$ is an eigenvalue of $T$ if and only if $\lambda$ is an eigenvalue of ${}_0T$.
\end{prop}
\begin{proof}
Let us recall (Definition \ref{8thnov20191}) that $ \ahoma(M,N) $ denotes the set of all right $ \A $-linear left-covariant maps from $ M $ to $ N. $ We define a map 
	$$ \ahoma(M,N) \rightarrow \homc({}_0M, {}_0N); ~ T \mapsto {}_0T$$
	as in Proposition \ref{ahoma}. As $ T $ is left-covariant, by Proposition \ref{ahoma}, $ T ({}_0 M) \subseteq {}_0 N. $ Since $T$ is determined by its action on ${}_0(M)$, this map is one-one. Given an element $ {}_0 T $ in $ \homc ({}_0 M, {}_0 N), $ the map $ \widetilde{u}^N ({}_0 T \tensorc \id_\A) (\widetilde{u}^M)^{-1} $ defines an element, say $ T, $ in $ \Hom_\A (M, N) $ which can be easily checked to be left-covariant and whose image under the above map is $ {}_0 T. $ Thus, the map is a bijection.
	
	The equation \eqref{ahomadiag} implies that $T$ is invertible if and only if ${}_0T$ is invertible. Finally, $\lambda$ is an eigenvalue of ${}_0T$ if and only if ${}_0(T-\lambda . 1) = {}_0 T - \lambda . 1$ is not invertible and ${}_0(T-\lambda . 1)$ is not invertible if and only if $T- \lambda . 1$ is not invertible by the above argument. Hence, $\lambda$ is an eigenvalue of $T$ if and only if it is an eigenvalue of ${}_0T$.
\end{proof}	

\begin{prop} \label{14thfeb20192}
	Let $(M, \Delta_{ M})$ and $(N, \Delta_N)$ be left-covariant $\A$-$\A$ bimodules. Then a right $\A$-linear map $T: M \to N$ is left-covariant if and only if $T({}_0M)\subseteq {}_0N$. 
	
	In particular, if $ S: M \tensora N \rightarrow M \tensora N $ is a right $\A$-linear map, then Proposition \ref{3rdaugust20192} implies that $ S $ is left-covariant if and only if $ S ({}_0 M \tensorc {}_0 N) \subseteq {}_0 M \tensorc {}_0 N. $
\end{prop}
\begin{proof}
	 If the map $T$ is left-covariant, then by Proposition \ref{ahoma}, $T({}_0 M) \subseteq {}_0 N$. Conversely, suppose $T$ is a right $\A$-linear map and $T({}_0 M) \subseteq {}_0 N$. Let $\{m_i\}_i$ be a vector space basis of ${}_0 M$ and $\sum_i m_i a_i$ be an element of $M$. Then we have that
	\begin{equation*}
		\begin{aligned}
		&\Delta_N(T(\sum_i m_i a_i)) = \sum_i \Delta_N(T(m_i)a_i)
		= \sum_i \Delta_N(T(m_i))\Delta(a_i)\\
		=& \sum_i (1 \tensorc T(m_i))({a_i}_{(1)} \tensorc {a_i}_{(2)})
		= \sum_i ({a_i}_{(1)} \tensorc T(m_i){a_i}_{(2)})\\ =& (\id \tensorc T)(\sum_i {a_i}_{(1)} \tensorc m_i {a_i}_{(2)}) 
		= (\id \tensorc T)(\Delta_M(\sum_i m_i a_i)).
		\end{aligned}
	\end{equation*}
	Hence $T$ is a left-covariant map.
\end{proof}	

\begin{rmk}
	Analogues of Proposition \ref{ahoma}, Corollary \ref{4thmay20192}, Proposition \ref{inviff} and Proposition \ref{14thfeb20192} also hold for right-covariant right $\A$-linear maps from $(M, {}_M\Delta)$ to $(N, {}_N\Delta).$ 
\end{rmk}

We end this subsection by proving two results related to bicovariant right $\A$-linear maps.
\begin{prop} \label{18thsep20196}
Let $ (M, \Delta_M, {}_M \Delta) $ and $ (N, \Delta_N, {}_N \Delta) $ be bicovariant $\A$-bimodules and $T$ be a left-covariant right $\A$-linear map from $M$ to $N.$ If the map $ {}_0 T = T|_{{}_0 M}: {}_0 M \rightarrow {}_0 N $ as in Proposition \ref{ahoma} is right-covariant, i.e, $ {}_N \Delta {}_0 T = ({}_0 T \tensorc \id) {}_M \Delta, $ then the map $ T $ is also right-covariant. 
\end{prop}
\begin{proof}
Let $ m $ be an element of $ {}_0 M $ and $ a $ an element of $\A.$ Then by right $\A$-linearity of $T$ and right-covariance of $ {}_0 T, $ we get
\begin{eqnarray*}
 && {}_N \Delta (T (m a)) = {}_N \Delta (T (m) a) = {}_N \Delta (T (m)) \Delta (a)\\
 &=& {}_N \Delta ({}_0 T (m)) \Delta (a) = ({}_0 T \tensorc \id) ({}_M \Delta (m)) \Delta (a)\\ &=& (({}_0 T \tensorc \id) (m_{(0)} \tensorc m_{(1)})) (\aone \tensorc \atwo)\\
&=& ({}_0 T) (m_{(0)}) \aone \tensorc m_{(1)} \atwo = T (m_{(0)} \aone) \tensorc m_{(1)} \atwo \\
&=& (T \tensorc \id) ((m_{(0)} \tensorc m_{(1)}) (\aone \tensorc \atwo)) = (T \tensorc \id) {}_M \Delta (m a).
\end{eqnarray*}
Since $ {}_0 M $ is right $\A$-total in $M,$ this proves that $ T $ is a right covariant map. 
\end{proof} 
Before stating the next result, let us note that if $T$ is a bicovariant right $\A$-linear map from $M$ to $N$ such that $\{m_i\}_i$ and $\{n_j\}_j$ are vector space bases for ${}_0 M$ and ${}_0 N$ respectively, then by Theorem 2.4 of \cite{woronowicz}, we get 
$${}_{{}_0 M} \Delta(m_i) = \sum_k m_k \tensorc a_{ki} {\rm \ and \ } {}_{{}_0N} \Delta(n_j) = \sum_l n_l \tensorc b_{lj} ,$$
 for some elements $\{a_{ki}\}_{ki}$ and $\{b_{lj}\}_{lj}$ in $\A$.
\begin{lemma} \label{21stjune20191}
	If an element $T$ of $\homc({}_0 M, {}_0N)$ is such that for all $m$, $T(m_i)= \sum_{j} n_j T^i_j$ for some elements $T^i_j$ in $\IC$, then $T$ is a right-covariant map from ${}_0 M $ to ${}_0 N $ if and only if
	\begin{equation} \label{27thjune20193}
	\sum_{il} n_l \tensorc b_{lj} T^i_j = \sum_{jk} n_j \tensorc T^k_j a_{ki}.
	\end{equation}
\end{lemma}
\begin{proof}
	If $ T$ is a right-covariant complex linear map from $ {}_0 M $ to $ {}_0 N, $ then ${}_{{}_0N}\Delta \circ T= (T \tensorc \id) {}_{{}_0M}\Delta$. Now:
	\begin{equation} \label{27thjune20191}
	{}_{{}_0N}\Delta (T(m_i)) = {}_{{}_0N}\Delta(\sum_{j} n_j T^i_j)
	= \sum_l n_l \tensorc \sum_{i} b_{lj} T^i_j.
	\end{equation}
	On the other hand,
	\begin{equation} \label{27thjune20192}
	 (T \tensorc \id){}_{{}_0M}\Delta(m_i) = (T \tensorc \id)(\sum_k m_k \tensorc a_{ki}) = \sum_j n_j \tensorc \sum_{k} T^k_j a_{ki} 
	\end{equation}
	Comparing equations \eqref{27thjune20191} and \eqref{27thjune20192}, we get that $T$ is an element of $\homc^\A({}_0M, {}_0N)$ if and only if \eqref{27thjune20193} holds.
\end{proof}

\subsection{The space of two-forms} \label{29thnov20191}

Following Woronowicz (\cite{woronowicz}), let us define the space of two forms associated to a bicovariant different calculus.
\begin{defn} \label{25thmay20194}
Let $ (\E, d) $ be a bicovariant first order differential calculus and $ \sigma $ be the map as in Proposition \ref{4thmay20193}. We define 
$$\twoform := (\E \tensora \E) \big/ \rm {\rm Ker} (\sigma - 1).$$
 The symbol $\wedge$ will denote the quotient map
	$$ \wedge: \E \tensora \E \to \twoform. $$
	Finally, we will denote $\rm{Ker}$$(\wedge)$ by the symbol $\E \tensorsym \E.$ Thus,
	\begin{equation} \label{22ndaugust20191}
	{\rm Ker}(\wedge) = {\rm Ker}(\sigma - 1) = \E \tensorsym \E
	\end{equation}
\end{defn}
It can be easily checked that $\wedge$ is a bimodule map and hence $\rm Ker$$(\wedge)$ is an $\A$-$\A$ sub-bimodule. The higher order forms are defined similarly. We refer to \cite{woronowicz} for the details. In particular, we define $ \Omega^0 (\A) = \A $ and $ \Omega^1 (\A) = \E. $ Then $ \wedge $ extends to a map
 $$ \wedge : \Omega^k (\A) \tensora \Omega^l (\A) \rightarrow \Omega^{k + l} (\A). $$
Let us now collect some more facts from \cite{woronowicz} which are going to be useful in the sequel. Note that by virtue of \eqref{21staugust20194jb}, the map $\sigma$ is both left and right covariant.

\bppsn \label{25thmay20192}
Suppose $ (\E, d) $ is a first order bicovariant differential calculus on $\A$ and $ \Omega^k (\A) $ be the higher order forms. The left and right comodule coactions $ \Delta_{\E \tensora \E} $ and $ {}_{\E \tensora \E} \Delta $ of $ \A $ on $ \E \tensora \E $ descend to comodule coactions of $ \A $ on $ \twoform $ as $ {\rm Ker} (\sigma - 1) $ is left and right-invariant. This makes $ \twoform $ a bicovariant $\A$-$\A$-bimodule. The same is true for $ \Omega^k (\A) $ for all $ k \geq 0. $ 

Moreover, the map $ d $ extends to a bicovariant map from $ \oplus_{k \geq 0} \Omega^k (\A) $ to itself and satisfies $ d^2 = 0 $ and 
$$ d (\theta \wedge \theta^\prime) = d \theta \wedge \theta^\prime + (- 1)^k \theta \wedge d \theta^\prime $$
if $ \theta \in \Omega^k (\A) $ and $ \theta^\prime \in \Omega^l (\A).$
\eppsn

Our definition of two-forms is in general different than that considered in \cite{heckenberger}. 

\begin{rmk} \label{28thnov20191} Suppose $\A$ is a $q$-deformation of a classical compact semisimple Lie group and $\E$ be a bicovariant bimodule over $\A.$ Then typically, the (q-dependent) eigenvalues of $\sigma$ consist of real numbers other than $\pm 1.$ Let $I$ be the set of eigenvalues of $\sigma$ which tend (in limit) to $1$ as $ q  $ tends to $1.$ 

The authors of  \cite{watamura} define $$ \twoform = \frac{ \E \tensora \E }{  \Pi_{\lambda \in I } (\sigma - \lambda   ) }. $$ 
It is this definition of $\twoform$ which was taken in \cite{heckenberger}. Thus, the definition of two-forms considered in this article are different than that in \cite{heckenberger} unless the only eigenvalues of $\sigma$ are $\pm 1.$
\end{rmk}

\section{Splitting of the space of two tensors and the symmetric two tensors} \label{24thmay20192}

For a Riemannian manifold $ M, $ the $\smoothfn$-$\smoothfn$ bimodule $\oneformclassical \otimes_{\smoothfn} \oneformclassical$ admits the canonical decomposition: 
\begin{equation} \label{21staugust20198} \oneformclassical \otimes_{\smoothfn} \oneformclassical = \text{Ker} (\wedge) \oplus \F , \end{equation}
 The bimodule $\text{Ker} (\wedge)$ can be identified with the symmetric two-tensors and the bimodule $ \F $ can be identified with the antisymmetric $2$-tensors and therefore, $ \F $ is isomorphic to $ \Omega^2 (M). $ If $ \sigma $ denotes the canonical flip map, i.e. $ \sigma (e \otimes_{C^\infty (M) } f) = f \otimes_{C^\infty (M) } e, $ we have an idempotent
$$ \Psym:= \frac{\sigma + 1}{2} \in \Hom_{C^\infty (M)} (\Omega^1 (M) \otimes_{C^\infty (M) } \Omega^1 (M), \Omega^1 (M) \otimes_{C^\infty (M) } \Omega^1 (M)) $$
 with range $ {\rm Ker} (\wedge) $ and kernel $ \F. $ The aim of this section is to prove a noncommutative analogue of the decomposition \eqref{21staugust20198} under a mild assumption (Theorem \ref{25thmay20195}). This decomposition will help us to prove the existence of a bicovariant torsionless connection on a bicovariant differential calculus (see Theorem \ref{20thaugust20192}).

Let $(\E,d)$ be a bicovariant differential calculus on a Hopf algebra $\A$. Then the coactions $\Delta_{\E \tensora \E}$ and ${}_{\E \tensora \E} \Delta$ turn $\E \tensora \E$ into a bicovariant bimodule. In particular, Proposition \ref{moduleiso} guarantees the isomorphism of the multiplication map
\begin{equation} \label{22ndaugust20193}
\widetilde{u}^{\E \tensora \E}: (\zeroE \tensorc \zeroE) \tensorc \A = {}_0(\E\tensora \E) \tensorc \A \to \E \tensora \E
\end{equation} 
Moreover, by Proposition \ref{4thmay20193}, we have a canonical $\A$-$\A$ bimodule map $\sigma$ from $\E \tensora \E$ to $\E \tensora \E$. As noted in the previous subsection, by \eqref{21staugust20194jb}, the map $\sigma$ is both left and right covariant. Then Proposition \ref{ahoma} and Proposition \ref{3rdaugust20192} imply that there exists a unique map
\begin{equation} \label{21staugust20195}
	{}_0 \sigma : \zeroE \tensorc \zeroE = {}_0 (\E \tensora \E) \to {}_0 (\E \tensora \E) = \zeroE \tensorc \zeroE
\end{equation}
such that 
\begin{equation} \label{8thnov20196} (\widetilde{u}^{\E \tensora \E})^{-1} \sigma = ({}_0\sigma \tensorc \id)(\widetilde{u}^{\E \tensora \E})^{-1}.\end{equation}
 For the rest of this article, we will make the assumption that the map ${}_0 \sigma: \zeroE \tensorc \zeroE \to \zeroE \tensorc \zeroE$ is diagonalisable. This assumption holds for a large class of Hopf algebras as indicated in the next proposition.

\begin{prop} \label{14thoct20191}
	Let $\E$ be the space of one-forms of a first order differential calculus over a Hopf algebra and ${}_0 \sigma : \zeroE \tensorc \zeroE \to \zeroE \tensorc \zeroE$ be the map as in \eqref{21staugust20195}. Then
	\begin{enumerate}
		\item[(i)] For the classical bicovariant differential calculus on a Lie group, the map ${}_0 \sigma$ is diagonalisable.
		\item[(ii)] Let $(\E,d)$ be the bicovariant differential calculus on the algebra $\A$ of regular functions on a linear algebraic group $G$ such that the category of finite dimensional representations of $G$ is semisimple. Suppose $\A_\Omega$ is the cocycle deformation of $\A$ with respect to an invertible $2$-cocycle $\Omega$ (see Section \ref{sectioncocyclelc}). Then we have a canonical bicovariant differential calculus $(\E_\Omega,d_\Omega)$ on $\A_\Omega$ obtained by deforming the usual bicovariant differential calculus on $\A$ (see Proposition \ref{11thmay20192}). Let $\sigma_\Omega$ be the braiding map of Proposition \ref{4thmay20193} applied to the calculus $(\E_\Omega, d_\Omega)$. Then ${}_0 (\sigma_\Omega) : {}_0(\E_\Omega \tensora \E_\Omega) \to {}_0 (\E_\Omega \tensora \E_\Omega)$ is diagonalisable.
		\item[(iii)] The assumption holds for the bicovariant differential calculi on $SL_q(N)$, $O_q(N)$, $Sp_q(N)$ studied in \cite{heckenberger}. More generally, if the map $\sigma$ satisfies a Hecke-type relation $\Pi_i (\sigma - \lambda_i) = 0$ for distinct scalars $\lambda_i$, then ${}_0 \sigma$ is diagonalisable.
	\end{enumerate} 
\end{prop}
\begin{proof}
	Suppose the map $\sigma$ satisfies a relation $\Pi_i (\sigma - \lambda_i) = 0$ for distinct scalars $\lambda_i$. Since ${}_0 \sigma(\zeroE \tensorc \zeroE) \subseteq \zeroE \tensorc \zeroE$, we have the equality $\Pi_i ({}_0\sigma - \lambda_i) = 0$ as maps from $\zeroE \tensorc \zeroE$ to itself. Therefore, the minimal polynomial of the map ${}_0 \sigma$ is a product of distinct linear factors and so ${}_0 \sigma$ is diagonalisable. Since the bicovariant differential calculi on $SL_q(N)$, $O_q(N)$ and $Sp_q(N)$ studied in \cite{heckenberger} satisfy Hecke-type relations as above, this completes the proof of part (iii). The classical case follows similarly, since here $\sigma(e \tensora f) = f \tensora e$ for all $e, f$ in $\E$, so that $\sigma^2 - 1 = 0$ and therefore, the above reasoning applies. Finally, we refer to Theorem \ref{11thmay20192} for the proof of part (ii).
\end{proof}

Let us introduce the following notations.
\begin{defn} \label{22ndaugust20192}
	Suppose the map ${}_0 \sigma$ is diagonalisable. The eigenspace decomposition of $\zeroE \tensorc \zeroE$ will be denoted by $\zeroE \tensorc \zeroE = \bigoplus_{\lambda \in \Lambda} V_\lambda$, where $\Lambda$ is the set of distinct eigenvalues of ${}_0 \sigma$ and $V_\lambda$ is the eigenspace of ${}_0\sigma$ corresponding to the eigenvalue $\lambda$. Thus, $ V_1 $ will denote the eigenspace of $ {}_0 \sigma $ for the eigenvalue $ \lambda = 1. $
	
Moreover,	we define $\zeroE \tensorc^{\rm sym} \zeroE $ to be the set of all left-invariant elements of $\E \tensorsym \E, $ i.e, 
	\[ \zeroE \tensorc^{\rm sym} \zeroE := \Big\{\sum_k \rho_k \tensora \nu_k \in \E \tensora \E: \Delta_{\E \tensora \E} (\sum_k \rho_k \tensora \nu_k)= 1\tensorc \sum_k \rho_k \tensora \nu_k, \sum_k \rho_k \wedge \nu_k = 0 \Big\}.\]
 	We also define ${}_0 \F := \bigoplus_{\lambda \in \Lambda \backslash \{1\}} V_\lambda.$ 
\end{defn}
The assumption that $ {}_0 \sigma $ is diagonalisable is enough to prove Theorem \ref{25thmay20195} As a first step to prove that theorem, we make the following observation: 
\begin{lemma} \label{13thjune20191}
 Let ${}_0\sigma$ be the map in \eqref{21staugust20195}. Then we have
 $$ {\rm We} ~ {\rm have} ~ \zeroE \otimes^{{\rm sym}}_{\mathbb{C}} \zeroE = {\rm Ker} ({}_0 \sigma - 1).$$
\end{lemma}
\begin{proof}
	The result follows by a simple computation. Indeed,
\begin{equation*}
	\begin{aligned}
		 &\zeroE \otimes^{{\rm sym}}_{\mathbb{C}} \zeroE\\
		=&\{\sum_k \rho_k \tensora \nu_k \in \E \tensora \E: \Delta_{\E \tensora \E} (\sum_k \rho_k \tensora \nu_k)\\ & \quad = 1\tensorc \sum_k \rho_k \tensora \nu_k, \sum_k \rho_k \wedge \nu_k = 0\}\\
		=&\{\sum_k \rho_k \tensora \nu_k \in \E \tensora \E: \Delta_{\E \tensora \E} (\sum_k \rho_k \tensora \nu_k) \\ =& 1\tensorc \sum_k \rho_k \tensora \nu_k, (\sigma -1)(\sum_k \rho_k \tensora \nu_k)=0\}\\
		&{\rm \ (since \ Ker}(\wedge) = {\rm \ Ker}(\sigma - 1) {\rm \ by \ \eqref{22ndaugust20191}})\\
		=&\{\sum_k \rho_k \tensora \nu_k \in \zeroE \tensorc \zeroE: ({}_0\sigma -1)(\sum_k \rho_k \tensorc \nu_k)=0\}\\
		& {\rm (} ~ {\rm as} ~ {}_0 (\E \tensora \E) = \zeroE \tensorc \zeroE ~ {\rm by} ~ {\rm Proposition} ~ \ref{3rdaugust20192} ~ {\rm)}\\ 
		=&{\rm Ker} ({}_0 \sigma - 1).
	\end{aligned}	
\end{equation*}	
\end{proof} 

\begin{rmk}
Let $\zeroE \otimes^{{\rm sym}}_{\mathbb{C}} \zeroE $ and $ {}_0 \F $ be as in Definition \ref{22ndaugust20192}.	We note that $\zeroE \otimes^{{\rm sym}}_{\mathbb{C}} \zeroE = V_1$, where $V_\lambda$ is as in Definition \ref{22ndaugust20192}. Furthermore, since $ {}_0 \sigma $ is diagonalisable, we have the following decomposition:
	\begin{equation} \label{18thjuly20191}
	\zeroE \tensorc \zeroE = \zeroE \otimes^{{\rm sym}}_{\mathbb{C}} \zeroE \oplus {}_0\F.
	\end{equation}
\end{rmk}

\begin{thm} \label{25thmay20195}
	Let $\widetilde{u}^{\E \tensora \E}$ be the isomorphism of \eqref{22ndaugust20193}. We define $\F := \widetilde{u}^{\E \tensora \E}({}_0 \F \tensorc \A)$. Then $\wedge|_\F: \F \to \twoform$ defines an isomorphism of right $\A$-modules. Moreover, 
	$$\E \tensora \E = {\rm Ker} (\wedge) \oplus \F = \E \tensora^{\rm sym} \E \oplus \F.$$
\end{thm} 
\begin{proof}
The proof easily follows by a computation and the following observation:
\begin{equation} \label{22ndaugust20194}
{}_0(\E \tensorsym \E) = {}_0 \E \tensorc^{\rm sym} \zeroE~ {\rm and} ~ {\rm so} ~ {\widetilde{u}}^{\E \tensora \E} ({}_0 \E \tensorc^{\rm sym} \zeroE \tensorc \A) = \E \tensorsym \E. 
\end{equation}
The equation $ {}_0(\E \tensorsym \E) = {}_0 \E \tensorc^{\rm sym} \zeroE $ follows directly from the definitions of $ {}_0 \E \tensorc^{\rm sym} \zeroE $ and $ \E \tensorsym \E = {\rm Ker} (\wedge). $ Then the second equation of \eqref{22ndaugust20194} follows from Proposition \ref{moduleiso} once we prove that $ \E \tensorsym \E $ is a bicovariant bimodule. 

$ \E \tensorsym \E $ is a bimodule since it is the kernel of the bimodule map $ \wedge. $ Since $ \E \tensorsym \E \subseteq \E \tensora \E, $ it is enough to check that $ \E \tensorsym \E $ is invariant under both $ \Delta_{\E \tensora \E} $ and $ {}_{\E \tensora \E} \Delta. $ Now, if $ X \in \E \tensorsym \E, $ then by \eqref{21staugust20194jb}, we get
$$ (\id \tensorc (\sigma - \id)) \Delta_{\E \tensora \E} (X) = \Delta_{\E \tensora \E} (\sigma - 1) (X) = 0$$
as $ X \in \E \tensorsym \E = {\rm Ker} (\wedge) = {\rm Ker} (\sigma - \id) $ by \eqref{22ndaugust20191}. This proves that 
$$ \Delta_{\E \tensora \E} (\E \tensorsym \E) \subseteq \A \tensorc {\rm Ker} (\sigma - \id) = \A \tensorc {\rm Ker} (\wedge) = \A \tensorc (\E \tensorsym \E). $$
Similarly, $ \E \tensorsym \E $ is also invariant under $ {}_{\E \tensora \E} \Delta. $ This proves our claim that $ \E \tensorsym \E $ is a bicovariant $\A$-bimodule which completes the proof of \eqref{22ndaugust20194}.
 
Now we can compute:
	\begin{equation*}
	\begin{aligned}
		 &\E \tensora \E = \widetilde{u}^{\E \tensora \E}(\widetilde{u}^{\E \tensora \E})^{-1}(\E \tensora \E)\\
		 =& \widetilde{u}^{\E \tensora \E}((\zeroE \tensorc \zeroE) \tensorc \A) = \widetilde{u}^{\E \tensora \E}(((\zeroE \tensorc^{\rm sym} \zeroE) \oplus {}_0 \F) \tensorc \A) {\rm\ (by \eqref{18thjuly20191})}\\
		=& \widetilde{u}^{\E \tensora \E}(((\zeroE \tensorc^{\rm sym} \zeroE) \tensorc \A) \oplus ({}_0 \F \tensorc \A)) = \E \tensorsym \E \oplus \F \\ &{\rm \ (by \ \eqref{22ndaugust20194} \ and \ the \ definition \ of \ \F)}\\
		=& {\rm Ker}(\wedge) \oplus \F {\rm \ (by \ the \ definition \ of \ \E \tensorsym \E)}.
	\end{aligned}	
	\end{equation*}
	Finally, since $\E \tensora \E = {\rm Ker}(\wedge) \oplus \F$, we have that
	$$ \F \cong (\E \tensora \E)/{\rm Ker}(\wedge) = (\E \tensora \E)/{\rm Ker}(\sigma - 1) = \twoform, $$
	by \eqref{22ndaugust20191} and the definition of $\twoform$. Hence, $\wedge|_{\F}: \F \to \twoform $ is an isomorphism of right $\A$-modules.
\end{proof}

\subsection{The idempotent $\Psym$ and its properties}
In this subsection, we study the idempotent element of $ \Hom_\A (\E \tensora \E, \E \tensora \E) $ with range $ \E \tensorsym \E $ and kernel $ \F. $ 
\begin{defn} \label{27thmay20193}
	We will denote by ${}_0(\Psym)$ the idempotent element in Hom$(\zeroE \tensorc \zeroE, \zeroE \tensorc \zeroE)$ with range $\zeroE \tensorc^{\rm sym} \zeroE$ and kernel ${}_0\F$. By Proposition \ref{inviff}, ${}_0(\Psym)$ extends to a right $\A$-linear left-covariant map from $\E \tensora \E$ to $\E \tensora \E$. We are going to denote the extension by the symbol $\Psym$. More concretely, $$\Psym:= \widetilde{u}^{\E \tensora \E}({}_0(\Psym) \tensorc \id)(\widetilde{u}^{\E \tensora \E})^{-1}.$$
\end{defn}

\begin{prop} \label{16thseptember20191}
	The map $\Psym$ is the idempotent map from $\E \tensora \E$ to itself, with range onto $\E \tensorsym \E$ and with kernel $\F$. In fact, $\Psym$ is also a left $\A$-linear and bicovariant map. Thus $\Psym$ is bicovariant map from $\E \tensora \E$ to itself which is both left and right $\A$-linear.
\end{prop}
\begin{proof}
	By Definition \ref{27thmay20193}, $\Psym$ is a left-covariant right $\A$-linear map from $\E \tensora \E$ to itself. Since ${}_0(\Psym)$ is an idempotent, $\Psym= \widetilde{u}^{\E \tensora \E}({}_0(\Psym) \tensorc \id)(\widetilde{u}^{\E \tensora \E})^{-1}$ is also idempotent. We have that
	\begin{equation*}
		\begin{aligned}
		 &{\rm Ran}(\Psym) = \widetilde{u}^{\E\tensora \E}({}_0(\Psym) \tensorc \id)(\widetilde{u}^{\E \tensora \E})^{-1}(\E \tensora \E)\\
		=&\widetilde{u}^{\E\tensora \E}({}_0(\Psym) \tensorc \id)((\zeroE \tensorc \zeroE) \tensorc \A) = \widetilde{u}^{\E\tensora \E}((\zeroE \tensorc^{\rm sym} \zeroE) \tensorc \A)\\
		 &{\rm (by \ the \ definition \ of \ } {}_0(\Psym){\rm)}\\
		=&\E \tensorsym \E ~ {\rm (} ~ {\rm by} ~ \eqref{22ndaugust20194} {\rm)}.
		\end{aligned}
	\end{equation*}
	 Now we prove that Ker$(\Psym) = \F$. We note that $\Psym$ is an idempotent from the complex vector space $\E\tensora \E$ to itself with range $\E \tensorsym \E$ and kernel containing the subspace $\widetilde{u}^{\E \tensora \E}({}_0 \F \tensorc \A) = \F$. Since $\E \tensora \E = \E \tensorsym \E \oplus \F$ (Theorem \ref{25thmay20195}), this proves that ${\rm Ker}(\Psym) = \F$.\\
	 Finally, we prove that $\Psym$ is a bicovariant $\A$-bimodule map. this follows from the observation that $ {}_0 (\Psym) $ is a polynomial in $ {}_0 \sigma.$ Indeed, in the notation of Definition \ref{22ndaugust20192}, $ {}_0(\Psym) $ is the idempotent with range $ V_{1} $ and kernel $ \oplus_{\lambda \in \Lambda, \lambda \neq 1} V_\lambda $ and so 
	\begin{equation} \label{24thaugust20191}
		{}_0(\Psym) = \Pi_{\lambda \in \Lambda \backslash \{1\} } \frac{{}_0 \sigma - \lambda}{1 - \lambda}.
	\end{equation}
	 Therefore, 
	 \begin{equation*}
	 \begin{aligned}
	 &\Psym = \widetilde{u}^{\E \tensora \E}\big({}_0(\Psym) \tensorc \id \big)(\widetilde{u}^{\E \tensora \E})^{-1}\\ =& \widetilde{u}^{\E \tensora \E}\Big(\big(\Pi_{\lambda \in \Lambda \backslash \{1\} } \frac{1}{1 - \lambda}({}_0 \sigma - \lambda)\big) \tensorc \id\Big)(\widetilde{u}^{\E \tensora \E})^{-1}\\
	 =&\Pi_{\lambda \in \Lambda \backslash \{1\} }\Big(\widetilde{u}^{\E \tensora \E} \big((\frac{1}{1 - \lambda}({}_0 \sigma - \lambda)) \tensorc \id \big)(\widetilde{u}^{\E \tensora \E})^{-1} \Big) = \Pi_{\lambda \in \Lambda \backslash \{1\}}\big(\frac{1}{1 - \lambda}(\sigma - \lambda)\big)
	 \end{aligned}
	 \end{equation*}
	by \eqref{8thnov20196}.Hence,
	\begin{equation} \label{24thaugust20192} \Psym = \Pi_{\lambda \in \Lambda \backslash \{1\}}\big(\frac{1}{1 - \lambda}(\sigma - \lambda)\big). \end{equation}
	 Now $\sigma$ is a bicovariant $\A$-bimodule map from $\E \tensora \E$ to itself and so $\Psym$, being a composition of bicovariant $\A$-bimodule maps from $ \E \tensora \E $ to $ \E \tensora \E $ is itself a bicovariant $\A$-bimodule map from from $ \E \tensora \E $ to $ \E \tensora \E. $
\end{proof}
In the classical case, we have $ \Lambda = \pm 1 $ and so in this case, we recover the formula $ \Psym = \frac{1}{2}(1 + \sigma) $ from \eqref{24thaugust20192}. Let us collect two facts which will be used in the sequel in the following remark.
\begin{rmk} \label{24thaugustremark}
 Since $ \wedge $ and $ \Psym $ are bimodule maps, the right $ \A $-modules $ \E \tensorsym \E = {\rm Ker} (\wedge) $ and $ \F = {\rm Ran} (1 - \Psym) $ are actually $\A$-$\A$-bimodules.
\end{rmk}

\begin{defn} \label{5thseptember20191sm}
	Let $\F$ be the submodule of $\E \tensora \E$ as in Theorem \ref{25thmay20195}. By Theorem \ref{25thmay20195}, we have a right $\A$-linear isomorphism $ \wedge|_\F: \F \to \twoform$ which we will denote by $Q$.
\end{defn}
Let us note that since $ \wedge $ is an $\A$-bimodule map and $ \F $ is an $\A$-$\A$-bimodule, $ Q $ is actually a bimodule isomorphism from $ \F $ to $ \twoform. $ We will use this fact in the following result which is a corollary to Proposition \ref{16thseptember20191}.
\begin{corr} \label{25thmay20193}
	If $(\E,d)$ is a bicovariant first order differential calculus, then $d\omega$ is in ${}_0(\twoform)$ for all $\omega$ in $\zeroE = {}_0(\oneform)$. Moreover, $\wedge$ and $ Q $ are bicovariant maps.
\end{corr}
\begin{proof}
As noted above, $Q$ is an $\A$-$\A$ bimodule isomorphism. Moreover, by Proposition \ref{25thmay20192}, the comodule coactions $\Delta_{\E \tensora \E}$ and ${}_{\E \tensora \E}\Delta$ descend to the coactions $\Delta_{\twoform}$ and ${}_{\twoform}\Delta$ respectively as ${\rm Ker}(\sigma - 1)$ is both right and left-invariant. Thus, the map $Q$ is covariant with respect to the left and right $\A$-coactions on $\F$ and $\twoform$. In particular, this implies that 
\begin{equation} \label{25thmay20196} Q^{-1}({}_0(\twoform)) \subseteq {}_0\F ~ {\rm (}~ {\rm Proposition} ~ \ref{14thfeb20192} ~ {\rm)}. \end{equation} 
Since $\omega$ is in $\zeroE$ and $d$ is bicovariant (\eqref{23rdaugust20191}), we have 
$$\Delta_{\twoform}(d\omega)=(\id_\A \tensorc d)\Delta_\E(\omega)= 1 \tensorc d\omega.$$
For proving the second statement, we note that since $ \Psym $ is a bicovariant map (Proposition \ref{16thseptember20191}), $ \F = {\rm Ker} (\Psym) $ is also invariant under $ \Delta_{\E \tensora \E} $ and $ {}_{\E \tensora \E} \Delta.$ Let $ x $ be an element of $ \E \tensora \E. $ Then Theorem \ref{25thmay20195} guarantees the existence of an element $ y $ in ${\rm Ker} (\wedge) = {\rm Ker} (\sigma - 1) $ and $ z $ in $\F$ such that $ x = y + z. $

As ${\rm Ker} (\sigma - 1) $ and $\F$ are invariant under $ \Delta_{\E \tensora \E}, $ we have 
\begin{eqnarray*}
&& (\wedge \tensorc \id) \Delta_{\E \tensora \E} (x) = (\wedge \tensorc \id) (\Delta_{\E \tensora \E} (y) + \Delta_{\E \tensora \E} (z))\\
&=& 0 + (\wedge \tensorc \id) \Delta_{\E \tensora \E} (z) = (Q \tensorc \id) \Delta_{\E \tensora \E} (z)\\
&=& \Delta_{\Omega^2 (\A)} Q (z) = \Delta_{\Omega^2 (\A)} (\wedge (z)) = \Delta_{\Omega^2 (\A)} (\wedge (x)),
\end{eqnarray*} 
where, in the last step, we have used the left-covariance of the map $Q.$ Therefore, $\wedge$ is left-covariant. Similarly, $\wedge$ is right-covariant. 
\end{proof}
We end this section with one more lemma which will be needed in the proofs of Lemma \ref{15thjune20192} and Theorem \ref{15thjune20193}. This will need a notation.

\begin{defn} \label{2ndaugust20191} 
	Let $V$ and $W$ be finite dimensional complex vector spaces. The canonical vector space isomorphism from $ V \tensora W^* $ to $ \Hom_\IC (W, V) $ will be denoted by the symbol $ \zeta_{V, W}. $ It is defined by the formula:
	\begin{equation} \label{24thaugust20194}
	\zeta_{V,W}(\sum_i v_i \tensora \phi_i)(w) = \sum_i v_i \phi_i(w).
	\end{equation}
\end{defn}

\begin{lemma} \label{15thjune20192}
	The following maps are vector space isomorphisms: 
	$$\zeta_{\zeroE \tensorc \zeroE , \zeroE}:(\zeroE \tensorc^{\rm sym} \zeroE) \tensorc (\zeroE)^* \to {\rm Hom}_\IC(\zeroE, \zeroE \tensorc^{\rm sym} \zeroE), $$ 
	$$ \zeta_{\zeroE, \zeroE \tensorc \zeroE }: \zeroE \tensorc (\zeroE \tensorc^{\rm sym} \zeroE)^* \to \homc (\zeroE \tensorc^{\rm sym} \zeroE, \zeroE).$$ 
\end{lemma}
\begin{proof} By the definition of the map $ \zeta_{\zeroE \tensorc \zeroE , \zeroE},$
	 $$\zeta_{\zeroE\tensorc \zeroE , \zeroE} ((\zeroE \tensorc^{\rm sym} \zeroE) \tensorc (\zeroE)^*) \subseteq \homc(\zeroE, \zeroE \tensorc^{\rm sym} \zeroE).$$
	Since $\zeta_{\zeroE\tensorc \zeroE , \zeroE}$ is an isomorphism from $(\zeroE \tensorc \zeroE) \tensorc (\zeroE)^*$ onto $\homc(\zeroE, \zeroE \tensorc \zeroE)$ and $${\rm dim}((\zeroE \tensorc^{\rm sym} \zeroE) \tensorc (\zeroE)^*)={\rm dim}(\zeroE, \zeroE \tensorc^{\rm sym} \zeroE),$$
	we have proved the first assertion.
	
	Now we prove the second assertion. By the definition of ${}_0(\Psym)$ (Definition \ref{27thmay20193}), 
	$$\zeroE \tensorc \zeroE = {\rm Ran}({}_0(\Psym)) \oplus {\rm Ran}(1- {}_0(\Psym))$$
	and hence an element $ \psi $ of $(\zeroE \tensorc^{\rm sym} \zeroE)^* = ({\rm Ran}({}_0(\Psym)))^* $ extends to an element $\widetilde{\psi}$ of $(\zeroE \tensorc \zeroE)^*$ by the formula 
	$$\widetilde{\psi}(X)=\psi({}_0(\Psym)(X)) ~ \forall X ~ \in \zeroE \tensorc \zeroE.$$
	More generally,
	\begin{equation} \label{23rdaugust20193}
	\homc(\zeroE \tensorc^{\rm sym} \zeroE, \IC) = \{ \psi \in \homc(\zeroE \tensorc \zeroE, \IC) : \psi((1-{}_0(\Psym))(X))=0 \ \forall X \in \zeroE \tensorc \zeroE \}. 
	\end{equation}
	This allows us to view $ \psi \in (\zeroE \tensorc^{\rm sym} \zeroE)^* $ as an element $ \widetilde{\psi} \in (\zeroE \tensorc \zeroE)^* $ such that $ \widetilde{\psi} ((1 - {}_0(\Psym)) (X)) = 0. $ 
	
	Thus, for $e$ in $\zeroE$, $\widetilde{\psi}$ as above and for all $X$ in $\zeroE \tensorc \zeroE$, we have
	$$ (\zeta_{\zeroE, \zeroE \tensorc \zeroE })(e \tensorc \widetilde{\psi})((1 - {}_0(\Psym))(X)) = e \widetilde{\psi}((1 - {}_0(\Psym))(X)) = 0.$$
	This implies that 
	$$ \zeta_{\zeroE, \zeroE \tensorc \zeroE }(\zeroE \tensorc (\zeroE \tensorc^{\rm sym} \zeroE)^*) \subseteq \homc(\zeroE \tensorc^{\rm sym} \zeroE, \zeroE).$$
	As $\zeta_{\zeroE, \zeroE \tensorc \zeroE }$ is an isomorphism from $\zeroE \tensorc (\zeroE \tensorc \zeroE)^*$ onto $\homc(\zeroE \tensorc \zeroE, \zeroE)$ and ${\rm dim}(\zeroE \tensorc (\zeroE \tensorc^{\rm sym} \zeroE)^*) = {\rm dim}(\homc(\zeroE \tensorc^{\rm sym} \zeroE, \zeroE))$, $\zeta_{\zeroE, \zeroE \tensorc \zeroE }$ maps $\zeroE \tensorc (\zeroE \tensorc^{\rm sym} \zeroE)^*$ isomorphically onto $\homc(\zeroE \tensorc^{\rm sym} \zeroE, \zeroE)$. This finishes the proof of the lemma.
\end{proof}

\section{Some generalities on pseudo-Riemannian metrics} \label{section2}

 In this section, we discuss some results on covariant metrics on bicovariant bimodules. After defining pseudo-Riemannian metrics, we collect some relevant results from \cite{heckenberger} and \cite{article5}. The Subsection \ref{gtwoadjoint} deal with the adjoints of the maps $ {}_0 \sigma $ and $ {}_0(\Psym) $ with respect to a certain map $ g^{(2)} $ on $ \zeroE \tensorc \zeroE. $ 

\begin{defn} \label{24thmay20191}
	Suppose $ \E $ is a bicovariant $\A$-$\A$ bimodule $\E$ and $ \sigma: \E \tensora \E \rightarrow \E \tensora \E $ be the map as in Proposition \ref{4thmay20193}. A pseudo-Riemannian metric for the pair $ (\E, \sigma) $ is a right $\A$-linear map $g:\E \tensora \E \to \A$ such that the following conditions hold:
	\begin{enumerate}
	\item[(i)] $ g \circ \sigma = g. $
	\item[(ii)] If $g(\rho \tensora \nu)=0$ for all $\nu$ in $\E,$ then $\rho = 0.$ 
	\end{enumerate}
\end{defn}
Let us recall (\cite{heckenberger}) that a pseudo-Riemannian metric $g$ on a left-covariant $\A$-$\A$ bimodule $\E$ is said to be left-invariant if for all $\rho, \nu$ in $\E$,
	$$ (\id \tensorc \epsilon g)(\Delta_{(\E \tensora \E)}(\rho \tensora \nu)) = g(\rho \tensora \nu). $$
	Similarly, a pseudo-Riemannian metric $g$ on a right-covariant $\A$-$\A$ bimodule $\E$ is said to be right-invariant if for all $\rho, \nu$ in $\E$,
	$$ (\epsilon g \tensorc \id)({}_{(\E \tensora \E)}\Delta(\rho \tensora \nu)) = g(\rho \tensora \nu). $$
	Finally, a pseudo-Riemannian metric $g$ on a bicovariant $\A$-$\A$ bimodule $\E$ is said to be bi-invariant if it is both left-invariant as well as right-invariant. Then we have the following results.
	
\begin{lemma} (\cite{heckenberger}, \cite{article5}) \label{14thfeb20191}
 Suppose $\E$ is a bicovariant bimodule and $g$ a pseudo-Riemannian metric on $\E.$ We will say that $ g $ is left (respectively right)-covariant if $ g $ is a left (respectively, right)-covariant map between the bicovariant bimodules $ \E \tensora \E $ and $ \A. $
	
 Then $g$ is left-invariant if and only if $g: \E \tensora \E \to \A$ is a left-covariant map. Similarly, $g$ is right-invariant if and only if $g: \E \tensora \E \to \A$ is a right-covariant map.
	
 Consequently, if $g$ is a pseudo-Riemannian metric which is left-invariant on $\E$, then $g(\omega_1 \tensora \omega_2) \in \IC.1$ for all $\omega_1, \omega_2$ in $\zeroE.$ Similarly, if $ g $ is a right-invariant pseudo-Riemannian metric on $\E,$ then $ g (\eta_1 \tensora \eta_2) \in \mathbb{C}. 1 $ for all $ \eta_1, \eta_2 $ in $ \E_0. $
\end{lemma}
\begin{proof}
The first assertion was proved in Proposition 3.3 of \cite{article5}. The second assertion was already noted in \cite{heckenberger}. It follows by a combination of Proposition \ref{14thfeb20192} and Proposition \ref{3rdaugust20192}. 
\end{proof}

\bdfn \label{24thaugust20195}
 Let $ \E $ and $ g $ be as above. For a fixed basis $ \{ \omega_1, \cdots , \omega_n \} $ of $ \zeroE, $ we define $ g_{ij} = g (\omega_i \tensora \omega_j). $ 
We define a map 
 $$ V_g: \E \rightarrow \E^*, ~ V_g (e) (f) = g (e \tensora f). $$
\edfn

The following result will be used multiple times in the sequel.

\begin{prop}[Proposition 3.6 of \cite{article5}] \label{23rdmay20192}
	Let $g$ be a left-invariant pseudo-Riemannian metric for a pair $(\E,d)$ as in Definition \ref{24thmay20191}. Then the following statements hold:
	\begin{itemize}
	 \item[(i)] The map $ V_g $ is a one-one right $ \A $-linear map from $ \E $ to $ \E^*. $ If $ e \in \E $ is such that $ g (e \tensora f) = 0 $ for all $ f \in {}_0 \E, $ then $ e = 0. $ In particular, the map $ V_{g} $ is one-one and hence a vector space isomorphism from $ {}_0 \E $ to $ ({}_0 \E)^*.$
	\item[(ii)] The matrix $((g_{ij}))_{ij}$ is invertible. Let $g^{ij}$ denote the $(i,j)$-th entry of the inverse of the matrix $((g_{ij}))_{ij}$. Then $g^{ij}$ is an element of $\IC.1$ for all $i,j$.
	\end{itemize}
\end{prop}

\subsection{The $ g^{(2)} $-adjoint of a left-covariant map} \label{gtwoadjoint}
Suppose $\E$ is a bicovariant bimodule and $g$ a pseudo-Riemannian metric. Then following the lines of \cite{article1} and \cite{article2}, it is straightforward to define (Definition \ref{27thaugust2019night1}) a complex valued map $ g^{(2)} $ on $ \zeroE \tensorc \zeroE. $ The goal of this subsection is to show that any complex linear map from $ \zeroE \tensorc \zeroE $ to itself admits an adjoint with respect to $ g^{(2)}.$ Moreover, in Lemma \ref{14thjune20191} and Proposition \ref{15thjune20195}, we show that the maps $ {}_0 \sigma $ and $ {}_0(\Psym) $ are actually self-adjoint. These results will be used in Lemma \ref{15thnov20191} and Theorem \ref{15thjune20193} for deriving a sufficient condition for the existence of a Levi-Civita connection. For similar results in the context of a class of centered bimodules, we refer to Lemma 3.5 and Lemma 4.17 of \cite{article1}. However, in contrast to \cite{article1}, we do not have the left $\A$-linearity of the metric but the bicovariance of $g$ and $\sigma$ help us to derive the results.

Let $ \E $ be a bicovariant bimodule over $\A$ and $ \{ \omega_i \}_i $ a basis of $ \zeroE. $ Then the set $ \{ \omega_i \tensorc \omega_j \}_{ij} $ is a basis for the finite dimensional vector space $ \zeroE \tensorc \zeroE. $ Thus, we are allowed to make the following definition.
 \begin{defn} \label{27thaugust2019night1}
	Suppose $g$ is a left-covariant pseudo-Riemannian metric on $\E$.
	We define a map 
		$$ g^{(2)}: (\zeroE \tensorc \zeroE) \tensorc (\zeroE \tensorc \zeroE) \rightarrow \mathbb{C} ~ {\rm by} ~ {\rm the} ~ {\rm formula} $$
		$$ g^{(2)} ((\omega_1 \tensorc \omega_2) \tensorc (\omega_3 \tensorc \omega_4)) = g(\omega_1 \tensora g(\omega_2 \tensora \omega_3) \tensora \omega_4) $$
	for all $\omega_1, \omega_2, \omega_3, \omega_4 \in \zeroE.$
		
	We also define a map $ V_{g^{(2)}} : (\zeroE \tensorc \zeroE) \rightarrow (\zeroE \tensorc \zeroE)^*: = \Hom_{\mathbb{C}} (\zeroE \tensorc \zeroE, \mathbb{C}) $ by the formula
		$$ V_{g^{(2)}} (\omega_1 \tensorc \omega_2) (\omega_3 \tensorc \omega_4) = g^{(2)} ((\omega_1 \tensorc \omega_2) \tensorc (\omega_3 \tensorc \omega_4)).$$
\end{defn}
Since $ g (\omega_1 \tensora \omega_2) \in \mathbb{C} $ by the second assertion of Lemma \ref{14thfeb20191}, it is clear that the element 
	$ g^{(2)} ((\omega_1 \tensora \omega_2) \tensora (\omega_3 \tensorc \omega_4)) $ indeed belongs to $ \mathbb{C}. $ 
	
Let us note that the maps $ g^{(2)} $ and $ V_{g^{(2)}} $ are both right $\A$-linear. The following non-degeneracy property is going to be crucial in the sequel.
\begin{prop} \label{13thjune20192}
	Let $Y$ be an element of $\zeroE \tensorc \zeroE$. If $g^{(2)}(X \tensorc Y)= 0$ for all $X$ in $\zeroE \tensorc \zeroE$, then $Y=0$. Similarly, if $g^{(2)}(Y \tensorc X)=0$ for all $X$ in $\zeroE \tensorc \zeroE$, then $Y=0$. In particular, the map 
	$V_{g^{(2)}}$ defined in Definition \ref{27thaugust2019night1} is a vector space isomorphism from $ \zeroE \tensorc \zeroE $ to $ (\zeroE \tensorc \zeroE)^*.$ 
\end{prop}
\begin{proof}
	Let $\{ \omega_i \}_i$ be a basis for $\zeroE$ so that $\{ \omega_i \tensorc \omega_j \}_{ij}$ is a basis for $\zeroE \tensorc \zeroE$. By Proposition \ref{23rdmay20192}, the matrix whose $i,j$-th element is $g_{ij} = g(\omega_i \tensorc \omega_j)$ is invertible in $M_n(\IC)$. We will denote by $g^{ij}$ the $i,j$-th entry of the inverse of the matrix $((g_{ij}))_{ij}$. 
	
	Suppose $ \{ b_{ij} \}_{ij} $ are complex numbers $b_{ij}$ such that 
	$$Y = \sum_{ij} \omega_i \tensorc \omega_j b_{ij}.$$
	Let us fix the indices $i_0, j_0$ and define 
	$$X= \sum_{kl} g^{i_0l}g^{j_0k} \omega_k \tensorc \omega_l.$$ 
	Then we get
	\begin{align*}
		& 0 = g^{(2)}(X \tensorc Y) = g^{(2)}(\sum_{ijkl} g^{i_0l}g^{j_0k}(\omega_k \tensorc \omega_l) \tensorc (\omega_i \tensorc \omega_j)b_{ij})\\
	 = & \sum_{ijkl} g^{i_0 l} g^{j_0 k} g (\omega_k \tensora g_{li} \omega_j) b_{ij}		= \sum_{ijkl} g^{i_0l} g_{li} g^{j_0k} g_{kj} b_{ij}
			= \sum_{ij} \delta_{i_0i} \delta_{j_0j}b_{ij}
			= b_{i_0 j_0}.
	\end{align*}
	Hence, if $g^{(2)}(X \tensorc Y) = 0$ for all $X$, then $Y=0$.\\
	To prove the second statement, fix indices $i_0, j_0$ and define $X= \sum_{kl} g^{li_0} g^{kj_0} \omega_k \tensorc \omega_l$. Then, we compute the following.
	\begin{align*}
		g^{(2)}(Y \tensorc X) = & g^{(2)}(\sum_{ijkl} (\omega_i \tensorc \omega_j b_{ij}) \tensorc (\omega_k \tensorc \omega_l g^{li_0} g^{kj_0}))\\
			= & \sum_{ijkl} g_{il}g^{li_0} g_{jk}g^{kj_0} b_{ij}
			= \sum_{ij} \delta_{i_0i} \delta_{j_0j} b_{ij}
			= b_{i_0j_0}.
	\end{align*}
	Hence, if $g^{(2)}(Y \tensorc X) = 0$ for all $X$, then $Y=0$.
\end{proof}	
As a consequence of Proposition \ref{13thjune20192}, we are in a position to define the $ g^{(2)} $-adjoint of a complex linear map from $\zeroE \tensorc \zeroE$ to itself. Since the proof is elementary, we omit it.
\begin{prop} \label{27thmay20191}
	Suppose $h: V \tensorc V \to \IC$ be a linear map such that the following holds:
	$$h(v \tensorc w)= 0 \ \forall w \in V \textit{ implies } v = 0.$$
	Then, if $T$ is a $\IC$-linear map from $V$ to $V$,there exists a unique $\IC$-linear map $T^* : V \to V$ such that $$h(T^*(v) \tensorc w)= h(v \tensorc T(w)) \ \forall v, w \in V$$
\end{prop}

Now, ${}_0\sigma$ and ${}_0(\Psym)$ are linear maps from $\zeroE \tensorc \zeroE$ to itself. By virtue of Proposition \ref{13thjune20192}, we can apply Proposition \ref{27thmay20191} to $h=g^{(2)}$ and $T={}_0\sigma$ or ${}_0(\Psym)$. Thus, $({}_0\sigma)^*$ and $({}_0(\Psym))^*$ exist.

\begin{lemma} \label{14thjune20191}
	Let $\E$ be a bicovariant $\A$-$\A$-bimodule, $ \sigma $ the braiding map of Proposition \ref{4thmay20193} and $g$ be a bi-invariant pseudo-Riemannian metric on $\E$, then $ ({}_0\sigma)^* = {}_0\sigma. $ 
\end{lemma}
\begin{proof}
	We will actually prove a stronger statement. Since $g^{(2)}$ is a map from $({}_0 \E \tensorc {}_0 \E) \tensorc ({}_0 \E \tensorc {}_0 \E)$ to $\IC$, it extends uniquely to a right $\A$-linear left-covariant map (to be denoted by $g^{(2)}$ again, by an abuse of notation) from $(\E \tensora \E) \tensora (\E \tensora \E)$ to $\A$ by Proposition \ref{inviff}. We will prove that for $e,f,e^\prime, f^\prime$ in $\E$,
	\begin{equation} \label{14thjune20192}
		g^{(2)}(\sigma(e \tensora f)\tensora (e^\prime \tensora f^\prime))=g^{(2)}((e \tensora f)\tensora \sigma(e^\prime \tensora f^\prime))
	\end{equation}
	To this end, we claim that it is enough to prove that for all $\omega, \omega^{\prime}$ in $\zeroE$ and $\eta, \eta^{\prime}$ in $\Ezero$,
	\begin{equation}
	g^{(2)}(\sigma(\omega \tensora \eta) \tensora (\omega^{\prime} \tensora \eta^{\prime}))=g^{(2)}((\omega \tensora \eta) \tensora \sigma(\omega^{\prime} \tensora \eta^{\prime})). \label{13thfeb20192}
	\end{equation}
	Indeed, by Corollary \ref{3rdaugust20191}, for every element $a$ in $\A$, there exists elements $ x_i \in \zeroE, y_i \in \E_0 $ and $ a_i \in \A $ such that 
	$$a(\omega^\prime \tensora \eta^\prime)= \sum_{i} x_i \tensora y_i a_{i}.$$
	Hence, if \eqref{13thfeb20192} is true, the right $\A$-linearity of the map $g^{(2)}$ implies that
	\begin{align*}
		&g^{(2)}(\sigma(\omega \tensora \eta a) \tensora (\omega^\prime \tensora \eta^\prime b)) = g^{(2)}(\sigma(\omega \tensora \eta) \tensora a(\omega^{\prime} \tensora \eta^{\prime})) b \\
	=&	\sum_{i} g^{(2)}(\sigma(\omega \tensora \eta) \tensora (x_i \tensora y_i))a_{i}b = \sum_{i} g^{(2)}((\omega \tensora \eta) \tensora \sigma(x_i \tensora y_i))a_{i}b \\
		=& \sum_{i} g^{(2)}((\omega \tensora \eta) \tensora \sigma(x_i \tensora y_i a_{i}))b = g^{(2)}((\omega \tensora \eta) \tensora a \sigma (\omega^\prime \tensora \eta^\prime))b \\
		&= g^{(2)}((\omega \tensora \eta a) \tensora \sigma (\omega^\prime \tensora \eta^\prime b)).
	\end{align*}
	Here we have used the bilinearity of the map $\sigma$. Since $\zeroE \tensora \Ezero$ is right $\A$-total in $\E \tensora \E$ (by Corollary \ref{3rdaugust20191}), this proves \eqref{14thjune20192} provided we prove \eqref{13thfeb20192}. This proves our claim.
	
	Thus, we are left with proving \eqref{13thfeb20192} which follows from the following computation:
	\begin{align*}
		&g^{(2)}(\sigma(\omega \tensora \eta) \tensora (\omega^\prime \tensora \eta^\prime))=g^{(2)}((\eta \tensora \omega) \tensora (\omega^\prime \tensora \eta^\prime))\\
		=&g(\eta \tensora \eta^\prime) g (\omega \tensora \omega^\prime)
		=g^{(2)}((\omega \tensora \eta) \tensora (\eta^\prime \tensora \omega^\prime))\\
		=&g^{(2)}((\omega \tensora \eta) \tensora \sigma(\omega^\prime \tensora \eta^\prime)),
	\end{align*}
	where we have used \eqref{30thapril20191} twice and the facts that $g(\omega \tensora \omega^\prime)$ and $g(\eta \tensora \eta^\prime)$ take values in $\IC.1$ (second assertion of Lemma \ref{14thfeb20191}). This completes the proof of the lemma.
\end{proof}

\begin{prop} \label{15thjune20195}
	We have $({}_0(\Psym))^*={}_0(\Psym)$. If	 $V_{g^{(2)}}: \zeroE \tensorc \zeroE \to (\zeroE \tensorc \zeroE)^*$ is the map defined in Definition \ref{27thaugust2019night1}, then
	\begin{equation} \label{27thaugust20194} V_{g^{(2)}}({}_0(\Psym)(X))(Y) = V_{g^{(2)}}(X) \circ {}_0(\Psym)(Y) \ \forall X,Y \in \zeroE \tensorc \zeroE. \end{equation} 
	In particular, $V_{g^{(2)}}$ is a vector space isomorphism from $\zeroE \tensorc^{\rm sym} \zeroE$ onto $(\zeroE \tensorc^{\rm sym} \zeroE)^*$.
\end{prop}
\begin{proof}
Since $({}_0\sigma)^*={}_0\sigma$ by Lemma \ref{14thjune20191} and $ {}_0 (\Psym) $ is a polynomial in ${}_0 \sigma $ by \eqref{24thaugust20191}, we have $({}_0(\Psym))^*={}_0(\Psym)$. Then \eqref{27thaugust20194} follows from the definition of $V_{g^{(2)}}$.

 Finally, for the last assertion, let us recall the identification 
\begin{equation} \label{15thnov20193}	(\zeroE \tensorc^{\rm sym} \zeroE)^* = \{ \phi \in (\zeroE \tensorc \zeroE)^* : \phi(X) = \phi({}_0(\Psym)(X)) \ \forall X \in \zeroE \tensorc \zeroE \} \end{equation} 
	from \eqref{23rdaugust20193}. Now, if $X$ is in $\zeroE \tensorc^{\rm sym} \zeroE = {\rm Ran} ({}_0(\Psym)), $ then for all $Y$ in $\zeroE \tensorc \zeroE$, we have \begin{equation*} V_{g^{(2)}}(X)(Y) = V_{g^{(2)}}({}_0(\Psym)(X))(Y) = V_{g^{(2)}}(X)({}_0(\Psym)(Y)). \end{equation*}
	by \eqref{27thaugust20194}. Therefore, $V_{g^{(2)}}(\zeroE \tensorc^{\rm sym} \zeroE)$ is a subspace of $(\zeroE \tensorc^{\rm sym} \zeroE)^*$ by \eqref{15thnov20193}. Now by Proposition \ref{13thjune20192}, the map $V_{g^{(2)}}$ is one one and so we reach our
 our desired conclusion by a dimension argument.
\end{proof}

\section{Connections on a bicovariant differential calculus and their torsion} \label{section5}
In this section, we prove the existence of a bicovariant torsionless connection on any bicovariant differential calculus which satisfies the condition that ${}_0 \sigma$ is diagonalisable. We will be working with right connections as opposed to left connections. We start by recalling the definition of right connections and their torsion for which we will need the space of two-forms $\twoform$ defined in Definition \ref{25thmay20194}. 
\begin{defn} (\cite{heckenberger})
	Let $ (\E, d)$ be a bicovariant differential calculus on $\A.$ A (right) connection on $\E$ is a $\IC$-linear map $\nabla: \E \to \E \tensora \oneform$ such that, for all $a$ in $\A$ and $\rho$ in $\E$, the following equation holds:
	$$ \nabla(\rho a)= \nabla(\rho)a + \rho \tensora da. $$
	The torsion of a connection $\nabla$ on $\E$ is the right $\A$-linear map 
	$$T_\nabla:= \wedge \circ \nabla + d :\E \to \twoform.$$
	$\nabla$ is said to be torsionless if $T_{\nabla}=0$.
\end{defn}
Our notion of torsion is the same as that of \cite{heckenberger}, with the only difference being that they work with left connections.

 We say that $ \nabla $ is left (respectively, right)-covariant if $ \nabla $ is a left (respectively, right)-covariant map between the bicovariant bimodules $ \E $ and $ \E \tensora \E. $
\begin{lemma} (\cite{heckenberger}) \label{nablaleftinv}
	If $\nabla$ is a left-covariant connection on a bicovariant differential calculus $ (\E, d),$ then $\nabla(\zeroE) \subseteq \zeroE \tensorc \zeroE$.
\end{lemma}
\begin{proof}
	This follows by combining Proposition \ref{14thfeb20192} and Proposition \ref{3rdaugust20192}.
\end{proof}	
Now we state and prove the main result of this section which requires the diagonalisability of the map $ {}_0 \sigma. $ Indeed, we will be using the map $ Q = \wedge|_{\F} : \F \rightarrow \twoform $ (Definition \ref{5thseptember20191sm}) which makes sense due to the splitting $ \E \tensora \E = (\E \tensorsym \E) \oplus \F $ (Theorem \ref{25thmay20195}) which in turn follows from the assumption of diagonalisability of the map $ {}_0 \sigma. $ Let us recall that $Q$ is a bimodule isomorphism from $\F$ to $ \twoform. $

\begin{thm} \label{20thaugust20192}
	Suppose $(\E,d)$ is a bicovariant differential calculus on $\A$ such that ${}_0 \sigma$ is diagonalisable. Then $\E$ admits a bicovariant torsionless connection.
\end{thm}
\begin{proof}
The proof of existence of a torsionless connection $ \nabla_0 $ follows exactly along the lines of Theorem 2.13 of \cite{article1}. The only difference here is that we need to define $ \nabla_0 $ in such a way that it remains bicovariant.

	We define $\widetilde{\nabla}_0: \zeroE \to \zeroE \tensorc \zeroE$ by 
	$$\widetilde{\nabla}_0(\omega) = Q^{-1} (-d(\omega)).$$
	Indeed, by Corollary \ref{25thmay20193} and \eqref{25thmay20196}, $\widetilde{\nabla}_0(\omega) $ is an element of $ \zeroE \tensorc \zeroE$ for all $\omega$ in $\zeroE.$ Let $\{ \omega_i \}_i$ be a vector space basis of $\zeroE$. By the right $\A$-totality of $\zeroE$ in $\E$, we extend $\widetilde{\nabla}_0$ to a map $\nabla_0 : \E \to \E \tensora \E$ by the formula
	$$ \nabla_0 (\sum_i \omega_i a_i) = \sum_i \widetilde{\nabla}_0(\omega_i) a_i + \sum_i \omega_i \tensora da_i .$$
	Since $\E$ is a free module with basis $ \{ \omega_i \}_i, $ the above formula is well-defined. 
	It follows that for all $\omega$ in $\zeroE$ and for all $a$ in $\A$, 
	$$\nabla_0(\omega a) = \widetilde{\nabla}_0(\omega)a + \omega \tensora da.$$
	From this equation and the right $\A$-totality of $\zeroE$ in $\E,$ it can be easily checked that $ \nabla_0 $ is a connection.
	Now we prove that $\nabla_0$ is torsionless. Indeed, since by Definition \ref{5thseptember20191sm}, we have $ \wedge \circ Q^{-1} = \id_{\Omega^2 (\A)},$ we can deduce that
	\begin{align*}
		\wedge \circ \nabla_0(\omega a) = & \wedge \circ (\widetilde{\nabla}_0(\omega)a + \omega \tensora da)
		= \wedge \circ Q^{-1} (-d(\omega)) a + \omega \wedge da\\
		= & -d(\omega)a + \omega \wedge da
		= -d(\omega a).
	\end{align*}
	Before proceeding further, let us note that since $ \nabla_0 $ coincides with $ \widetilde{\nabla}_0 $ on $ \zeroE $ and $ \widetilde{\nabla}_0 (\omega) $ belongs to $ \zeroE \tensorc \zeroE $ if $ \omega \in \zeroE, $ $ \nabla_0 (\omega)$ is in $\zeroE \tensorc \zeroE. $ We will use this fact in the rest of the proof where $ \omega $ and $a$ will stand for arbitrary elements of $\zeroE$ and $\A$ respectively. 
	
	To show that $\nabla_0$ is left-covariant, we observe that since $ \nabla_0 (\omega) \in \zeroE \tensorc \zeroE, $ $ \Delta_{\E \tensora \E} (\nabla_0 (\omega)) = 1 \tensorc \nabla_0 (\omega).$ Using this, we get
	\begin{align*}
		 &(\id \tensorc \nabla_0)(\Delta_\E(\omega a))
		= (\id \tensorc \nabla_0)(\Delta_\E(\omega)\Delta(a))\\
		= & (\id \tensorc \nabla_0)((1 \tensorc \omega)(a_{(1)} \tensorc a_{(2)}))
		= a_{(1)} \tensorc \nabla_0(\omega a_{(2)})\\
		= & a_{(1)} \tensorc (\nabla_0(\omega) a_{(2)} + \omega \tensora d a_{(2)})
		= (1 \tensorc \nabla_0(\omega))(a_{(1)} \tensorc a_{(2)}) + a_{(1)} \tensorc \omega \tensora da_{(2)}\\
		= & (1 \tensorc \nabla_0(\omega))(a_{(1)} \tensorc a_{(2)}) + (da)_{(-1)} \tensorc \omega \tensora (da)_{(0)} \ \big({\rm by \ part \ (i) \ of \ Lemma \ \ref{6thjune20192} } \big)\\
		= & \Delta_{\E \tensora \E}(\nabla_0(\omega))\Delta(a) + \Delta_{\E \tensora \E}(\omega \tensora da)
		= \Delta_{\E \tensora \E}(\nabla_0(\omega)a + \omega \tensora da)\\
		= & \Delta_{\E \tensora \E}(\nabla_0(\omega a)).
	\end{align*}
	 Finally, we show that $\nabla_0$ is also right-covariant. Since $\E$ is a bicovariant bimodule, $ {}_\E \Delta (\omega) = \omega_{(0)} \tensorc \omegaone $ belongs to $ \zeroE \tensorc \E $ by Theorem 2.4 of \cite{woronowicz}. Hence $ \omega_{(0)} $ belongs to $ \zeroE $ and we are allowed to write
		$$ \nabla_0 (\omega_{(0)} \aone) = Q^{-1} (- d (\omega_{(0)})) \aone + \omega_{(0)} \tensorc d (\aone).$$
		Thus, we obtain
	\begin{equation*}
	\begin{aligned}
		 & (\nabla_0 \tensorc \id){}_{\E}\Delta(\omega a) = (\nabla_0 \tensorc \id)(\omega_{(0)} a_{(1)} \tensorc \omega_{(1)} a_{(2)})\\
		=& \nabla_0 (\omega_{(0)} a_{(1)}) \tensorc \omega_{(1)} a_{(2)} = (Q^{-1} (-d(\omega_{(0)}))a_{(1)} + \omega_{(0)} \tensora d (a_{(1)})) \tensorc \omega_{(1)} a_{(2)}\\
		=& (Q^{-1} \tensorc \id)\big(((-d) \tensorc \id)(\omega_{(0)} \tensorc \omega_{(1)})\big)(a_{(1)} \tensorc a_{(2)}) + \omega_{(0)} \tensora d (a_{(1)}) \tensorc \omega_{(1)} a_{(2)}\\
		=&(Q^{-1} \tensorc \id)\big(((-d) \tensorc \id)(\omega_{(0)} \tensorc \omega_{(1)})\big)(a_{(1)} \tensorc a_{(2)}) + \omega_{(0)} \tensora (da)_{(0)} \tensorc \omega_{(1)} (da)_{(1)}\\
		& \big({\rm by \ Part \ (ii) \ of \ Lemma \ \ref{6thjune20192} } \big)\\
		=&(Q^{-1} \tensorc \id)\big(((-d) \tensorc \id)({}_{\E \tensora \E}\Delta(\omega))\big)(\Delta(a)) + {}_{\E \tensora \E}\Delta(\omega \tensora da)\\
		=& (Q^{-1} \tensorc \id) \big({}_{\twoform} \Delta(-d(\omega))\big)\Delta(a) + {}_{\E \tensora \E}\Delta(\omega \tensora da)\\
		& \big({\rm since \ d \ is \ a \ bicovariant \ map \ from \ \E \ to \ \Omega^2 (\A) \ by \ Proposition \ \ref{25thmay20192} \ } \big)\\
		=& {}_{\E \tensora \E}\Delta(Q^{-1}(-d(\omega)))\Delta(a) + {}_{\E \tensora \E}\Delta(\omega \tensora da) \\
		& \big({\rm since \ Q \ is \ right \ covariant \ by \ Corollary \ \ref{25thmay20193} \ } \big)\\
		=& {}_{\E \tensora \E}\Delta(\nabla_0(\omega))\Delta(a) + {}_{\E \tensora \E}\Delta(\omega \tensora da)\\
		=& {}_{\E \tensora \E}\Delta(\nabla_0(\omega)a + \omega \tensora da)\\
		=& {}_{\E \tensora \E}\Delta(\nabla_0(\omega a)).
	\end{aligned}	
	\end{equation*}
	This finishes the proof.
\end{proof}

We end this section by proving the following result which will be needed in the proof of Proposition \ref{Phigiso}.
\begin{lemma} \label{15feb20191}
	If $\nabla_1$ and $\nabla_2$ are two left-covariant torsionless connections on a bicovariant differential calculus $ (\E, d)$ on $\A,$ then $\nabla_1 - \nabla_2$ is an element of $\homc(\zeroE, \zeroE \tensorc^{\rm sym} \zeroE)$.
\end{lemma}
\begin{proof}
	If $\nabla_1$ and $\nabla_2$ are two torsionless connections, we have that $ \wedge \circ \nabla_1 = - d = \wedge \circ \nabla_2$. Therefore, 
	$$\rm Ran(\nabla_1 -\nabla_2) \subseteq {\rm Ker} (\wedge) = \E \tensorsym \E.$$
	Moreover, by Lemma \ref{nablaleftinv}, if $\omega$ is an element of $\zeroE$, then $(\nabla_1 -\nabla_2)(\omega)$ is in $\zeroE \tensorc \zeroE,$ i.e, $ (\nabla_1 - \nabla_2) (\omega) $ is invariant under $ \Delta_{\E \tensora \E}. $ Hence, by \eqref{22ndaugust20194}, $(\nabla_1 -\nabla_2)(\omega)$ is an element of $ {}_0(\E \tensorsym \E) = \zeroE \tensorc^{\rm sym} \zeroE$.
\end{proof}

\section{Metric Compatibility of a bicovariant connection} \label{section4}

In this section, we define the notion of metric-compatibility of a left-covariant connection with a left-invariant pseudo-Riemannian metric. We will need the map $ {}_0 (\Psym) $ introduced in Definition \ref{27thmay20193}. Our definition coincides with that in the classical case (Proposition \ref{20thnov20191}) and also with that in \cite{heckenberger} for cocycle deformations of classical Lie groups. The latter statement is derived at the end of Section \ref{sectioncocyclelc}.
\begin{defn}
	Let $\nabla$ be a left-covariant connection on a bicovariant calculus $(\E, d)$ and $ g $ a left-invariant pseudo-Riemannian metric. Then we define 
	$$\widetilde{\Pi^0_g}(\nabla):{}_0\E \tensorc {}_0\E \rightarrow {}_0\E ~ {\rm by} ~ {\rm the} ~ {\rm following} ~ {\rm formula}:$$
\begin{equation} \label{1stmay20191}	\widetilde{\Pi^0_g}(\nabla)(\omega_i \tensorc \omega_j)=2(\id \tensorc g)(\sigma \tensorc \id)(\nabla \tensorc \id){}_0(\Psym)(\omega_i \tensorc \omega_j). \end{equation}
	Next, for all $ \omega_1,\omega_2$ in $\zeroE $ and $a$ in $\A$, we define $\widetilde{\Pi_g}(\nabla):\E \tensora \E \rightarrow \E$ by 
	$$\widetilde{\Pi_g}(\nabla)\circ {\widetilde{u}}^{\E \tensora \E}(\omega_1 \tensorc \omega_2 \tensorc a)= \widetilde{\Pi^0_g}(\nabla)(\omega_1 \tensorc \omega_2)a + g(\omega_1 \tensora \omega_2)da.$$
\end{defn}
It is easy to see that $ \widetilde{\Pi^0_g} (\nabla) $ indeed maps $ \zeroE \tensorc \zeroE $ to $\zeroE.$ Indeed, let $\omega_1, \omega_2$ be elements of $\zeroE$. Since ${}_0(\Psym)$ is a map from $\zeroE \tensorc \zeroE$ to itself, $ {}_0(\Psym) (\omega_1 \tensorc \omega_2)$ is in $\zeroE \tensorc \zeroE. $ Then, by Lemma \ref{nablaleftinv}, $(\nabla \tensorc \id) ({}_0(\Psym))(\omega_1 \tensorc \omega_2)$ is in $\zeroE \tensorc \zeroE \tensorc \zeroE$. Since $\sigma$ is left-covariant and $g$ is left-invariant, Proposition \ref{14thfeb20192} and the second assertion of Lemma \ref{14thfeb20191} imply that the element $(\id \tensorc g)(\sigma \tensorc \id)(\nabla \tensorc \id) ({}_0(\Psym))(\omega_1 \tensorc \omega_2)$ belongs to $\zeroE .$

Finally, by Proposition \ref{moduleiso} and the notation adopted in Proposition \ref{3rdaugust20192} $ \widetilde{u}^{\E \tensora \E}: \zeroE \tensorc \zeroE \tensorc \A \to \E \tensora \E $ is an isomorphism, hence $\widetilde{\Pi_g}$ is well-defined.
\begin{rmk} \label{23rdnov20194}
If $ \nabla $ is left-covariant and $g$ is left-invariant, the above argument shows that $ \widetilde{\Pi_g} (\nabla) (\zeroE \tensorc \zeroE) \subseteq \zeroE$ and thus by Proposition \ref{14thfeb20192}, the map $ \widetilde{\Pi_g} (\nabla) $ is left-covariant. 
\end{rmk}
For the rest of the article, $ dg $ will denote the map
$$ dg: \E \tensora \E \rightarrow \E, \ dg (e \tensora f) = d (g (e \tensora f)). $$
Now we define the notion of metric compatibility of a bicovariant connection. 
\begin{defn} \label{metriccompatibility}
	Suppose $(\E, d)$ is a left-covariant differential calculus over $\A$ and $g$ is a left-invariant pseudo-Riemannian metric. We say that a left-covariant connection $\nabla$ on $\E$ is compatible with $g$ if, as maps from $ \E \tensora \E $ to $\E$,
	\[ \widetilde{\Pi_g}(\nabla)=dg .\]
\end{defn}
We now show that our formulation of metric-compatibility of a connection coincides with that in the classical case of commutative Hopf algebras. 
\begin{prop} \label{20thnov20191}
The above definition of metric compatibility coincides with that in the classical case.
\end{prop}
\begin{proof} Let $G$ be a linear algebraic group, $\A$ be its (commutative) Hopf algebra of regular functions and $g$ be a left-invariant pseudo-Riemannian metric on the classical space of forms. In this case, the canonical braiding map $ \sigma $ is equal to the flip map $ \sigmacan, $ i.e, for all $e, f $ in $\Omega^1(\A),$
\[ \sigma (e \tensora f) = \sigmacan (e \tensora f) = f \tensora e. \]

Since $ g \circ \sigma = g, $ we have $ g (e \tensora f) = g (f \tensora e).$ Moreover, the map $ \Psym $ is equal to $ \frac{1 }{2}(1 + \sigma). $ 
Let us recall (\cite{article1}) that a connection $\nabla$ on $\Omega^1(\A)$ is compatible with $g$ if and only if
$$g_{13}\big(\nabla(e)\otimes_{\A} e^\prime + \nabla(e^\prime)\otimes_{\A} e \big)=dg(e \otimes_{\A} e^\prime),$$
for all $e, e^\prime$ in $\Omega^1(\A)$, where $g_{13}=(\id \otimes_{\A} g)(\sigma^{{\rm can}} \otimes_{\A} \id) $.\\
Let $\{ e_i \}_i$ be a basis of left-invariant one-forms of $\Omega^1(\A)$. If $e, e^\prime$ belong to $\Omega^1(\A),$ then there exist elements $a_i, b_j$ in $\A$ such that $ e = \sum_i e_i a_i$ and $ e^\prime = \sum_j e_j b_j.$ If $\nabla$ is metric compatible in the sense of Definition \ref{metriccompatibility}, i.e, $\widetilde{\Pi_g}(\nabla) - dg = 0$, then using the Leibniz properties of $ \nabla $ and $d$ and the equation $ g (e_i \tensora e_j) = g (e_j \tensora e_i), $ it is easy to check that
	\begin{align*} &g_{13}(\nabla(e)\otimes_{\A} e^\prime +\nabla(e^\prime) \otimes_{\A} e) = \widetilde{\Pi_g}(\nabla)(\sum_{ij} e_i \tensora e_j a_i b_j)\\ =& dg(\sum_{ij} e_i \tensora e_j a_i b_j) = dg(e \tensora e^\prime). \end{align*}

This argument is reversible and thus, our definition of metric compatibility coincides with that in the classical case.
\end{proof}
It is also true that our definition of metric compatibility coincides with that of \cite{heckenberger} for cocycle deformations of classical Lie groups. We state this result at the end of this section (Proposition \ref{25thnov20191}). 

In Remark \ref{28thnov20191}, we mentioned that the definition of two-forms considered in \cite{heckenberger} is in general different from the definition taken in this article. If we work with the two-forms of \cite{heckenberger}, our definition of metric-compatibility will need to be suitably modified. This will be taken up elsewhere.  

\subsection{Covariance properties of the map $ \widetilde{\Pi_g} $} \label{4thdec20191}

In this subsection, we derive some covariance properties of the maps $ \widetilde{\Pi_g},$ $\widetilde{\Pi_g} (\nabla) - dg $ which will be used in Subsection \ref{26thseptember20192} and Subsection \ref{24thnov20191}.

\begin{lemma} \label{21staugust20192}
	If $\nabla$ is a bicovariant connection on $\E$ and $g$ is a bi-invariant pseudo-Riemannian metric, then $\widetilde{\Pi_g}$ is a right-covariant map.
\end{lemma}
\begin{proof}
 The maps $ \sigma $ and $ {}_0(\Psym) $ are bicovariant (Proposition \ref{16thseptember20191}). Therefore, if $\nabla$ is also right-covariant, and $g$ is bi-invariant (and hence by the first assertion of Lemma \ref{14thfeb20191} also bicovariant), then $\widetilde{\Pi_g}$ is a composition of right-covariant maps and therefore, right-covariant. 
\end{proof}
\begin{prop} \label{proppignabla}
If the connection $\nabla$ is left-covariant and the pseudo-Riemannian metric $g$ is left-invariant, then the map $ \widetilde{\Pi_g} (\nabla) - dg: \E \tensora \E \rightarrow \E $ is a left-covariant right $\A$-linear map. Moreover, if $\nabla$ is bicovariant and the pseudo-Riemannian metric $g$ is bi-invariant, then $\widetilde{\Pi_g} (\nabla) - dg$ is also a bicovariant map.
\end{prop}
\begin{proof} We start by proving that $\widetilde{\Pi_g} - dg$ is a right $\A$-linear. Since $ \{ \omega \tensora \omega^\prime: \omega, \omega^\prime \in \zeroE \} $ is right $\A$-total in $ \E \tensora \E, $ it suffices to show that for all $ \omega_1, \omega_2 \in \zeroE $ and $a, b \in \A, $ we have: 
$$(\widetilde{\Pi_g} (\nabla) - dg)((\omega_1\tensora \omega_2 a)b) =\big((\widetilde{\Pi_g} (\nabla) - dg)(\omega_1\tensora \omega_2 a)\big)b.$$
This equation can be checked easily and we omit the proof.
Now, we prove that $\widetilde{\Pi_g} - dg$ is a left-covariant map. Since $g$ is left-invariant, for any $\omega_1, \omega_2$ in $\zeroE$, $g(\omega_1 \tensora \omega_2) \in \IC$ by the second assertion of Lemma \ref{14thfeb20191}, and so $dg(\omega_1 \tensora \omega_2)=0$. Hence, 
$$(\widetilde{\Pi_g}(\nabla)-dg)(\omega_1 \tensora \omega_2)= \widetilde{\Pi^0_{g}}(\omega_1 \tensorc \omega_2),$$
 which is in $\zeroE$. Therefore, by Proposition \ref{14thfeb20192} , the map $\widetilde{\Pi_g}(\nabla)-dg$ is a left-covariant map.\\
Finally, if $\nabla$ is bicovariant and $g$ is bi-invariant, then by Lemma \ref{21staugust20192}, $\widetilde{\Pi^0_{g}}$ is a right-covariant map. Moreover, $g$ and $d$~are bicovariant (first assertion of Lemma \ref{14thfeb20191} and Proposition \ref{cdc}). Hence $\widetilde{\Pi_g}(\nabla)-dg$ is also a bicovariant map.
\end{proof}

\begin{corr} \label{remark4.29}
Suppose $ \nabla $ is a bicovariant connection and $g$ is a bi-invariant pseudo-Riemannian metric on $(\E, d).$ Then the map $ \widetilde{\Pi_g} (\nabla) - dg $ is a right-covariant $\mathbb{C}$-linear map from $ \zeroE \otimes^{{\rm sym}}_{\mathbb{C}} \zeroE $ to $ \zeroE. $ 
\end{corr}
\begin{proof}
Since $ \zeroE \otimes^{{\rm sym}}_{\mathbb{C}} \zeroE \subseteq \zeroE \tensorc \zeroE $ and $ g (\zeroE \tensorc \zeroE) \in \mathbb{C}.1 $ (second assertion of Lemma \ref{14thfeb20191}), the map $ dg $ is equal to zero on $ \zeroE \otimes^{{\rm sym}}_{\mathbb{C}} \zeroE.$ Hence, 
$$ \widetilde{\Pi_g} (\nabla) - dg = \widetilde{\Pi_g} (\nabla) = \widetilde{\Pi^0_g} (\nabla) ~ {\rm on} ~ \zeroE \otimes^{{\rm sym}}_{\mathbb{C}} \zeroE \subseteq \zeroE \tensorc \zeroE.$$ 
However, as noted before, $ \widetilde{\Pi_g} (\nabla) (\zeroE \tensorc \zeroE) \subseteq \zeroE.$ The right-covariance follows from Proposition \ref{proppignabla}.
\end{proof}

The following result is an immediate corollary of the proof of Proposition \ref{proppignabla} and Definition \ref{metriccompatibility}.
\begin{corr} \label{28thoct20191sm}
	A connection $\nabla$ on a bicovariant calculus $(\E, d)$ is compatible with a bi-invariant pseudo-Riemannian metric $g$ if and only if $\widetilde{\Pi^0_{g}}(\nabla) = 0$ as a map on ${}_0 \E \tensorc {}_0 \E$.
\end{corr}
Let us remark that in Lemma 3.4 of \cite{heckenberger}, Heckenberger and Schm\"udgen prove an exact analogue of Corollary \ref{28thoct20191sm} for their formulation of metric compatibility. 

We end this section by comparing our notion of metric-compatibility with that of Heckenberger and Schm\"udgen (\cite{heckenberger}). Before we state our result, let us recall that a left connection on $\E$ is a $\IC$-linear map $\nabla: \E \rightarrow \E \tensora \E$ such that $ \nabla (a e) = a \nabla (e) + da \tensora e. $ Similarly, a left $\A$-linear pseudo-Riemannian metric on $\E$ is a left $\A$-linear map $ g: \E \tensora \E \rightarrow \A $ such that $ g \circ \sigma = g $ satisfying the condition that if $ g (e \tensora f) = 0 $ for all $e$ in $\E,$ then $f = 0.$

Suppose $ (\E, d) $ is a bicovariant differential calculus and $g$ a left $\A$-linear bi-invariant pseudo-Riemannian metric on $\E.$ The authors of \cite{heckenberger} call a left connection $\nabla$ on $\E$ to be compatible with $g$ if 
$$ (\id \tensorc g) (\nabla \tensorc \id) + (g \tensorc \id) (\id \tensorc \sigma) (\id \tensorc \nabla) = 0 ~ {\rm on} ~ \zeroE \tensorc \zeroE. $$
Therefore, we need to define the analogue of our compatibility for a bicovariant left connection $\nabla$ with respect to a left $\A$-linear bi-invariant pseudo-Riemannian metric $g$ in order to compare our definition with that in \cite{heckenberger}. To this end, we define a map
$$ \widetilde{{}_L \Pi^0_g} (\nabla):= 2(g \tensorc \id) (\id \tensorc \sigma) (\id \tensorc \nabla) {}_0(\Psym) : \zeroE \tensorc \zeroE \rightarrow \zeroE. $$
Then as before, we define an extension $ \widetilde{{}_L \Pi_g} (\nabla): \E \tensora \E \rightarrow \E $ by 
$$ \widetilde{{}_L \Pi_g} (\nabla) {\widetilde{v}}^{\E \tensora \E} (a \tensorc \omega_1 \tensorc \omega_2) = a \widetilde{{}_L \Pi^0_g} (\nabla) (\omega_1 \tensorc \omega_2) + (da) g (\omega_1 \tensorc \omega_2) ,$$
where $ {\widetilde{v}}^{\E \tensora \E}: \A \tensorc \zeroE \tensorc \zeroE \rightarrow \E \tensora \E $ is the multiplication map which we know is an isomorphism from Proposition \ref{moduleiso} and Corollary \ref{3rdaugust20191}. We say that the bicovariant left connection $\nabla$ is compatible with the left $\A$-linear bi-invariant pseudo-Riemannian metric $g$ if 
\begin{equation} \label{25thnov20192} \widetilde{{}_L \Pi_g} (\nabla) = dg. \end{equation}
It is easy to check that this definition coincides with the definition of metric-compatibility in the classical case. Then a result analogous to Corollary \ref{28thoct20191sm} can be derived to deduce that
\begin{equation} \label{25thnov20193} \widetilde{{}_L \Pi_g} (\nabla) = dg ~ {\rm if} ~ {\rm and} ~ {\rm only} ~ {\rm if} ~ \widetilde{{}_L \Pi^0_g} (\nabla) = 0. \end{equation} 
The next result compares the above two definitions of metric-compatibility. However, since this result needs the definitions and some results on cocycle deformations, we have proved this at the end of Section \ref{sectioncocyclelc}. 
\bppsn \label{25thnov20191}
Let $\A$ be the Hopf algebra of regular functions on a linear algebraic group, $ (\E, d) $ be the classical bicovariant differential calculus on $\A$ and $\Omega$ a normalised dual $2$-cocycle on $\A.$ Consider the bicovariant differential calculus $ (\E_\Omega, d_\Omega) $ over the Hopf algebra $ \A_\Omega $ (see Proposition \ref{11thmay20192}) and let $g^\prime$ be a left $\A$-linear bi-invariant pseudo-Riemannian metric on $\E_\Omega.$ 

A bicovariant left connection $\nabla^\prime$ on $ \E_\Omega $ is compatible with $g^\prime$ in the sense of \eqref{25thnov20192} if and only if $\nabla$ is compatible with $g^\prime$ in the sense of \cite{heckenberger}.
\eppsn

\section{Sufficient conditions for the existence of Levi-Civita connections} \label{21staugust20197} 

In this section, we will derive some sufficient conditions for the existence of Levi-Civita connections for bicovariant differential calculus on quantum groups. As before, unless otherwise mentioned, $(\E,d)$ will denote a bicovariant differential calculus on $\A$ such that the restricted braiding map ${}_0\sigma$ is diagonalisable, and $g$ a bi-invariant pseudo-Riemannian metric on $\E$.
\begin{defn}
Let $ (\E, d) $ be a bicovariant differential calculus such that the map $ {}_0 \sigma $ is diagonalisable and $ g $ a pseudo-Riemannian bi-invariant metric on $ \E. $ A left-covariant connection $ \nabla $ on $ \E $ is called a Levi-Civita connection for the triple $ (\E, d, g) $ if it is torsionless and compatible with $g.$ 
\end{defn}
The strategy to derive our results are the same as in \cite{article1} and \cite{article2}. However, since we are not working with a centered bimodule and the pseudo-Riemannian metric is only right $\A$-linear, the arguments become more delicate. 
 Given a bicovariant differential calculus $ (\E, d) $ and a bi-invariant pseudo-Riemannian metric $g,$ we start by defining a map 
 $$\widetilde{\Phi_g} : \homc(\zeroE, \zeroE \tensorc^{\rm sym} \zeroE) \rightarrow \homc(\zeroE \tensorc^{\rm sym} \zeroE, \zeroE)$$ and show (Proposition \ref{Phigiso}) that the isomorphism of $ \widetilde{\Phi_g} $ guarantees the existence of a unique left-covariant Levi-Civita connection for the triple $(\E, d, g).$ 

However, since our metric is bi-invariant, it is to be expected that our Levi-Civita connection should be bicovariant. This is the second main result of this section (Theorem \ref{17thseptember20191sm}) which requires the Hopf algebra $\A$ to be cosemisimple. We remark that the bicovariance of the Levi-Civita connection (with respect to a different metric-compatibility condition) for $ SL_q (n), Sp_q (n) $ and $ O_q (n) $ were derived in \cite{heckenberger}.

Finally, our third result is Theorem \ref{15thjune20193} where we prove that the map $\widetilde{\Phi_g} : \homc(\zeroE, \zeroE \tensorc^{\rm sym} \zeroE) \rightarrow \homc(\zeroE \tensorc^{\rm sym} \zeroE, \zeroE)$ is an isomorphism if and only if the map 
$$ ({}_0(\Psym))_{23}: (\zeroE \tensorc^{{\rm sym}} \zeroE) \tensorc \zeroE \rightarrow \zeroE \tensorc (\zeroE \tensorc^{{\rm sym}} \zeroE) $$
is an isomorphism. The proofs of Theorem \ref{17thseptember20191sm} and Theorem \ref{15thjune20193} need some preparations which are made in subsection \ref{19thnov20192}.

The main steps involved in the proof are as follows:

{\bf Step 1:} We prove that the isomorphism of $\widetilde{\Phi_g} : \homc(\zeroE, \zeroE \tensorc^{\rm sym} \zeroE) \rightarrow \homc(\zeroE \tensorc^{\rm sym} \zeroE, \zeroE)$ guarantees the existence of a unique left-covariant Levi-Civita connection.

{\bf Step 2:} We prove that the following diagram commutes:
\[
	\begin{tikzcd}
		\homc({}_0 \E, {}_0 \E \tensorc^{\rm sym} {}_0 \E) \arrow[r, rightarrow, "\zeta_{\zeroE \tensorc \zeroE , \zeroE}^{-1}"] \arrow[d, "\widetilde{\Phi_g}"] &[3 ex] ({}_0 \E \tensorc^{\rm sym} {}_0 \E) \tensorc ({}_0 \E)^* \arrow[r, "\id \tensorc V_g^{-1}"] & ({}_0 \E \tensorc^{\rm sym} {}_0 \E) \tensorc {}_0 \E \arrow[d, " ({}_0(\Psym))_{23} "]\\
		\homc({}_0 \E \tensorc {}_0 \E, {}_0 \E) &[3 ex] \arrow[ swap, l, " \zeta^{-1}_{{}_0 \E, {}_0 \E \tensorc {}_0 \E} "] {}_0 \E \tensorc ({}_0 \E \tensorc^{\rm sym} {}_0 \E)^* & \arrow[swap, l, "\id \tensorc V_{g^{(2)}}"] {}_0 \E \tensorc ({}_0 \E \tensorc^{\rm sym} {}_0 \E)
	\end{tikzcd} 
\] 
We note that by virute of Lemma \ref{15thjune20192} and Proposition \ref{15thjune20195}, all the arrows in the diagram except possibly $ ({}_0 (\Psym))_{23}: (\zeroE \tensorc^{{\rm sym}} \zeroE) \tensorc \zeroE \rightarrow \zeroE \tensorc (\zeroE \tensorc^{{\rm sym}} \zeroE) $ have already been proved to be isomorphisms. Thus, the isomorphism of $ ({}_0 (\Psym))_{23}$ implies the isomorphism of $\widetilde{\Phi_g} $ so that by {\bf Step 1}, we have the existence of a unique left-covariant Levi-Civita connection.

For {\bf Step 2} and the right-covariance of the Levi-Civita connection, we need to introduce an auxiliary map $ \widetilde{\Psi_g} $ and obtain certain isomorphisms. This is done in Subsection \ref{19thnov20192}.

\begin{defn} \label{26thseptember20191}
We define a map
$\widetilde{\Phi_g} : \homc(\zeroE, \zeroE \tensorc^{\rm sym} \zeroE) \rightarrow \homc(\zeroE \tensorc^{\rm sym} \zeroE, \zeroE)$ by the following formula: 
$$ \widetilde{\Phi_g}(L)=2(\id \tensorc g) \sigma_{12} (L \tensorc {\rm id}) {}_0 (\Psym).$$
\end{defn}

We start with the following proposition for which we will need a bicovariant torsionless connection whose existence was proved in Theorem \ref{20thaugust20192}.

\begin{prop}\label{Phigiso}
	Suppose $ (\E, d) $ is a bicovariant differential calculus such that ${}_0 \sigma$ is diagonalisable, and $ g $ is a bi-invariant pseudo-Riemannian metric. If the map $ \widetilde{\Phi_g} $ is a vector space isomorphism from $ \homc(\zeroE, \zeroE \tensorc^{\rm sym} \zeroE) $ to $ \homc(\zeroE \tensorc^{\rm sym} \zeroE, \zeroE), $ then there exists a unique left-covariant connection on $\E$ which is torsionless and compatible with $g$.
\end{prop} 
\begin{proof} Recall the torsionless bicovariant connection $\nabla_0$ constructed in Theorem \ref{20thaugust20192}. Then Corollary \ref{remark4.29} allows us to view $dg-\widetilde{\Pi_g}(\nabla_0) $ as an element of $ \homc(\zeroE \tensorc^{\rm sym} \zeroE,\zeroE)$. Since $\widetilde{\Phi_g}$ is an isomorphism, there exists a unique pre-image of the element $ dg - \widetilde{\Pi_g}(\nabla_0) $ under the map $ \widetilde{\Phi_g}.$ Define the $\IC$-linear map 
$$\nabla_1:=\nabla_0+\widetilde{\Phi_g}^{-1}(dg - \widetilde{\Pi_g}(\nabla_0)): \zeroE \to \zeroE \tensorc \zeroE.$$

Then $\nabla_1-\nabla_0$ is an element of $\homc(\zeroE, \zeroE \tensorc^{\rm sym} \zeroE) \subseteq \homc(\zeroE, \zeroE \tensorc \zeroE)$ and by the proof of Proposition \ref{inviff}, $\nabla_1 -\nabla_0$ extends to an element $L \in \ahoma(\E, \E \tensora \E)$. Define a $\mathbb{C}$-linear map 
$$\nabla = L + \nabla_0 : \E \to \E \tensora \E.$$
Since $L$ and $\nabla_0$ are both left-covariant maps, $\nabla$ is a left-covariant map.
Since $ \nabla_0 $ is a connection and $ L $ is right $\A$-linear, it can be easily checked that $ \nabla $ is a connection.

Now we prove that $\nabla$ is torsionless. Since $(\nabla_1 - \nabla_0)$ is an element of $\homc(\zeroE, \zeroE \tensorc^{\rm sym} \zeroE)$, $L(\omega)$ is in $\zeroE \tensorc^{\rm sym} \zeroE$ for all $\omega$ in $\zeroE$. Since $L$ is right $\A$-linear and the right $\A$-linear span of $\zeroE \tensorc^{\rm sym} \zeroE = {\rm Ran} ({}_0(\Psym)) $ is equal to $\E \tensorsym \E = {\rm Ran} (\Psym) $ (\eqref{22ndaugust20194}), $L(\omega)$ is in $\E \tensorsym \E$ for all $\rho$ in $\E$. Hence, $\wedge \circ L(\rho) = 0$ for all $\rho$ in $\E$. Therefore, for all $\rho$ in $\E$, we have
$$ \wedge \circ \nabla(\rho) = \wedge \circ (L + \nabla_0)(\rho) = \wedge \circ \nabla_0(\rho) = - d(\rho).$$
Therefore, $\nabla$ is torsionless.

Now we prove that $\nabla$ is compatible with $g.$ The fact that $ \nabla $ is torsionless means in particular that $ (\nabla - \nabla_0) (\omega) \in {\rm Ker} (\wedge) = \E \tensorsym \E. $ Thus, $ \nabla - \nabla_0 \in \Hom_\A (\E, \E \tensorsym \E) $ and so $ \widetilde{\Phi}_g (\nabla - \nabla_0) $ is well-defined. It can be easily checked that
\begin{equation} \label{1stmay20192} \widetilde{\Pi_g}(\nabla) - \widetilde{\Pi_g}(\nabla_0)=\widetilde{\Phi_g}(\nabla - \nabla_0) \end{equation}
as maps on $\zeroE \tensorc \zeroE$.

By the definition of $ \nabla, $ 
\begin{equation} \label{1stmay20193} \widetilde{\Phi}_g (\nabla - \nabla_0) = dg - \widetilde{\Pi}_g (\nabla_0) ~ {\rm on} ~ \zeroE \tensorc \zeroE. \end{equation}
Combining \eqref{1stmay20192} and \eqref{1stmay20193}, we conclude that 
 $$ \widetilde{\Pi}_g (\nabla) - \widetilde{\Pi}_g (\nabla_0) = dg - \widetilde{\Pi}_g (\nabla_0) ~ {\rm on} ~ \zeroE \tensorc \zeroE. $$
 Since $\widetilde{\Pi}_g(\nabla) - dg$ is right $\A$-linear by Proposition \ref{proppignabla} and $ \{ \omega_1 \tensorc \omega_2: \omega_1, \omega_2 \in \zeroE \} $ is right $\A$-total in $ \E \tensora \E, $ 
	$$ \widetilde{\Pi}_g (\nabla) - dg = 0 ~ {\rm as} ~ {\rm maps} ~ {\rm on} ~ \E \tensora \E. $$
	Hence, $\nabla$ is compatible with $g.$
	
To show uniqueness, suppose $\nabla^\prime$ is another torsionless left-covariant connection compatible with the metric $g$. Then, by Lemma \ref{15feb20191}, $\nabla-\nabla^\prime \in \homc(\zeroE,\zeroE \tensorc^{\rm sym} \zeroE)$ and
$$\widetilde{\Phi_g}(\nabla-\nabla^\prime)=\widetilde{\Pi_g}(\nabla)-\widetilde{\Pi_g}(\nabla^\prime)=dg - dg=0,$$
where we have used the fact that $\nabla$ and $\nabla^\prime$ are compatible with $g$.
As $\widetilde{\Phi_g}$ is an isomorphism, $\nabla-\nabla^\prime = 0$ as an element of $\homc(\zeroE, \zeroE \tensorc \zeroE)$. Since $\nabla - \nabla^{\prime}$ is a right $\A$-linear map, $\nabla = \nabla^{\prime}$ on $\E$.
\end{proof}

Proposition \ref{Phigiso} gives us a metric-dependent sufficient condition for the existence of a unique left-covariant Levi-Civita connection. Moreover, it also follows (Theorem \ref{17thseptember20191sm}) that if $(\E, d, g)$ satisfies the hypotheses of Proposition \ref{Phigiso}, then the left-covariant Levi-Civita connection is also bicovariant. However, we would like to have a metric independent sufficient condition. This is derived in Theorem \ref{15thjune20193}. Before we prove either of these results, we will need some preparatory lemmas which are derived in the next subsection.

\subsection{Some preparatory results} \label{19thnov20192}

In order to derive the right-covariance of the Levi-Civita connection, we need to define an auxiliary map $\widetilde{\Psi_g} : \homc(\zeroE, \zeroE \tensorc \zeroE) \to \homc(\zeroE \tensorc \zeroE, \zeroE).$ This definition is inspired by Definition 3.8 of \cite{article2}. In Proposition \ref{20thaugust20191}, we will prove that the map $\widetilde{\Psi_g}$ restricts to the map $\widetilde{\Phi_g}$. The goal of this subsection is to prove Proposition \ref{20thaugust20193} which states that $ \widetilde{\Psi_g} $ preserves right-covariance.

We start with an elementary lemma for which we recall that for finite dimensional vector spaces $ V, W, $ $ \zeta_{V, W} $ will be the isomorphism from $ W \tensorc V^* $ to $ \Hom_\IC (V, W) $ as introduced in Definition \ref{2ndaugust20191}. Moreover, $ V_{g^{(2)}} $ will be the map defined in Definition \ref{27thaugust2019night1}. 
\begin{lemma} \label{15thnov20191}
	For $\omega_1, \omega_2, \omega_3 \in \zeroE$, we have that 
	\begin{equation}
	\begin{split}	
	&\zeta_{\zeroE, \zeroE \tensorc \zeroE }\big((\id \tensorc V_{g^{(2)}})(\omega_1 \tensorc \omega_2 \tensorc \omega_3)\big)\circ {}_0(\Psym) \\
	=&\zeta_{\zeroE, \zeroE \tensorc \zeroE }\big((\id \tensorc V_{g^{(2)}})(\id \tensorc {}_0(\Psym))(\omega_1 \tensorc \omega_2 \tensorc \omega_3)\big). \label{2801192}
	\end{split}
	\end{equation}
	\end{lemma}
\begin{proof}
	Let $\omega_4$, $\omega_5$ be elements of $\zeroE$. Then, by the definition of $\zeta_{\zeroE , \zeroE \tensorc \zeroE}$,
	\begin{equation*}
	\begin{aligned}
		 &\zeta_{\zeroE , \zeroE \tensorc \zeroE}\big((\id \tensorc V_{g^{(2)}})(\omega_1 \tensorc \omega_2 \tensorc \omega_3) \big) \circ {}_0 (\Psym) (\omega_4 \tensorc \omega_5)\\
		=&\zeta_{\zeroE , \zeroE \tensorc \zeroE}\big(\omega_1 \tensorc V_{g^{(2)}}(\omega_2 \tensorc \omega_3) \big) \circ {}_0 (\Psym) (\omega_4 \tensorc \omega_5)\\
		=&\omega_1 V_{g^{(2)}}(\omega_2 \tensorc \omega_3) ({}_0 (\Psym) (\omega_4 \tensorc \omega_5))\\
		=&\omega_1 V_{g^{(2)}}(({}_0 (\Psym)(\omega_2 \tensorc \omega_3))) (\omega_4 \tensorc \omega_5) ~ {\rm (by \ \ref{27thaugust20194})}\\
		=&\zeta_{\zeroE , \zeroE \tensorc \zeroE}\big((\id \tensorc V_{g^{(2)}})(\id \tensorc {}_0 (\Psym))(\omega_1 \tensorc \omega_2 \tensorc \omega_3) \big) (\omega_4 \tensorc \omega_5)
	\end{aligned}
	\end{equation*}
	This proves the lemma.
\end{proof}
Now we define the map $\widetilde{\Psi_g} $ and discuss its properties. 
\begin{defn}
	We define a map $\widetilde{\Psi_g} : \homc(\zeroE, \zeroE \tensorc \zeroE) \to \homc(\zeroE \tensorc \zeroE, \zeroE)$ by the following formula:
	$$ \widetilde{\Psi_g} (L) = 2(\id \tensorc g) \circ (L \tensorc \id).$$
\end{defn}

\begin{lemma} \label{20thaugust20191}
	If $T$ is an element of $\homc(\zeroE, \zeroE \tensorc \zeroE)$, then we have that
	\begin{equation} \label{5thseptember20193sm} \widetilde{\Psi_g}(T) = 2\zeta_{\zeroE, \zeroE \tensorc \zeroE}((\id \tensorc V_{g}^{(2)})(\id \tensorc (V_{g})^{-1})(\zeta_{\zeroE \tensorc \zeroE , \zeroE}^{-1}(T))). \end{equation}
	Moreover, if $T$ is an element of $\homc({\zeroE, \zeroE \tensorc^{\rm sym} \zeroE})$, then the following two equations hold:
	\begin{equation} 
	\label{5thseptember20192sm} \widetilde{\Psi_g}(T)|_{\zeroE \tensorc^{{\rm sym}} \zeroE} = \widetilde{\Phi_g}(T), \end{equation}
	\begin{equation} \label{15thjune20191}
	\widetilde{\Phi_g}(L)=2\zeta_{\zeroE, \zeroE \tensorc \zeroE}\big((\id \tensorc V_{g^{(2)}})(\id \tensorc {}_0(\Psym))(\id \tensorc (V_g)^{-1})(\zeta^{-1}_{\zeroE \tensorc \zeroE,\zeroE}(L))\big).
	\end{equation}
 \end{lemma}
\begin{proof}
	We will use the facts (Proposition \ref{23rdmay20192}) that the elements $g_{ij} = g (e_i \tensorc e_j) $ are scalars and moreover, there exist scalars $ g^{ij} $ such that $ \sum_j g^{ij} g_{jk} = \delta_{ik}. 1 = \sum_j g_{ij} g^{jk}. $
	
	Suppose $T$ is an element of $\homc(\zeroE, \zeroE \tensorc \zeroE) $. Let $\{\omega_i\}_i$ vector space basis of $\zeroE$. Then there exist scalars $T_{ij}^m$ such that 
	$$T(\omega_m)=\sum_{ij} \omega_i \tensorc \omega_j T_{ij}^m $$
	for all $m.$ 
	Hence, we get
	\begin{equation} \label{23rdsept20191} \zeta^{-1}_{\zeroE \tensorc \zeroE,\zeroE}(T)=\sum_{ijkl}\omega_i \tensorc \omega_j \tensorc V_g(\omega_k)g^{lk}T_{ij}^l. \end{equation}
	We claim that
	\begin{equation} \label{23rdsept20192} \frac{1}{2}\widetilde{\Psi_g}(T) = \zeta_{\zeroE, \zeroE \tensorc \zeroE }\big((\id \tensorc V_{g^{(2)}})(\id \tensorc V_g^{-1})(\zeta_{\zeroE\tensorc \zeroE , \zeroE}^{-1}(T)) \big). \end{equation}
	Indeed, for all $m,n,$ we have
		\begin{align*}
		 &\frac{1}{2}\widetilde{\Psi_g}(T)(\omega_m \tensorc \omega_n)\\
		=&(\id \tensorc g)(T \tensorc \id)(\omega_m \tensorc \omega_n)\\
		=&\sum_{ij}(\id \tensorc g)(\omega_i \tensorc \omega_j \tensorc \omega_n T_{ij}^m)\\
		=&\sum_{ij}\omega_i g(\omega_j \tensorc \omega_n)T^m_{ij}\\
		=&\sum_{ijkl}\omega_i g(\omega_j \tensorc g^{lk}g(\omega_k \tensorc \omega_m)T^l_{ij} \omega_n) \\
		=&\sum_{ijkl}\omega_i g^{(2)}((\omega_j \tensorc \omega_k g^{lk}T^l_{ij})\tensorc (\omega_m \tensorc \omega_n))\\
		=&\zeta_{\zeroE, \zeroE \tensorc \zeroE }\big(\sum_{ijkl} \omega_i \tensorc V_{g^{(2)}}(\omega_j \tensorc \omega_k g^{lk} T^l_{ij}) \big)(\omega_m \tensorc \omega_n)\\
		=&\zeta_{\zeroE, \zeroE \tensorc \zeroE }\big((\id \tensorc V_{g^{(2)}})(\sum_{ijkl} \omega_i \tensorc \omega_j \tensorc \omega_k g^{lk} T^l_{ij}) \big)(\omega_m \tensorc \omega_n)\\
		=&\zeta_{\zeroE, \zeroE \tensorc \zeroE }\big((\id \tensorc V_{g^{(2)}})(\id \tensorc V_g^{-1})(\sum_{ijkl}\omega_i \tensorc \omega_j \tensorc V_g(\omega_k)g^{lk}T^l_{ij}) \big)(\omega_m \tensorc \omega_n)\\
		=&\zeta_{\zeroE, \zeroE \tensorc \zeroE }\big((\id \tensorc V_{g^{(2)}})(\id \tensorc V_g^{-1})(\zeta_{\zeroE\tensorc \zeroE , \zeroE}^{-1}(T)) \big)(\omega_m \tensorc \omega_n),\\
	\end{align*}
	where, in the last step, we have used \eqref{23rdsept20191} and also the fact (Proposition \ref{23rdmay20192}) that $ V_g $ is a vector space isomorphism from $\zeroE$ to $ (\zeroE)^*. $ This proves \eqref{23rdsept20192}.
	
	 Next, if $T$ is an element of $\rm Hom_\IC(\zeroE, \zeroE \tensorc^{\rm sym} \zeroE),$ then $ T (\omega_m) \in \zeroE \otimes^{{\rm sym}}_\C \zeroE \subseteq \E \tensorsym \E. $ Since $\sigma(X)=X$ for all $X$ in $\E \tensorsym \E = \rm Ker$$(\sigma- \id),$ we get that
	$$(\sigma T)(\omega_m)=\sigma(T(\omega_m))=T(\omega_m).$$
	Hence, 
	\begin{align*} &\widetilde{\Phi_g}(T) = 2 (\id \tensorc g)(\sigma \tensorc \id)(T \tensorc \id) ({}_0(\Psym))\\ =& 2 (\id \tensorc g)(T \tensorc \id) ({}_0(\Psym)) = \widetilde{\Psi_g} (T) ({}_0(\Psym)), \end{align*}
	which proves \eqref{5thseptember20192sm}.
	Finally, for proving \eqref{15thjune20191}, we use \eqref{23rdsept20192} and \eqref{5thseptember20192sm} to deduce that
	\begin{eqnarray*} 
	 && \widetilde{\Phi_g} (T) = \widetilde{\Psi_g} (T) ({}_0(\Psym))\\
		&=& 2 \zeta_{\zeroE, \zeroE \tensorc \zeroE }\big((\id \tensorc V_{g^{(2)}})(\id \tensorc V_g^{-1})(\zeta^{-1}_{\zeroE\tensorc \zeroE,\zeroE}(T))\big)\circ {}_0(\Psym)\\
		&=& 2 \zeta_{\zeroE, \zeroE \tensorc \zeroE }\big((\id \tensorc V_{g^{(2)}})(\id \tensorc {}_0(\Psym))(\id \tensorc V_g^{-1})(\zeta^{-1}_{\zeroE\tensorc \zeroE,\zeroE}(T))\big)
		\end{eqnarray*}
		and we have used \eqref{2801192} in the last step. This completes the proof of the lemma.
\end{proof}

For the rest of the subsection, we will be using the following notations:\\	
	The set of all right $ \A $-linear left covariant maps from $ M $ to $ N $ will be denoted by the symbol $\ahoma(M,N),$ the set of all right $\A$-linear right covariant maps from $ M $ to $ N $ will be denoted by $ {\rm Hom}^\A_\A (M, N) $ and finally, the set of all right $\A$-linear bicovariant maps will be denoted by ${}^\A \homaa(M,N)$.

\begin{prop} \label{20thaugust20193}
	If $T$ is an element of $\homc^\A(\zeroE, \zeroE \tensorc \zeroE)$, then $\widetilde{\Psi_g}(T)$ is an element of $\homc^\A(\zeroE \tensorc \zeroE, \zeroE)$. Moreover, $ \widetilde{\Phi_g} $ restricts to map from $ \Hom^\A_{\IC} (\zeroE, \zeroE \tensorc^{{\rm sym}} \zeroE) $ to $ \Hom^{\A}_{\IC} (\zeroE \tensorc^{{\rm sym}} \zeroE, \zeroE).$ 
\end{prop}
\begin{proof}
	Let us first observe that $ \zeroE, \zeroE \tensorc \zeroE, \zeroE \tensorc^{{\rm sym}} \zeroE $ are indeed right $\A$-comodules under the coactions $ {}_\E \Delta $ and $ {}_{\E \tensora \E} \Delta. $ Indeed, by Theorem 2.4 of \cite{woronowicz}, there exist elements $R_{ij}$ in $\A$ such that 
	\begin{equation} \label{18thnov20191} {}_\E \Delta (\omega_i) = \sum_j \omega_j \tensorc R_{ji} ~ {\rm so} ~ {\rm that} ~ {}_{\E \tensora \E} \Delta (\omega_i \tensorc \omega_j) = \sum_{k,l} \omega_k \tensorc \omega_l \tensorc R_{ki} R_{lj}. \end{equation}
	Now, let us recall that in the proof of Theorem \ref{25thmay20195}, we have proved that $ \E \tensorsym \E $ is a bicovariant bimodule. Since $ {}_0 (\E \tensorsym \E) = \zeroE \tensorc^{{\rm sym}} \zeroE $ by \eqref{22ndaugust20194}, we can again apply Theorem 2.4 of \cite{woronowicz} to deduce that $ \zeroE \tensorc^{{\rm sym}} \zeroE $ is invariant under $ {}_{\E \tensora \E} \Delta. $
		
	Now, we come to the proof of the result. Let $T$ be an element of $\homc^\A(\zeroE, \zeroE \tensorc \zeroE)$. Then in the notations of Lemma \ref{20thaugust20191}, there exist scalars $T^m_{ij} $ such that 
	$$T(\omega_m) = \sum_{ij} \omega_i \tensorc \omega_j T^m_{ij}.$$
	Since $T$ is right-covariant, applying Lemma \ref{21stjune20191} to the second equation of \eqref{18thnov20191} yields
	\begin{equation} \label{27thjune20195}
	\sum_{ij,n} \omega_i \tensorc \omega_j \tensorc T^n_{ij} R_{nm} = \sum_{ij,kl} \omega_k \tensorc \omega_l \tensorc R_{ki} R_{lj} T^m_{ij}.
	\end{equation}
	We note that $\zeta_{\zeroE \tensorc \zeroE , \zeroE}^{-1}(T)= \sum_{ijkl} \omega_i \tensorc \omega_j \tensorc T^l_{ij} g^{lk} V_{g}(\omega_k) $. 
	
	Then, by \eqref{5thseptember20193sm} in Lemma \ref{20thaugust20191}, 
	$$ \frac{1}{2} \zeta_{\zeroE , \zeroE \tensorc \zeroE}^{-1}(\widetilde{\Psi_g}(T))=\sum_{ijkl} \omega_i \tensorc T^l_{ij} g^{lk} V_{g}^{(2)}(\omega_j \tensorc \omega_k).$$
	Hence,
	\begin{equation} \label{26thsept20191jb} \widetilde{\Psi_g}(T)(\omega_m \tensorc \omega_n) = 2 \sum_{ijkl} \omega_i T^l_{ij} g^{lk} g^{(2)}((\omega_j \tensorc \omega_k) \tensorc (\omega_m \tensorc \omega_n)). \end{equation}
Applying Lemma \ref{21stjune20191} to the map $ \widetilde{\Psi_g} (T) $ and using \eqref{26thsept20191jb}, we can conclude that $\widetilde{\Psi_g}(T)$ is an element of $\homc^\A(\zeroE \tensorc \zeroE, \zeroE)$ if and only if, for all $m,n$, the following equation holds:
	\begin{equation}\label{27thjune20194}
	\begin{aligned} 
	&\sum_{ii^\prime jkl} \omega_{i^\prime} \tensorc R_{i^\prime i} T^l_{ij} g^{lk} g^{(2)}((\omega_j \tensorc \omega_k) \tensorc (\omega_m \tensorc \omega_n))\\ =& \sum_{ijkl,pq} \omega_i \tensorc T^l_{ij} g^{lk} g^{(2)}((\omega_j \tensorc \omega_k) \tensorc (\omega_p \tensorc \omega_q)) R_{pm} R_{qn}.
	\end{aligned}
	\end{equation}
	Hence if we prove \eqref{27thjune20194}, we are done with the first part of the theorem. 
	
	Let us note that
	\begin{equation*}
		\begin{aligned}
			& \sum_{i i^{\prime} jkl} \omega_{i^\prime} \tensorc R_{i^\prime i} T^l_{ij} g^{lk} g^{(2)}((\omega_j \tensorc \omega_k) \tensorc (\omega_m \tensorc \omega_n))\\
			=& \sum_{ii^\prime jkl} \omega_{i^\prime} \tensorc R_{i^\prime i} T^l_{ij} g^{lk} g(\omega_k \tensorc \omega_m) g (\omega_j \tensorc \omega_n) ~ {\rm (} ~ {\rm as} ~	 g(\omega_k \tensorc \omega_m) \in \IC ~ {\rm)}\\
			=& \sum_{ii^\prime jklqs} \omega_{i^\prime} \tensorc R_{i^\prime i} T^l_{ij} g^{lk} g(\omega_k \tensorc \omega_m) g (\omega_s \tensorc \omega_q) R_{sj} R_{qn},
			\end{aligned}
			\end{equation*}
		where, in the last step, we have used Proposition 3.10 of \cite{article5} by which we have 
		$$ g (\omega_j \tensorc \omega_n) = \sum_{q,s} g (\omega_s \tensorc \omega_q) R_{sj} R_{qn}.$$

			Let $ L: \A \rightarrow \Hom_{\IC} (\A, \A) $ denote the multiplication map. Since $ \zeroE \tensorc (\zeroE)^* \tensorc \Hom_\IC (\A, \A) $ is isomorphic to $ \Hom_\IC (\zeroE, \zeroE) \tensorc \Hom_\IC (\A, \A), $ we can write
	\begin{equation*}
		\begin{aligned}
		 & \sum_{ii^\prime jkl} \omega_{i^\prime} \tensorc R_{i^\prime i} T^l_{ij} g^{lk} g^{(2)}((\omega_j \tensorc \omega_k) \tensorc (\omega_m \tensorc \omega_n))\\
			=& \sum_{ii^\prime jklqs} [ \omega_{i^\prime} \tensorc V_{g}(\omega_s) \tensorc T^l_{ij} g^{lk} g(\omega_k \tensorc \omega_m) L_{(R_{i^{\prime} i} R_{sj})} ](\omega_q \tensorc R_{qn})\\
			=& \sum_{ii^\prime jlqs} [(\id \tensorc V_{g} \tensorc L)(\omega_{i^\prime} \tensorc \omega_s \tensorc T^l_{ij} \sum_k (g^{lk} g (\omega_k \tensorc \omega_m)) R_{i^{\prime} i} R_{sj})](\omega_q \tensorc R_{qn})\\
			=& \sum_{ii^\prime jlqs} [(\id \tensorc V_{g} \tensorc L)(\omega_{i^\prime} \tensorc \omega_s \tensorc T^l_{ij} \delta_{lm} R_{i^{\prime} i} R_{sj})](\omega_q \tensorc R_{qn})\\
			=& \sum_{ii^\prime jqs} [(\id \tensorc V_{g} \tensorc L)(\omega_{i^\prime} \tensorc \omega_s \tensorc R_{i^{\prime} i} R_{sj} T^m_{ij})](\omega_q \tensorc R_{qn})\\
			=& \sum_{i j pq} [(\id \tensorc V_{g} \tensorc L)(\omega_{i} \tensorc \omega_j \tensorc T^p_{ij} R_{pm})](\omega_q \tensorc R_{qn}) \ \big({\rm by \ \eqref{27thjune20195}} \big)\\
			=& \sum_{i j pql} [(\id \tensorc V_{g} \tensorc L)(\omega_{i} \tensorc \omega_j \tensorc T^l_{ij} \delta_{lp} R_{pm})](\omega_q \tensorc R_{qn})\\
			=& \sum_{i jl pq} [(\id \tensorc V_{g} \tensorc L)(\omega_{i} \tensorc \omega_j \tensorc T^l_{ij} (\sum_k g^{lk} g (\omega_k \tensorc \omega_p)) R_{pm})](\omega_q \tensorc R_{qn})\\
			=& \sum_{i jkl pq} (\omega_i \tensorc V_{g}(\omega_j) \tensorc T^l_{ij} g^{lk} g (\omega_k \tensorc \omega_p) L_{R_{pm}})(\omega_{q} \tensorc R_{qn})\\
			=& \sum_{ijklpq} \omega_{i} \tensorc T^l_{ij} g^{lk} g (\omega_j \tensorc \omega_q) g (\omega_k \tensorc \omega_p) R_{pm} R_{qn}\\
			=& \sum_{ijkl,pq} \omega_i \tensorc T^l_{ij} g^{lk} g^{(2)}((\omega_j \tensorc \omega_k) \tensorc (\omega_p \tensorc \omega_q)) R_{pm} R_{qn}.
		\end{aligned}
	\end{equation*}
	This proves \eqref{27thjune20194} and therefore, $ \widetilde{\Psi_g} (T) $ is right-covariant.
	
	Now we prove the second assertion of the proposition. Let $T$ be an element of $ \Hom^\A_{\IC} (\zeroE, \zeroE \tensorc^{{\rm sym}} \zeroE).$ Then the first assertion of the proposition implies that $ \widetilde{\Psi_g} (T) $ belongs to $ \Hom^\A_{\IC} (\zeroE \tensorc \zeroE, \zeroE). $ However, by \eqref{5thseptember20192sm}, $ \widetilde{\Psi_g} (T)|_{\zeroE \tensorc^{{\rm sym}} \zeroE} = \widetilde{\Phi_g} (T) $ and by the definition of $ \widetilde{\Phi_g},$ we know that $ \widetilde{\Phi_g} (T) $ belongs to $ \Hom_{\IC} (\zeroE \tensorc^{{\rm sym}} \zeroE, \zeroE).$ Hence, we conclude that $ \widetilde{\Phi_g} (T) $ belongs to $ \Hom^\A_{\IC} (\zeroE \tensorc^{{\rm sym}} \zeroE, \zeroE).$ This finishes the proof of the proposition.
	\end{proof}

\subsection{Right-covariance of the unique left-covariant connection} \label{26thseptember20192}

In this subsection, we prove that the unique torsion-less left-covariant connection compatible with a bi-invariant pseudo-Riemannian metric, obtained under the hypothesis of Proposition \ref{Phigiso}, is actually a bicovariant connection if the Hopf algebra $\A$ is cosemisimple, i.e, if the category of finite dimensional comodules of $\A$ is semisimple. For right $\A$-comodules $V$ and $W,$ the symbol $ \Hom^\A_{\IC} (V, W) $ will continue to denote the set of all right-covariant complex linear maps from $V$ to $W.$ 

If $ \A $ is a cosemisimple Hopf algebra and $V, W$ be finite dimensional comodules as above, then it is well-known that $ {\rm dim} (\Hom^\A_{\IC} (V, W)) = {\rm dim} (\Hom^\A_{\IC} (W, V)). $ Now, if $\A$ is a cosemisimple Hopf algebra and $ (\E, d) $ be a differential calculus such that $ {}_0 \sigma $ is diagonalisable, then in the proof of Proposition \ref{20thaugust20193}, we have noted that $ \zeroE $ and $ \zeroE \tensorc^{{\rm sym}} \zeroE $ are right $\A$-comodules. Hence, we can conclude that
 \begin{equation} \label{19thaugust20192} {\rm dim}(\homc^\A({}_0\E, {}_0\E \tensorc^{\rm sym} {}_0\E)) = {\rm dim}(\homc^\A({}_0\E\tensorc^{\rm sym} {}_0\E, {}_0\E)). \end{equation}

Then we have the following theorem.

\begin{thm} \label{17thseptember20191sm}
	Suppose $ (\E, d) $ is a bicovariant differential calculus over a cosemisimple Hopf algebra $\A$ such that the map ${}_0 \sigma$ is diagonalisable, and $ g $ is a bi-invariant pseudo-Riemannian metric. If the map $\widetilde{\Phi_g}$ is an isomorphism, then the unique left-covariant connection guaranteed by Proposition \ref{Phigiso} is in fact a bicovariant connection.
\end{thm}
\begin{proof}
 The proof follows from the claim that under the hypothesis of the theorem, the map $ \widetilde{\Phi_g} $ is an isomorphism from $ \Hom^\A_{\IC} (\zeroE, \zeroE \tensorc^{{\rm sym}} \zeroE) $ onto $ \Hom^\A_{\IC} (\zeroE \tensorc^{{\rm sym}} \zeroE, \zeroE). $ Indeed, let us recall that 
	in Proposition \ref{Phigiso}, under the assumption that the map $\widetilde{\Phi_g}$ is an isomorphism, we explicitly constructed a torsionless left-covariant connection $\nabla$ compatible with $g$ by the formula $\nabla := L + \nabla_0$. 
	
	Here $\nabla_0$ is the torsionless bicovariant connection constructed in Theorem \ref{20thaugust20192} and $L: \E \to \E \tensora \E$ is the left-covariant right $\A$-linear extension (via Proposition \ref{ahoma}) of the map 
	$$\widetilde{\Phi_g}^{-1}(dg - \widetilde{\Pi_g}(\nabla_0)) : \zeroE \to \zeroE \tensorc \zeroE.$$
	By Corollary \ref{remark4.29}, $dg - \widetilde{\Pi_g}(\nabla_0)$ is a right $\A$-covariant $\IC$-linear map from $ \zeroE \tensorc^{\rm sym} \zeroE $ to $\zeroE$. 
	Hence, our claim implies that $ {}_0 L = \widetilde{\Phi_g}^{-1}(dg - \widetilde{\Pi_g}(\nabla_0)) $ belongs to $ \Hom^\A_{\IC} (\zeroE, \zeroE \tensorc^{{\rm sym}} \zeroE). $
	
	Since $ L $ is left-covariant right $\A$-linear and $ {}_0 L = \widetilde{\Phi_g}^{-1}(dg - \widetilde{\Pi_g}(\nabla_0)) $ is right-covariant, Proposition \ref{18thsep20196} implies that the extension $L$ is a bicovariant right $\A$-linear map from $\E$ to $\E \tensora \E$. Again by the right-covariance of $\nabla_0$, $\nabla = L + \nabla_0$ is a right-covariant map as well.
	
	So we are left with proving that the map $ \widetilde{\Phi_g} : \Hom^\A_{\IC} (\zeroE, \zeroE \tensorc^{{\rm sym}} \zeroE) \rightarrow \Hom^\A_{\IC} (\zeroE \tensorc^{{\rm sym}} \zeroE, \zeroE) $ is an isomorphism. To this end, we observe that since $ \widetilde{\Phi_g} $ is an isomorphism from $ \Hom_{\IC} (\zeroE, \zeroE \tensorc^{{\rm sym}} \zeroE) $ to $ \Hom_{\IC} (\zeroE \tensorc^{{\rm sym}} \zeroE, \zeroE), $ Proposition \ref{20thaugust20193} implies that $ \widetilde{\Phi_g} $ is a one-one map from $ \Hom^\A_{\IC} (\zeroE, \zeroE \tensorc^{{\rm sym}} \zeroE) $ into $ \Hom^\A_{\IC} (\zeroE \tensorc^{{\rm sym}} \zeroE, \zeroE). $ However, by \eqref{19thaugust20192}, 
	$$ {\rm dim} (\Hom^\A_{\IC} (\zeroE \tensorc^{{\rm sym}} \zeroE, \zeroE)) = {\rm dim} (\Hom^\A_{\IC} (\zeroE, \zeroE \tensorc^{{\rm sym}} \zeroE)). $$
	Therefore, $ \widetilde{\Phi_g} $ is a one-one and onto map from $ \Hom^\A_{\IC} (\zeroE, \zeroE \tensorc^{{\rm sym}} \zeroE) $ to $ \Hom^\A_{\IC} (\zeroE \tensorc^{{\rm sym}} \zeroE, \zeroE). $
\end{proof}

\subsection{Sufficient conditions for the isomorphism of $ \widetilde{\Phi_g} $} \label{psym23}
In this subsection, we prove a metric-independent sufficient condition for the map $ \widetilde{\Phi_g} $ to be an isomorphism. We will continue to use the notation $ \zeta_{\E, \F} $ introduced in Definition \ref{2ndaugust20191}.

\begin{thm} \label{15thjune20193}
	Suppose $(\E, d)$ is a bicovariant differential calculus over a cosemisimple Hopf algebra $\A$ such that the map ${}_0 \sigma$ is diagonalisable and $ g $ be a bi-invariant pseudo-Riemannian metric.
	
	The map $ \widetilde{\Phi_g} : \Hom_{\IC} (\zeroE, \zeroE \tensorc^{{\rm sym}} \zeroE) \rightarrow \Hom_{\IC} (\zeroE \tensorc^{{\rm sym}} \zeroE, \zeroE) $ is an isomorphism if and only if $ ({}_0(\Psym))_{23}: (\zeroE \tensorc^{\rm sym} \zeroE) \tensorc \zeroE \rightarrow \zeroE \tensorc (\zeroE \tensorc^{\rm sym} \zeroE) $ is an isomorphism.
In particular, Theorem \ref{17thseptember20191sm} implies that under either of these assumptions, the triple $ (\E, d, g) $ admits a unique bicovariant Levi-Civita connection.
\end{thm}
\begin{proof}
	 Suppose $({}_0(\Psym))_{23}: (\zeroE \tensorc^{\rm sym} \zeroE) \tensorc \zeroE \rightarrow \zeroE \tensorc (\zeroE \tensorc^{\rm sym} \zeroE)$ is an isomorphism. Since $g$ is left-invariant, part (i) of Proposition \ref{23rdmay20192} implies that $V_{g}^{-1}((\zeroE)^*)=\zeroE$. By the first assertion of Lemma \ref{15thjune20192} and our hypothesis, we can conclude that the following map is an isomorphism:
	$$({}_0(\Psym))_{23}(\id \tensorc V_{g}^{-1})\zeta_{\zeroE \tensorc \zeroE , \zeroE}^{-1}: \homc(\zeroE, \zeroE \tensorc^{\rm sym} \zeroE) \rightarrow \zeroE \tensorc (\zeroE \tensorc^{\rm sym} \zeroE).$$
	Now, by Proposition \ref{15thjune20195}, $V_{g^{(2)}}$ is an isomorphism from $\zeroE \tensorc^{\rm sym} \zeroE$ to $(\zeroE \tensorc^{\rm sym} \zeroE)^*$. Finally, by the second assertion of Lemma \ref{15thjune20192}, $\zeta_{\zeroE, \zeroE \tensorc \zeroE }$ is an isomorphism from $\zeroE \tensorc (\zeroE \tensorc^{\rm sym} \zeroE)^*$ to $\homc(\zeroE \tensorc^{\rm sym} \zeroE,\zeroE)$. Therefore, by \eqref{15thjune20191}, 
	is a composition of isomorphisms and hence an isomorphism itself.

	Conversely, suppose $\widetilde{\Phi_g}: \Hom_{\IC} (\zeroE, \zeroE \tensorc^{{\rm sym}} \zeroE) \rightarrow \Hom_{\IC} (\zeroE \tensorc^{{\rm sym}} \zeroE, \zeroE) $ is an isomorphism. If $({}_0 (\Psym))_{23}$ is not an isomorphism from $ (\zeroE \tensorc^{\rm sym} \zeroE) \tensorc \zeroE$ to $\zeroE \tensorc (\zeroE \tensorc^{\rm sym} \zeroE), $ then it is not one-one. Hence by \eqref{15thjune20191}, $\widetilde{\Phi_g}$ is not an isomorphism, which is a contradiction.
\end{proof}
\begin{rmk}
 In \cite{suq2}, the isomorphism $ ({}_0(\Psym))_{23}: (\zeroE \tensorc^{\rm sym} \zeroE) \tensorc \zeroE \rightarrow \zeroE \tensorc (\zeroE \tensorc^{\rm sym} \zeroE) $ is verified by an explicit computation. We refer to Theorem \ref{18thjuly20194} for a cocycle-twisted version of the above isomorphism.
\end{rmk}	
Our next proposition states that if $ \sigma^2 = 1, $ then the hypothesis of Theorem \ref{15thjune20193} is satisfied.
\begin{prop} \label{15thjune20197}
	If $\sigma^2=1$, then the map $({}_0(\Psym))_{23}$ is an isomorphism from $ ({}_0\E \tensorc^{\rm sym} \E) \tensorc \E$ to $\E \tensorc (\E \tensorc^{\rm sym} \E).$
\end{prop}
\begin{proof}
	Since $\sigma^2 = 1,$ $\pm 1$ are the only eigenvalues of ${}_0\sigma$ in this case and so by \eqref{24thaugust20192}, ${}_0(\Psym)= \frac{1}{2}(1 + {}_0\sigma)$. Now, let $X$ be an element of $ ({}_0\E \tensorc^{\rm sym} \zeroE) \tensorc \zeroE$ such that $({}_0(\Psym))_{23}(X)=0$. Then
	$ (\Psym)_{(12)} (X) = X $ so that $ ({}_0\sigma)_{12}(X)=X.$
	
	Moreover, $ ({}_0\sigma)_{23}(X)= (2({}_0(\Psym))_{23}-1)(X) =-X.$
	We further obtain that \[({}_0\sigma)_{12}({}_0\sigma)_{23}({}_0\sigma)_{12}(X)=-X \quad {\rm and} \quad ({}_0\sigma)_{23}({}_0\sigma)_{12}({}_0\sigma)_{23}(X)=X.\]
	Since ${}_0\sigma$ is a braiding, this implies that $X=0$. Hence $({}_0(\Psym))_{23}$ is a one-one map from $ ({}_0\E \tensorc^{\rm sym} \zeroE) \tensorc \zeroE$ to $\zeroE \tensorc (\zeroE \tensorc^{\rm sym} \zeroE)$ and therefore, by a dimension count, $({}_0(\Psym))_{23}$ is also onto $\zeroE \tensorc (\zeroE \tensorc^{\rm sym} \zeroE)$. Hence $({}_0(\Psym))_{23}$ is an isomorphism from ${}_0\E \tensorc^{\rm sym} \zeroE \tensorc \zeroE$ to $\zeroE \tensorc \zeroE \tensorc^{\rm sym} \zeroE$.
	
	${}_0\sigma$ is an involution, $\frac{1}{2}(1 + {}_0\sigma)$ is an idempotent map and hence
	\[{\rm Ran}(\frac{1}{2}(1 + {}_0\sigma))= {\rm Ker}({}_0\sigma -1) = {}_0\E \tensorc^{\rm sym} {}_0\E.\]
	Since $\pm 1$ are the only eigenvalues of ${}_0\sigma$ in this case, by \eqref{24thaugust20192}, ${}_0(\Psym)= \frac{1}{2}(1 + {}_0\sigma)$. 
\end{proof}

\begin{rmk}
In Corollary \ref{18thsep20194}, we show that the hypothesis of Proposition \ref{15thjune20197} holds for the space of one-forms for cocycle deformations of a linear algebraic group $G$ whose category of finite dimensional representations is semisimple. Thus, for these examples, we indeed have a unique bicovariant Levi-Civita connection by Proposition \ref{15thjune20197} (see Proposition \ref{18thsep20195}).
\end{rmk}

\section{Levi-Civita connections for cocycle deformations} \label{sectioncocyclelc}

This section concerns the Levi-Civita connections on bicovariant differential calculus on cocycle deformations of Hopf algebras. We will start by recalling the cocycle deformations of bicovariant differential calculus and pseudo-Riemannian metrics. Then we discuss the effect of cocycle deformations on the map $ \Psym $ as well as bicovariant connections. Finally, we prove the main theorem which states that if $ (\E, d) $ is a bicovariant differential calculus such that ${}_0 \sigma$ is diagonalisable and $g$ is a pseudo-Riemannian bi-invariant metric on $\E$ such that $ (\E, d, g) $ admits a bicovariant Levi-Civita connection, then $\nabla$ deforms to a Levi-Civita connection for the deformed triple $ (\E_\Omega, d_\Omega, g_\Omega). $ Throughout the section, we will use several facts about cocycle deformations for which our main reference is \cite{paganietal}.
For a Hopf algebra $ (\A, \Delta), $ we will denote its restricted dual by the symbol $ (\widehat{\A}, \widehat{\Delta}). $
\bdfn
A normalized dual $2$-cocycle on a Hopf algebra $ \A $ is an invertible element $ \Omega$ in $\widehat{A} \tensorc \widehat{A} $ such that the following equations hold:
$$ (\widehat{\epsilon} \tensorc \id_{\widehat{\A}}) (\Omega) = 1 = (\id_{\widehat{\A}} \tensorc \widehat{\epsilon}) (\Omega),~
 (\Omega \tensorc 1) [ (\widehat{\Delta} \tensorc \id_{\widehat{\A}}) (\Omega) ] = (1 \tensorc \Omega) [ (\id_{\widehat{\A}} \tensorc \widehat{\Delta}) (\Omega) ]. $$
Here, $ \widehat{\epsilon} $ denotes the counit of the Hopf algebra $ \widehat{A}. $
\edfn
Given a Hopf algebra $ (\A, \Delta) $ and such a cocycle $ \Omega $, we have a new Hopf algebra $ (\A_\Omega, \Delta_\Omega) $ which is equal to $ \A $ as a vector space, the coproduct $ \Delta_\Omega = \Delta. $ However, the algebra structure $ \ast_\Omega $ on $ \A_\Omega $ is twisted and to describe it, we will need an auxiliary map $\gamma.$ Indeed, if $ \Omega = \sum_i h_i \tensorc k_i \in \widehat{A} \tensorc \widehat{\A}, $ let us define a map $ \gamma: \A \tensorc \A \rightarrow \mathbb{C} $ by:
 $$ \gamma (a \tensorc b) = \sum_i h_i (a) k_i (b). $$
Then $\gamma$ is a unital and convolution invertible map which satisfies the following equation for all $ a, b, c \in \A: $ 
$$ \gamma (a_{(1)} \tensorc b_{(1)}) \gamma (a_{(2)} b_{(2)} \tensorc c) = \gamma (b_{(1)} \tensorc c_{(1)}) \gamma (a \tensorc b_{(2)} c_{(2)}). $$
If $ \overline{\gamma} $ is the convolution inverse to $ \gamma, $ then the deformed product $ \ast_\Omega $ is defined by:
$$ a \ast_\Omega b = \gamma (a_{(1)} \tensorc b_{(1)}) a_{(2)} b_{(2)} \overline{\gamma} (a_{(3)} \tensorc b_{(3)}).$$
Bicovariant bimodules over $\A$ and bicovariant maps between any two such bimodules are canonically twisted in the presence of a cocycle.
Suppose $ M $ is a bicovariant $\A$-bimodule and $ \Omega $ is a dual $2$-cocycle on $\A.$ Then by Proposition 4.2 of \cite{article5} (see also Proposition 2.27 of \cite{paganietal}), we have a bicovariant $\A_\Omega $-bimodule $ M_\Omega $ which is equal to $ M $ as a vector space but the left and right $ \A_\Omega $-module structures are defined by the following formulas:
	\begin{eqnarray} & \label{9thmay20197} a *_\Omega m = \gamma(a_{(1)} \tensorc m_{(-1)}) a_{(2)} . m_{(0)} \overline{\gamma}(a_{(3)} \tensorc m_{(1)})\\ & \label{9thmay20198} m *_\Omega a = \gamma(m_{(-1)} \tensorc a_{(1)}) m_{(0)} . a_{(2)} \overline{\gamma}(m_{(1)} \tensorc a_{(3)}), \end{eqnarray}
	for all elements $m$ of $M$ and for all elements $a$ of $\A$. Here, $*_\Omega$ denotes the right and left $\A_\Omega$-module actions, and $.$ denotes the right and left $\A$-module actions.
	The $\A_\Omega$-bicovariant structures are given by \begin{equation} \label{10thseptember20191sm} {}_{M_\Omega} \Delta:= {}_M \Delta: M_\Omega \rightarrow \A_\Omega \tensorc M_\Omega ~ {\rm and} ~ {}_{M_\Omega} \Delta:= {}_M \Delta: M_\Omega \rightarrow M_\Omega \tensorc \A_\Omega. \end{equation} 
In particular, if $ \omega \in {}_0 M $ and $ a \in \A, $ we have
\begin{equation} \label{9thmay20192} \omega *_\Omega a = \omega_{(0)}. a_{(1)} \overline{\gamma} (\omega_{(1)} \tensorc a_{(2)}). \end{equation}
Now we come to the deformation of maps:
\begin{prop} \label{11thjuly20192} (Proposition 2.27 of \cite{paganietal})
	Let $(M, \Delta_M, {}_M \Delta)$ and $(N, \Delta_N, {}_N \Delta)$ be bicovariant $\A$-bimodules, $T: M \to N$ be a $\IC$-linear bicovariant map and $\Omega$ be a cocycle as above. Then there exists a map $T_\Omega: M_\Omega \to N_\Omega$ defined by $T_\Omega (m) = T(m)$ for all $m$ in $M$. Thus, $T_\Omega = T$ as $\IC$-linear maps. Moreover, we have the following:
	\begin{itemize}
		\item[(i)] the deformed map $T_\Omega: M_\Omega \to N_\Omega$ is an $\A_\Omega$ bicovariant map,
		\item[(ii)] if $T$ is a bicovariant right (respectively left) $\A$-linear map, then $T_\Omega$ is a bicovariant right (respectively, left) $\A_\Omega$-linear map,
		\item[(iii)] if $(P, \Delta_P, {}_{P} \Delta)$ is another bicovariant $\A$-bimodule, and $S: N \to {P} $ is a bicovariant map, then $(S \circ T)_\Omega : M_\Omega \to P_\Omega$ is a bicovariant map and $S_\Omega \circ T_\Omega = (S \circ T)_\Omega$.
	\end{itemize} 
\end{prop}
As a corollary to Proposition \ref{11thjuly20192}, we obtain:
\begin{prop} \label{11thjuly20193}
	Let $(M, \Delta _M, {}_M \Delta)$ and $(N, \Delta_M, {}_N \Delta)$ be bicovariant bimodules of a Hopf algebra $\A$, $T$ be a bicovariant right $\A$-linear map from $M$ to $N$ and $\Omega$ be a cocycle as above. Then 
	\begin{equation} \label{27thaugust20192}
	 T_\Omega = u^{N_\Omega} \circ ({}_0 T \tensorc \id) \circ (u^{M_\Omega})^{-1}.
	\end{equation}
	In particular, the $\IC$-linear map ${}_0(T_\Omega)$ from ${}_0(M_\Omega) = {}_0 M$ to itself (as in Proposition \ref{ahoma}) coincides with ${}_0T$.
	Moreover, $T_\Omega$ is an invertible map if and only if $ T$ is invertible, and more generally, $\lambda$ is an eigenvalue of $T_\Omega$ if and only if it is an eigenvalue of $T$.
\end{prop}
\begin{proof}
	Since $T$ is a bicovariant right $\A$-linear map from $M$ to $N$, by Proposition \ref{11thjuly20192}, $T_\Omega$ is an $\A_\Omega$ bicovariant right linear map. Since ${}_0 (M_\Omega) ={}_0 M$ and ${}_0 (N_\Omega)$ as vector spaces, and $T_\Omega$ is a left-covariant map, hence for all $m$ in ${}_0 (M_\Omega)$, the element $T_\Omega(m)$ belongs to ${}_0 (N_\Omega)$. Then we compute, for any $m$ in ${}_0 (M_\Omega)$ and any element $a$ of $\A_\Omega,$
	\begin{equation*}
		\begin{aligned}
			& (u^{N_\Omega})^{-1} \circ T_\Omega (m *_\Omega a) = (u^{N_\Omega})^{-1} (T_\Omega (m) *_\Omega a)
			= T_\Omega (m) \tensorc a {\rm \ (by \ the \ definition \ of \ } u^{N_\Omega})\\
			=& T(m) \tensorc a
			= ({}_0 T) (m) \tensorc a = ({}_0 T \tensorc \id)(u^{M_\Omega})^{-1}(m *_\Omega a),
		\end{aligned}	
	\end{equation*}
	as $m$ belongs to ${}_0 (M_\Omega)$. Thus we have that 
	$$(u^{N_\Omega})^{-1} \circ T_\Omega = ({}_0 T \tensorc \id)(u^{M_\Omega})^{-1}, ~ {\rm i.e.,} ~ T_\Omega = u^{N_\Omega} \circ ({}_0 T \tensorc \id) \circ (u^{M_\Omega})^{-1}.$$
	Evaluating this equation on an element of ${}_0(M_\Omega) = {}_0 M$ yields ${}_0 (T_\Omega) = {}_0 T$.
	
	Finally, applying Proposition \ref{inviff} to $T_\Omega$ and using the fact that ${}_0 (T_\Omega) = {}_0 T$, we get that $T_\Omega$ is invertible if and only if $T$ is invertible. More generally, $\lambda $ is an eigenvalue of $T_\Omega$ if and only if it is an eigenvalue of $T$.
\end{proof}	

Now we are ready to state the main theorem of this subsection. 
\bthm \label{11thmay20192} 
 Suppose $ (\E, d) $ is a bicovariant differential calculus on a Hopf algebra $\A$ and $ \Omega $ a dual $2$-cocycle. Then the following statements hold:
\begin{itemize}
\item[(i)] $ (\E_\Omega, d_\Omega) $ is a differential calculus on $ \aomega. $ 
\item[(ii)] The deformation $\sigma_\Omega$ of $\sigma$ is the unique bicovariant $\A_\Omega$-bimodule braiding map on $\E_\Omega$ given by Proposition \ref{4thmay20193}. If the map ${}_0\sigma$ is diagonalisable. Then the map ${}_0(\sigma_\Omega)$ is also diagonalisable.
\item[(iii)] If $g$ is a bi-invariant pseudo-Riemannian metric on $\E,$ then $g$ deforms to a bi-invariant pseudo-Riemannian metric $g_\Omega$ on $\E_\Omega$. Any bi-invariant pseudo-Riemannian metric on $\E_\Omega$ is a deformation of some bi-invariant pseudo-Riemannian metric on $\E$.
\end{itemize}
\ethm
\begin{proof}
	The first assertion follows from Proposition 3.2 and Corollary 3.4 of \cite{majidcocycle}. The first part of the second assertion follows from Theorem 4.7 of \cite{article5}. Next, by Proposition \ref{11thjuly20193}, we have that the $\IC$-linear maps ${}_0(\sigma_\Omega)$ and ${}_0 \sigma$ coincide. Therefore, if ${}_0 \sigma$ is diagonalisable, so is ${}_0(\sigma_\Omega)$ is diagonalisable.
	
	Finally, the third assertion follows from Theorem 4.15 of \cite{article5}.
\end{proof}
As a direct consequence of the above theorem, we have the following:
\begin{corr} \label{18thsep20194}
	If the unique bicovariant $\A$-bimodule braiding map $\sigma$ for a bicovariant $\A$-bimodule $\E$ satisfies the equation $\sigma^2 = 1$, then the bicovariant $\A_\Omega$-bimodule braiding map $ \sigma_\Omega$ for the bicovariant $\A_\Omega$-bimodule $\E_\Omega$ also satisfies $\sigma_\Omega^2 = 1$.
	
	In particular, if $\A$ is the commutative Hopf algebra of regular functions on a linear algebraic group $G$ and $\E$ is its canonical space of one-forms, then the braiding map $\sigma_\Omega$ for $\E_\Omega$ satisfies $\sigma_\Omega^2 = 1$.
\end{corr}

Before we proceed further, let us recall from Proposition 2.26 of \cite{paganietal} that there exists a bicovariant $\A_\Omega$- bimodule isomorphism 
	$$ \xi: \E_\Omega \otimes_{\A_\Omega} \E_\Omega \rightarrow (\E \tensora \E)_\Omega. $$
	The isomorphism $\xi$ and its inverse are respectively given by
	\begin{eqnarray}
	\xi (m \otimes_{\A_\Omega} n) & = \gamma(m_{(-1)} \tensorc n_{(-1)}) m_{(0)} \tensora n_{(0)} \overline{\gamma}(m_{(1)} \tensorc n_{(1)}) \label{23rdnov20192} \\
	 \xi^{-1} (m \tensora n) & = \overline{\gamma}(m_{(-1)} \tensorc n_{(-1)}) m_{(0)} \otimes_{\A_\Omega} n_{(0)} {\gamma}(m_{(1)} \tensorc n_{(1)}) \label{23rdnov20193}
	\end{eqnarray}
	We will use this isomorphism for the rest of this subsection as well as the next subsection.

If $(\E,d)$ is a bicovariant differential calculus such that ${}_0\sigma$ is diagonalisable, then we have proved (Theorem \ref{25thmay20195}) that $\E \tensora \E = {\rm Ker}(\wedge) \oplus \F$, where $\F = \widetilde{u}^{\E \tensora \E}({}_0 \F \tensorc \A)$. Here, ${}_0 \F$ is the direct sum of eigenspaces of ${}_0 \sigma$ corresponding to the eigenvalues which are not equal to $1$ and $\widetilde{u}^{\E \tensora \E}$ is the isomorphism defined in \eqref{22ndaugust20193}. Moreover, we have a bicovariant $\A$-bilinear idempotent map $\Psym$ on $\E \tensora \E$ with range ${\rm Ker}(\wedge)$ and kernel $\F$. $\Psym$ is defined by the equation (see Definition \ref{27thmay20193})
\[ \Psym = \widetilde{u}^{\E \tensora \E}({}_0 (\Psym) \tensorc \id)(\widetilde{u}^{\E \tensora \E})^{-1}. \]
Since $ \Psym : \E \tensora \E \rightarrow \E \tensora \E $ is bicovariant, we have a deformed map $ (\Psym)_\Omega: (\E \tensora \E)_\Omega \rightarrow (\E \tensora \E)_\Omega. $ With an abuse of notation, we will denote the map $ \xi^{-1} (\Psym)_\Omega \xi : \E_\Omega \otimes_{\A_\Omega} \E_\Omega \rightarrow \E_\Omega \otimes_{\A_\Omega} \E_\Omega $ by the symbol $ (\Psym)_\Omega $ again.

Now, let us consider the bicovariant differential calculus $(\E_\Omega, d_\Omega)$. By Theorem \ref{11thmay20192}, we can apply Theorem \ref{25thmay20195} (to $(\E_\Omega, d_\Omega)$) to get a bicovariant $\A_\Omega$-bilinear idempotent $(\Psym)_{\E_\Omega}$ on $\E_\Omega \otimes_{\A_\Omega} \E_\Omega$. It is worthwhile to note that the map $(\Psym)_{\E_\Omega}$ coincides with the cocycle deformation $(\Psym)_\Omega$ of the map $\Psym$. Indeed, since ${}_0 (\sigma_\Omega) = {}_0 \sigma$ on ${}_0(\E_\Omega) \tensorc {}_0 (\E_\Omega) = {}_0 \E \tensorc \zeroE$, the kernel of $(\Psym)_{\E_\Omega}$ is equal to $\widetilde{u}^{\E_\Omega \otimes_{\A_\Omega} \E_\Omega}({}_0 \F \tensorc \A_\Omega)$. However, using the isomorphism $(\E \tensora \E)_\Omega \cong \E_\Omega \otimes_{\A_\Omega} \E_\Omega$, it is easy to check that
\begin{align*}
	&\widetilde{u}^{\E_\Omega \otimes_{\A_\Omega} \E_\Omega}({}_0\F \tensorc \A_\Omega) = (\widetilde{u}^{\E \tensora \E})_\Omega (({}_0 \F \tensorc \A)_\Omega)\\ = &((\widetilde{u}^{\E \tensora \E})({}_0 \F \tensorc \A))_\Omega = \F_\Omega = {\rm Ker}((\Psym)_\Omega).
\end{align*}
On the other hand, by the definition of $(\Psym)_{\E_\Omega}$,
\begin{align*}
	&{\rm Ran}((\Psym)_{\E_\Omega}) = {\rm Ker}(\sigma_\Omega - 1) = ({\rm Ker}(\sigma - 1))_\Omega \\
	=&({\rm Ran}(\Psym))_\Omega = {\rm Ran}((\Psym)_\Omega)
\end{align*}
Since $(\Psym)_{\E_\Omega}$ and $(\Psym)_\Omega$ are both idempotents on $\E_\Omega \otimes_{\A_\Omega} \E_\Omega$ with the same kernel and the same range, we can conclude that $(\Psym)_\Omega = (\Psym)_{\E_\Omega}$. We collect the observations made above in the following proposition.

\begin{prop} \label{25thseptember20191}
	Let $(\E,d)$ be a bicovariant differential calculus over $\A$ such that ${}_0 \sigma$ is diagonalisable and $\Omega$ be a normalised dual 2-cocycle. Then the maps $(\Psym)_{\E_\Omega}$ and $(\Psym)_\Omega$ coincide. Moreover, we have
	\[ \E_\Omega \otimes_{\A_\Omega} \E_\Omega = {\rm Ker}(\wedge_\Omega) \oplus \F_\Omega = {\rm Ker}(\sigma_\Omega - 1) \oplus \F_\Omega. \]
\end{prop}

\subsection{Deformation of connections} \label{18thseptember20192}
Now we deal with bicovariant connections on $ \E_\Omega. $ Suppose that $ \nabla $ is a bicovariant connection on $ \E. $ Then Proposition \ref{11thjuly20192} yields a $ \mathbb{C} $-linear map $ \nabla_\Omega $ from $ \E_\Omega $ to $ (\E \tensora \E)_\Omega. $ However, we would like to have the deformed map to take value in $ \E_{\Omega} \otimes_{\A_\Omega} \E_{\Omega}.$ For this, we will need to use the isomorphism 
$ \xi: \E_\Omega \otimes_{\A_\Omega} \E_\Omega \rightarrow (\E \tensora \E)_\Omega $ introduced in the previous subsection and the equations \eqref{23rdnov20192}, \eqref{23rdnov20193}.

	The following lemma will be needed to prove that $ \nabla_\Omega $ is actually a connection.
\begin{lemma} \label{14thseptember20191}
Suppose $ \omega \in \zeroE $ and $ a \in \A. $ Then the following equation holds:
	$$ \xi^{-1} (\omega_{(0)} \tensora (da)_{(0)} \overline{\gamma} (\omega_{(1)} \tensorc (da)_{(1)})) = \omega \otimes_{\A_\Omega} d_{\Omega} (a).$$
Moreover, if $ \nabla $ is a bicovariant connection on a bicovariant differential calculus $ (\E, d) $ and we write 
 $$ {}_{\E \tensora \E} \Delta (\nabla (e)) = (\nabla (e))_{(0)} \tensorc (\nabla (e))_{(1)},$$ 
	then for all $ \omega \in \zeroE $ and $ a $ in $ \A, $ we have
$$ \xi^{-1} ((\nabla (\omega))_{(0)} a_{(1)} \overline{\gamma} ((\nabla (\omega))_{(1)} \tensorc a_{(2)})) = \xi^{-1}(\nabla_\Omega (\omega)) \staromega a, $$
where $ \nabla_\Omega: \E_\Omega \rightarrow (\E \tensora \E)_\Omega $ is the deformation of the $ \mathbb{C} $-linear bicovariant map $ \nabla: \E \rightarrow \E \tensora \E. $ 
\end{lemma}
\begin{proof}
Let us first clarify that we view $ \omega_{(0)} \tensora (da)_{(0)} \overline{\gamma} (\omega_{(1)} \tensorc (da)_{(1)}) $ and $ (\nabla (\omega))_{(0)} a_{(1)} \overline{\gamma} ((\nabla (\omega))_{(1)} \tensorc a_{(2)}) $ as elements in $ (\E \tensora \E)_\Omega.$
The first equation follows from the following computation:
\begin{equation*}
	\begin{aligned}
		 & \xi^{-1} (\omega_{(0)} \tensora (da)_{(0)} \overline{\gamma} (\omega_{(1)} \tensorc (da)_{(1)})) = \xi^{-1} (\omega_{(0)} \tensora (da)_{(0)})\overline{\gamma} (\omega_{(1)} \tensorc (da)_{(1)}) \\
		=& \overline{\gamma}(1 \tensorc (da)_{(-1)}) \omega_{(0)} \otimes_{\A_\Omega} (da)_{(0)} \gamma(\omega_{(1)} \tensorc (da)_{(1)}) \overline{\gamma}(\omega_{(2)} \tensorc (da)_{(2)}) {\rm (} ~ {\rm as} ~ \omega \in \zeroE ~ {\rm)}\\
		=& \epsilon((da)_{(-1)}) \omega_{(0)} \otimes_{\A_\Omega} (da)_{(0)} \gamma(\omega_{(1)} \tensorc (da)_{(1)}) \overline{\gamma}(\omega_{(2)} \tensorc (da)_{(2)})\\
		=& \omega_{(0)} \otimes_{\A_\Omega} (da)_{(0)} \gamma(\omega_{(1)} \tensorc (da)_{(1)}) \overline{\gamma}(\omega_{(2)} \tensorc (da)_{(2)})\\
		=& \omega_{(0)} \otimes_{\A_\Omega} (da)_{(0)} \epsilon(\omega_{(1)}) \epsilon ((da)_{(1)}) \ {\rm(since \ \overline{\gamma} \ is \ the \ convolution \ inverse \ of \ \gamma)}\\
		=& \omega \otimes_{\A_\Omega} da = \omega \otimes_{\A_\Omega} d_\Omega a.
	\end{aligned}
\end{equation*}
For the proof of the second equation, we use the right $\A_\Omega$-module structure of $(\E \tensora \E)_\Omega$ and the bicovariance of the map $ \nabla_\Omega $ (Proposition \ref{11thjuly20192}). In particular, this implies that if $ \omega \in \zeroE, $ then $ \nabla_\Omega (\omega) $ is an element of $ {}_0 ((\E \otimes \E)_\Omega). $ Hence, by \eqref{9thmay20192}, we get:
\begin{equation*}
	\begin{aligned}
		 & \nabla_{\Omega}(\omega) * a = (\nabla_{\Omega}(\omega))_{(0)} . a_{(1)} \overline{\gamma}((\nabla_{\Omega}(\omega))_{(1)} \tensorc a_{(2)})\\
		=& (\nabla(\omega))_{(0)} . a_{(1)} \overline{\gamma}((\nabla(\omega))_{(1)} \tensorc a_{(2)}),
	\end{aligned}
\end{equation*}
where the equality is of elements in $(\E \tensora \E)_\Omega$. Then, by the right $\A_\Omega$-linearity of $\xi$, we have 
$$ \xi^{-1} ((\nabla (\omega))_{(0)} a_{(1)} \overline{\gamma} ((\nabla (\omega))_{(1)} \tensorc a_{(2)})) = \xi^{-1}(\nabla_\Omega (\omega)) \staromega a ,$$
 where the equality is of elements in $\E_\Omega \otimes_{\A_\Omega} \E_\Omega$. This completes the proof of the lemma.
\end{proof}
By an abuse of notation, we will denote the map $ \xi^{-1} \nabla_\Omega $ by the symbol $ \nabla_\Omega $ again. Thus, $ \nabla_\Omega $ takes value in $ \E_\Omega \otimes_{\A_\Omega} \E_\Omega $ as desired. Then we have the following theorem. 
\begin{thm} \label{18thseptember20191}
Suppose $ (\E, d) $ is a bicovariant differential calculus. Then a bicovariant connection $ \nabla $ deforms to a bicovariant connection $ \nabla_\Omega $ on $ \E_\Omega. $ In fact, bicovariant connections on $ \E $ and $ \E_\Omega $ are in bijective correspondence.
\end{thm}
\begin{proof} For $ \omega \in \zeroE $ and $ a \in \A, $ we have
\begin{equation*} 
	\begin{aligned}
		& \nabla_\Omega(\omega *_\Omega a) = \nabla_\Omega(\omega_{(0)} a_{(1)} \overline{\gamma}(\omega_{(1)} \tensorc a_{(2)})) ~ {\rm (} ~ {\rm by} ~ \eqref{9thmay20192} ~ {\rm)}\\
		=& \nabla_\Omega(\omega_{(0)} a_{(1)}) \overline{\gamma}(\omega_{(1)} \tensorc a_{(2)}) = \nabla(\omega_{(0)} a_{(1)}) \overline{\gamma}(\omega_{(1)} \tensorc a_{(2)}) \\
		= & (\nabla(\omega_{(0)}) a_{(1)} + \omega_{(0)} \tensora d (a_{(1)})) \overline{\gamma}(\omega_{(1)} \tensorc a_{(2)}) \\
		= & \nabla(\omega_{(0)}) a_{(1)} \overline{\gamma}(\omega_{(1)} \tensorc a_{(2)}) + \omega_{(0)} \tensora d (a_{(1)}) \overline{\gamma}(\omega_{(1)} \tensorc a_{(2)}). 
		\end{aligned}
		\end{equation*}
		Now, by the right covariance of the maps $ \nabla $ and $ d $ (see \eqref{23rdaugust20191}), the following equations hold:
		$$ \nabla(\omega_{(0)}) \tensorc \omega_{(1)} = (\nabla(\omega))_{(0)} \tensorc (\nabla(\omega))_{(1)}, ~ d(a_{(1)}) \tensorc a_{(2)} = (da)_{(0)} \tensorc (da)_{(1)},$$
		and therefore, the above expression is equal to
		\begin{equation*}
		 \begin{aligned}
		& (\nabla(\omega))_{(0)} a_{(1)} \overline{\gamma}((\nabla(\omega))_{(1)} \tensorc a_{(2)}) + \omega_{(0)} \tensora (da)_{(0)}\overline{\gamma}(\omega_{(1)} \tensorc (da)_{(1)}) \\
		=& \nabla_\Omega(\omega) *_\Omega a + \omega \otimes_{\A_\Omega} d_\Omega a
		\end{aligned}
\end{equation*}
where we have used the two equations of Lemma \ref{14thseptember20191}. This proves that for all $ \omega $ in $ \zeroE $ and $ a $ in $ \A, $ 
\begin{equation} \label{29thjuly20192} \nabla_\Omega (\omega \staromega a) = \nabla_\Omega (\omega) \staromega a + \omega \otimes_{\A_\Omega} da. \end{equation}
Since $\zeroE = {}_0(\E_\Omega)$ is right $\A_\Omega$-total in $\E_\Omega$, we are left to prove that for all $a,b$ in $\A$ and $\omega$ in $\zeroE$,
 $$ \nabla_{\Omega}((\omega *_\Omega a) *_\Omega b) = \nabla_\Omega(\omega *_\Omega a) *_\Omega b + \omega *_\Omega a \otimes_{\A_\Omega} d_\Omega b. $$ But this follows easily from \eqref{29thjuly20192}.
Since the right and left comodule structure of the calculus and its deformation are the same, hence $\nabla_\Omega$ is also bicovariant.\\
To show that the bicovariant connections of $\E$ and $\E_\Omega$ are in a bijective correspondence, we consider the bicovariant calculus $(\E,d)$ as a cocycle deformation of the calculus $(\E_\Omega, d_\Omega)$ under the cocycle $\Omega^{-1}$. If $\nabla^\prime$ is a bicovariant connection on $(\E_\Omega, d_\Omega)$, then by the above argument, $(\nabla^\prime)_{\Omega^{-1}}$ is a bicovariant connection on $((\E_\Omega)_{\Omega^{-1}}, (d_\Omega)_{\Omega{-1}}) = (\E,d)$. Moreover, $\nabla^\prime = ((\nabla^\prime)_{\Omega^{-1}})_{\Omega}$ and hence is a cocycle deformation of a bicovariant connection on $(\E,d)$ under the cocycle $\Omega$.
\end{proof}

\subsection{The existence of Levi-Civita connections} \label{24thnov20191}
In this subsection, we prove the main result of this section, namely that, if $(\E,d)$ is a bicovariant differential calculus on $\A$ satisfying the conditions of Theorem \ref{15thjune20193} and $g^\prime$ is a pseudo-Riemannian bi-invariant metric on the deformed bimodule $\E_\Omega$, then there exists a unique left-invariant connection which is torsionless and compatible with $g^\prime$. A similar result was proved in Section 7 of \cite{article1} for Connes-Landi deformations of bimodules. We will continue to use the notations $\sigma_\Omega, g_\Omega$ introduced in Theorem \ref{11thmay20192} and $\nabla_\Omega$ from Theorem \ref{18thseptember20191}. In particular, if $ g $ be a pseudo-Riemannian bi-invariant metric on $ \E, $ then $g_\Omega$ is a pseudo-Riemannian bi-invariant metric on $\E_\Omega$ by Theorem \ref{11thmay20192}. 
\begin{thm} \label{18thseptember20193}
	Suppose $ (\E, d) $ is a bicovariant differential calculus on a Hopf algebra $ \A, $ $ \sigma $ be the corresponding braiding map and $ \Omega $ a normalized dual $2$-cocycle on $\A$. If $ {}_0 \sigma $ is diagonalisable and $ g $ is a pseudo-Riemannian bi-invariant metric on $ \E, $ then the following statements hold:
\begin{itemize}	
\item[(i)] If $ \nabla $ is a bicovariant Levi-Civita connection for the triple $ (\E, d, g), $ then $ \nabla $ deforms to a bicovariant Levi-Civita connection $\nabla_\Omega$ for $ (\E_\Omega, d_\Omega, g_\Omega). $
\item[(ii)] In the set up of 1., if we assume that $ \nabla $ is the unique Levi-Civita connection for $ (\E, d, g), $ then $ \nabla_\Omega $ is the unique bi-covariant Levi-Civita connection for $ (\E_\Omega, d_\Omega, g_\Omega). $
\end{itemize}	
\end{thm}
\begin{proof} 
 We start by proving that $\nabla_\Omega$ is torsionless and metric compatible.
Since $\wedge$, $\nabla$ and $d$ are bicovariant, therefore the right $\A$-linear homomorphism $T_\nabla = \wedge \circ \nabla + d$ is also bicovariant. Therefore its cocycle deformation exists and $(T_\nabla)_\Omega = (\wedge \circ \nabla + d)_\Omega = \wedge_\Omega \circ \nabla_\Omega + d_\Omega = T_{\nabla_\Omega}$. Since $\nabla$ is torsionless, we have that $T_{\nabla_\Omega} = 0$.\\
To prove that $\nabla_\Omega$ is compatible with the metric $g$, let us recall that since $\nabla$ is bicovariant and $g$ is bi-invariant, Remark \ref{23rdnov20194} and Proposition \ref{proppignabla} imply that the map $ \widetilde{\Pi_g}(\nabla) $ is bicovariant. The map $ dg: \E \tensora \E \rightarrow \E $ can also be checked to be bicovariant. Therefore, the deformation of the map $\widetilde{\Pi_g}(\nabla) - dg$ exists and is equal to $\widetilde{\Pi_{g_\Omega}}(\nabla_\Omega) - d_\Omega g_\Omega$. Since $\widetilde{\Pi_g}(\nabla) - dg = 0$, therefore we have that $\widetilde{\Pi_{g_\Omega}}(\nabla_\Omega) - d_\Omega g_\Omega = 0$.\\
For the second part of the proof, assume that $\nabla^\prime$ is a bicovariant Levi-Civita connection for the triple $(\E_\Omega, d_\Omega, g_\Omega)$. Viewing $(\E, d, g)$ as a cocycle deformation of $(\E_\Omega, d_\Omega, g_\Omega)$ under the cocycle $\Omega^{-1}$, by the first part of the proof, $(\nabla^\prime)_{\Omega^{-1}}$ is a bicovariant Levi-Civita connection on $(\E_\Omega, d_\Omega, g_\Omega)$. By our hypothesis, such a connection is unique. Hence $(\nabla^\prime)_{\Omega^{-1}} = \nabla$, and hence $\nabla^\prime = \nabla_\Omega$. Thus $(\E_\Omega, d_\Omega, g_\Omega)$ admits a unique bicovariant Levi-Civita connection.
\end{proof}

In Theorem \ref{15thjune20193}, we proved that if the map $ ({}_0(\Psym))_{23} $ is an isomorphism from $ (\zeroE \otimes^{{\rm sym}}_{\mathbb{C}} \zeroE) \tensorc \zeroE $ to $ \zeroE \tensorc (\zeroE \otimes^{{\rm sym}}_{\mathbb{C}} \zeroE), $ then there exists a unique left-covariant Levi-Civita connection for $ (\E, d, g). $ The next theorem shows that under the same assumption, $ (\E_\Omega, d_\Omega) $ admits a unique left-covariant Levi-Civita connection for any bi-invariant pseudo-Riemannian metric.

\begin{thm} \label{18thjuly20194}	
	Suppose $ (\E, d) $ is a bicovariant differential calculus such that ${}_0 \sigma$ is diagonalisable. If the map 
	$$({}_0(\Psym))_{23}: (\zeroE \tensorc^{\rm sym} \zeroE) \tensorc \zeroE \to \zeroE \tensorc (\zeroE \tensorc^{\rm sym} \zeroE) $$ is an isomorphism, then
	\begin{enumerate} 
		\item[(i)] the following map is also an isomorphism: 
		$$({}_0((\Psym)_\Omega))_{23}: ({}_0(\E_\Omega) \tensorc^{\rm sym} {}_0(\E_\Omega)) \tensorc {}_0(\E_\Omega) \to {}_0(\E_\Omega) \tensorc({}_0(\E_\Omega) \tensorc^{\rm sym} {}_0(\E_\Omega)), $$
		\item[(ii)] the corresponding deformed calculus $(\E_\Omega, d_\Omega)$ admits, for every bi-invariant pseudo-Riemannian metric $g^\prime$, a unique left-covariant connection which is torsionless and compatible with $g^\prime$. Moreover, if $\A$ is cosemisimple, this connection is also right-covariant.
	\end{enumerate}
\end{thm}
\begin{proof}
	The first part of the theorem follows by recalling that ${}_0(\E_\Omega) = {}_0 \E$, and the fact that by Proposition \ref{11thjuly20193}, we have ${}_0 (\Psym) = {}_0((\Psym)_\Omega)$.\\
	By Proposition \ref{25thseptember20191}, $ (\Psym)_\Omega$ is the unique idempotent on $ \E_\Omega \otimes_{\A_\Omega} \E_\Omega $ with range $ \E_\Omega \otimes^{{\rm sym}}_{\A_\Omega} \E_\Omega $ and kernel $ \F_\Omega. $ The existence of a unique left-covariant Levi-Civita connection for $ (\E_\Omega, d_\Omega, g^\prime) $ follows by combining the first part and Theorem \ref{15thjune20193}. 
	
	If in addition, if $\A$ is cosemisimple, then $\A_\Omega$ is also cosemisimple and the right-covariance of the Levi-Civita connection follows from Theorem \ref{15thjune20193}.
\end{proof}

As a direct corollary to Theorem \ref{18thjuly20194} and the existence uniqueness theorem for Levi-Civita connection on a classical manifold, we have:	
\begin{prop} \label{18thsep20195}
 Let $\A$ be the Hopf algebra of regular functions on a linear algebraic group $G$ whose category of finite dimensional representations is semisimple. Suppose $ (\E, d) $ is the classical bicovariant differential calculus on $\A$ and $\Omega$ a normalised dual $2$-cocycle on $\A.$ If $g^\prime$ is a pseudo-Riemannian bi-invariant metric on the bicovariant differential calculus $ (\E_\Omega, d_\Omega) $ over the Hopf algebra $ \A_\Omega, $ then there exists a unique bicovariant Levi-Civita connection for the triple $ (\E_\Omega, d_\Omega, g^\prime).$
\end{prop}
\begin{proof}
The map $ g^\prime $ is a bi-invariant pseudo-Riemannian metric on $ \E_\Omega $ and so by Theorem \ref{11thmay20192}, there exists a bi-invariant pseudo-Riemannian metric $g$ on $\E$ such that $ g_\Omega = g^\prime. $ The Levi-Civita connection for the triple $ (\E, d, g) $ is bicovariant. This is well-known and can also be directly seen from Corollary \ref{18thsep20194}. Therefore, we can apply Theorem \ref{18thjuly20194} to reach the desired conclusion. 
\end{proof}

We conclude this section by proving Proposition \ref{25thnov20191} which shows that our definition of metric-compatibility coincides with that in \cite{heckenberger}. The proof will use the notations and discussions preceding the statement of Theorem \ref{25thnov20191}.

{\bf Proof of Proposition \ref{25thnov20191}:} By our assumption, $ \nabla^\prime $ and $g^\prime$ are bicovariant. It can be easily checked that analogues of Theorem \ref{18thseptember20191} for left connections and the third assertion of Theorem \ref{11thmay20192} for left $\A$-linear pseudo-Riemannian metrics hold. This implies that there exist a bicovariant left-connection $\nabla$ on $\E$ and a left $\A$-linear bi-invariant pseudo-Riemannian metric $g$ on $\E$ such that $ \nabla^\prime = \nabla_\Omega $ and $g^\prime = g_\Omega.$ 

Now suppose that $ \nabla^\prime = \nabla_\Omega $ is such that \eqref{25thnov20192} holds for the left $\A$-linear bi-invariant pseudo-Riemannian metric $g^\prime = g_\Omega.$ Then by \eqref{25thnov20193}, $ \widetilde{{}_L \Pi^0_{g_\Omega}} (\nabla_\Omega) = 0,$ i.e, 
$$ 2(g_\Omega \tensorc \id) (\id \tensorc \sigma_\Omega) (\id \tensorc \nabla_\Omega) {}_0 ((\Psym)_\Omega) = 0 $$
as maps on $ {}_0 (\E_\Omega) \tensorc {}_0 (\E_\Omega) = \zeroE \tensorc \zeroE. $ Since the maps $ g_\Omega, \sigma_\Omega, {}_0 (\Psym)_\Omega $ coincide with $g, \sigma, {}_0 (\Psym) $ respectively on $ \zeroE \tensorc \zeroE, $ we can conclude that 
 $$ 2(g \tensorc \id) (\id \tensorc \sigma) (\id \tensorc \nabla) {}_0 (\Psym) = 0 $$
 as maps on $\zeroE \tensorc \zeroE. $ But $\E$ is the classical space of forms on the group $G$ and therefore, our definition of metric-compatibility coincides with that in \cite{heckenberger}. Hence we have
 $$ (\id \tensorc g)(\nabla \tensorc \id) + (g \tensorc \id)(\id \tensorc \sigma)(\id \tensorc \nabla) = 0$$
Applying the same argument as above, we deduce that 
$$ (\id \tensorc g_\Omega)(\nabla_\Omega \tensorc \id) + (g_\Omega \tensorc \id)(\id \tensorc \sigma_\Omega)(\id \tensorc \nabla_\Omega) = 0,$$
i.e, $ \nabla^\prime = \nabla_\Omega $ is compatible with $ g^\prime = g_\Omega $ in the sense of \cite{heckenberger}. 

The converse part follows similarly and this completes the proof. \qed 
\\

We end the article with some comments and limitations of our article. First of all, we need to see whether we still get the existence and uniqueness theorem for Levi-Civita connections for a different construction of two-forms as in Remark \ref{28thnov20191}. Secondly, in this article, we did not take into consideration the $\ast$-compatibility of a connection as defined and studied in \cite{starcompatible}. It remains to be seen whether the Levi-Civita connections obtained in \cite{article3} (or \cite{article1}) or here are compatible with the $\ast$-structure if the algebra under question is equipped with an involution. It would be interesting to have existence and uniqueness type results for covariant differential calculi on quantum homogeneous spaces. For such a theorem on the Podl\'es sphere, we refer to \cite{majidpodles}. Finally, the construction of the curvature and the Ricci tensor also need to be investigated. We hope to answer some of these problems elsewhere.

\end{document}